\documentclass{article}

\usepackage{mathscinet}
\usepackage{array}
\usepackage{booktabs}
\usepackage{numprint}

\usepackage{caption}
\usepackage{subcaption}
\usepackage{combelow}
\usepackage{amsthm,amsfonts,amssymb,amsmath} % AMS packages
\usepackage{aliascnt} % Extended theorem numbering
\usepackage{mathrsfs} % Script font
\usepackage{xargs,ifthen} % Macro construction utilities
\usepackage{enumitem}

\usepackage{bigints}
\usepackage{comment}
\usepackage{color}
\definecolor{lightgrey}{gray}{0.8}

\usepackage{xypic} % Pictures/commutative diagrams
\usepackage{tikz}
\usetikzlibrary{calc,arrows,decorations.markings}

\definecolor{diskpurple}{rgb}{.5,0.,.5}

\tikzstyle inbdry=[thick,blue,postaction={decorate,decoration={markings,
    mark=at position .25 with {\arrow[arrowstyle]{<}}}}]
\tikzstyle outseg=[thick,red,postaction={decorate,decoration={markings,
    mark=at position .5 with {\arrow[arrowstyle]{>}}}}]
\tikzstyle outarc=[thick,diskpurple,postaction={decorate,decoration={markings,
    mark=at position 0.5 with {\arrow[arrowstyle]{<}}}}]
    
\newcommand{\placeinterior}[2]{
	\node (p#1) at (#2) {};
	\draw[inbdry,fill=white] (p#1) circle (\rad);
	\draw[blue,shift=(p#1)] (0,1.5*\rad) node {$\arc{p_#1}$};
}
    
\newcommand{\placeoutarc}[3]{
	\draw[outarc,rotate=#2] (#1) circle (\rad);
	\draw[diskpurple] (#2:0.7*\Rad) node {#3};
}    

\newcommand{\placeoutseg}[4]{
	\draw[outseg] (#1) arc (#2:#3:\Rad);
	\draw[red] ({(#2+#3)/2}:{(1.1+0.15*abs(cos((#2+#3)/2)))*\Rad}) node {#4};
}    

\newcommand{\placeoutpt}[2]{
	\draw[fill,white] (#1) circle (\rad);
}

\tikzstyle dot=[shape=circle,draw,fill,inner sep=1pt]
\tikzstyle vertex=[shape=circle,draw,fill,inner sep=1pt]
\tikzstyle arrowstyle=[scale=1.5]
\tikzstyle edge=[postaction={decorate,decoration={markings,
    mark=at position .55 with {\arrow[arrowstyle]{stealth}}}}]
\newcommand{\drawvertex}[3][]{	\node[vertex] (#2) at (#3) {};
								\node[#1] at (#2) {$#2$}; }
\newcommand{\drawsilentvertex}[3][]{	\node[vertex] (#2) at (#3) {};
								\node[#1] at (#2) {}; }

\newcommand{\drawgraph}[2][1]
{
	\begin{tikzpicture}[scale=#1,baseline={([yshift=0ex]current bounding box.center)}]
		#2
	\end{tikzpicture}
}

\newcommand{\intrograph}[1][1]{
	\drawgraph[#1]{
		\draw[dashed,blue] (-2,0)--(4.5,0);
		\drawsilentvertex{L}{0,0};
		\drawsilentvertex{R}{1,0};
		\drawsilentvertex{1}{.5,.5};
		\drawsilentvertex{2}{1,1};
		\drawsilentvertex{3}{0,1.5};
		\drawsilentvertex{4}{2,2};
		\drawsilentvertex{5}{-.5,2.5};
		\drawsilentvertex{6}{3,3};
		\drawsilentvertex{7}{-1.5,3.5};
		\drawsilentvertex{8}{4,3.75};
		\draw[edge] (1) to (L);
		\draw[edge] (1) to (R);
		\draw[edge] (2) to (1);
		\draw[edge] (2) to (R);
		\draw[edge] (3) to (2);
		\draw[edge] (3) to (L);
		\draw[edge] (4) to (3);
		\draw[edge] (4) to (R);
		\draw[edge] (5) to (4);
		\draw[edge] (5) to (L);
		\draw[edge] (6) to (5);
		\draw[edge] (6) to (R);
		\draw[edge] (7) to (6);
		\draw[edge] (7) to (L);
		\draw[edge] (8) to (7);
		\draw[edge] (8) to (R);
	}
}

\newcommand{\freewedgegraph}[1][1]{
	\drawgraph[#1]{
		\drawvertex[above]{z}{0,1};
		\drawvertex[above]{y}{1,0.5};
		\drawvertex[above]{x}{-1,0};
		\draw[edge] (z) to (x);
		\draw[edge] (z) to (y);
	}
}

\newcommand{\wedgegraph}[1][1]{
	\drawgraph[#1]{
		\draw[dashed,blue] (0.5,0)--(2.5,0);
		\drawvertex[above]{z}{2,1};
		\drawvertex[above]{x}{1,1};
		\drawvertex[below]{q}{1,0};
		\draw[edge] (z) to (q);
		\draw[edge] (z) to (x);
	}
}
\newcommand{\wedgegraphR}[1][1]{
	\drawgraph[#1]{
		\draw[dashed,blue] (0.5,0)--(2.5,0);
		\drawvertex[above]{z}{2,1};
		\drawvertex[above]{x}{1,1};
		\drawvertex[below]{1}{1,0};
		\draw[edge] (z) to (1);
		\draw[edge] (z) to (x);
	}
}

\newcommand{\pqwedgegraph}[1][1]{
	\drawgraph[#1]{
		\draw[dashed,blue] (-0.5,0)--(1.5,0);
		\drawvertex[above]{z}{0.5,1};
		\drawvertex[below]{p}{0,0};
		\drawvertex[below]{q}{1,0};
		\draw[edge] (z) to (p);
		\draw[edge] (z) to (q);
	}
}

\newcommand{\loopgraph}[1][1]{
	\drawgraph[#1]{
		\draw[blue, dashed] (-0.5,0) -- (1.5,0);
		\drawvertex[below]{0}{0,0}
		\drawvertex[above]{y}{0,1}
		\drawvertex[above]{x}{1,1}
		\draw[edge, bend right] (y) to (x);
		\draw[edge] (y) to (0);
		\draw[edge, bend right] (x) to (y);
	}	
}	

\newcommand{\zetagraph}[1][1]{
	\drawgraph[#1]{
		\draw[blue, dashed] (-0.5,0) -- (2.5,0);
		\drawvertex[below]{0}{0,0}
		\drawvertex[below]{1}{1,0}
		\drawvertex[above]{y}{0,1}
		\drawvertex[above]{x}{1,1}
		\drawvertex[above]{z}{2,1}
		\draw[edge] (y) to (0);
		\draw[edge,bend right] (y) to (x);
		\draw[edge,bend right] (x) to (y);
		\draw[edge] (x) to (1);
		\draw[edge] (z) to (x);
		\draw[edge] (z) to (1);
	}
}

\newcommand{\singleedgeL}[1][1]{
	\drawgraph[#1]{
		\draw[blue, dashed] (-0.5,0) -- (0.5,0);
		\drawvertex[below]{0}{0,0}
		\drawvertex[above]{x}{0,1}
		\draw[edge] (x) to (0);
	}
}

\newcommand{\singleedgeR}[1][1]{
	\drawgraph[#1]{
		\draw[blue, dashed] (-0.5,0) -- (0.5,0);
		\drawvertex[below]{1}{0,0}
		\drawvertex[above]{x}{0,1}
		\draw[edge] (x) to (1);
	}
}

\newcommand{\inwheelgraph}[1]{
	\begin{tikzpicture}[baseline=-0.5ex]
		\node[vertex] () at (0,0)  {};
		\foreach \i in {1,...,#1}
		{
			\node[vertex] () at (\i*360/#1:1) {};
			\draw[edge] (\i*360/#1:1) -- (0,0);
			\draw[edge] (\i*360/#1:1) arc (\i*360/#1:{\i*360/#1-360/#1}:1);
		}
	\end{tikzpicture}
}

\newcommand{\outwheelgraph}[1]{
	\begin{tikzpicture}[baseline=-0.5ex]
		\node[vertex] () at (0,0)  {};
		\foreach \i in {1,...,#1}
		{
			\node[vertex] () at (\i*360/#1:1) {};
			\draw[edge] (0,0) --  (\i*360/#1:1);
			\draw[edge] (\i*360/#1:1) arc (\i*360/#1:{\i*360/#1-360/#1}:1);
		}
	\end{tikzpicture}
}

\newcommand{\wheeldiffgraph}[1]{
	\begin{tikzpicture}[scale=0.85,baseline=-0.5ex]
	    \draw[dashed,blue] (-1.5,-1.5) -- (1.5,-1.5);
		\node[vertex] (sink) at (1,-1.5)  {};
		\node[vertex] (extra) at (-1,-1.5)  {};
		\foreach \i in {1,...,#1}
		{
			\node[vertex] () at ({(\i-1)*360/#1+180}:1) {};
			\draw[edge] ({(\i-1)*360/#1+180}:1) arc ({(\i-1)*360/#1+180}:{(\i-1)*360/#1-360/#1+180}:1);
		}
		\foreach \i in {2,...,#1}
		{
			\draw[edge] ({(\i-1)*360/#1+180}:1) -- (sink) ;
		}
		\draw[edge] (-1,0) -- (extra);
	\end{tikzpicture}
}

\newcommand{\HKRgraph}[1]{
	\begin{tikzpicture}
		\node[vertex] (int) at ({#1/2+1/2},1)  {};
		\draw[blue,dashed] (0.5,0) -- ({#1+0.5},0);
		\foreach \i in {1,...,#1}
		{
			\node[vertex] () at (\i,0) {};
			\draw[edge] (int) --  (\i,0);
		}
	\end{tikzpicture}
}

\usepackage{titling}
\usepackage{shuffle}% Symbol for the shuffle product

\usepackage{hyperref}
\definecolor{sxdarkblue}{RGB}{150,180,200} % Surfex images
\definecolor{sxlightblue}{RGB}{186,215,230} % Surfex images
\definecolor{sxred}{RGB}{153,0,0} % Surfex images
\definecolor{tocolor}{rgb}{.1,.1,.1}
\definecolor{urlcolor}{rgb}{.2,.2,.6}
\definecolor{linkcolor}{rgb}{.1,.1,.5}
\definecolor{citecolor}{rgb}{.4,.2,.1}
\definecolor{lightgrey}{rgb}{.9,.9,.9}
\hypersetup{backref=true, colorlinks=true, urlcolor=urlcolor, linkcolor=linkcolor, citecolor=citecolor}

\newcommand{\dch}{\mathsf{dch}}
\newcommand{\PathConc}{\star}
\newcommand{\tp}{\otimes}
\DeclareMathOperator*{\Rlim}{Rlim}
\newcommand{\Reglim}[2]{\Rlim_{#1 \rightarrow #2}}
\newcommand{\dlog}[1]{\omega_{#1}}
\newcommand{\bdlog}[1]{\overline{\omega}_{#1}}

\newcommand{\Tree}[1]{T_{#1}}
\newcommand{\BerG}[1]{\Gamma_{#1}}
\newcommand{\WheelOut}[1]{W^{\uparrow}_{#1}}
\newcommand{\WheelIn}[1]{W^{\downarrow}_{#1}}
\newcommand{\Amar}[2]{\Gamma_{#1,#2}}

\newcommand{\pInf}[1]{#1_\infty}

\newcommandx{\thdef}[2]{
	\newaliascnt{#1}{theorem}  
	\newtheorem{#1}[#1]{#2}
	\aliascntresetthe{#1}  
	\newtheorem*{#1*}{#2}
	\expandafter\newcommand\expandafter{\csname #1autorefname\endcsname}{#2}
}

\newtheorem{theorem}{Theorem}[section]

\thdef{lemma}{Lemma}
\thdef{corollary}{Corollary}
\thdef{conjecture}{Conjecture}
\thdef{proposition}{Proposition}
\newtheorem*{theorem*}{Theorem}

\theoremstyle{definition}
\thdef{definition}{Definition}

\theoremstyle{remark}
\thdef{remark}{Remark}

\theoremstyle{remark}
\thdef{example}{Example}

\newcommand{\defn}[1]{\textbf{\textit{#1}}} % definition

\newcommand{\spc}[1]{\mathsf{#1}} % For spaces
\newcommand{\shf}[1]{\mathcal{#1}} % For sheaves

\newcommand{\RR}{\mathbb{R}}
\newcommand{\CC}{\mathbb{C}}
\newcommand{\ZZ}{\mathbb{Z}}
\newcommand{\cZZ}{\underline{\mathbb{Z}}} % constant sheaf

\newcommand{\QQ}{\mathbb{Q}}
\newcommand{\PP}[1][]{\mathbb{P}^{#1}}

\newcommand{\Pp}{P\cup\sset{z}}
\newcommand{\Qq}{Q\cup\sset{q}}
\newcommand{\Spp}{S\cup\sset{z,\zb}}
\newcommand{\Sp}{S\cup\sset{z}}
\newcommand{\Sq}{S\cup\sset{q}}
\newcommand{\Qinf}{Q\cup\sset{\infty}}

\newcommand{\sset}[1]{\left\{#1\right\}}
\newcommand{\set}[2]{\left\{#1\,\middle |\ #2\right\}}
\newcommand{\rbrac}[1]{\left(#1\right)} % Round brackets
\newcommand{\abrac}[1]{\left\langle#1\right\rangle} % Angle brackets
\newcommand{\snorm}[1]{\left|#1\right|} % Single bar norm
\newcommand{\abs}[1]{\snorm{#1}}

\newcommand{\asyO}[1]{O\left( #1 \right)}

\newcommand{\mapdef}[5]{
	\begin{array}{ccccc}
	#1 &:& #2 &\to& #3 \\
		&&  #4 &\mapsto& #5
	\end{array}
}
\newcommandx{\fn}[2][2=]{#1\ifthenelse{\equal{#2}{}}{}{\!\rbrac{{#2}}}} % Functions with optional arguments

\newcommand{\colim}[1]{%
  \mathbin{\mathop{\mathrm{colim}}\displaylimits_{#1}\,}%
}

\newcommandx{\Aut}[2][1=]{\fn{\spc{Aut}_{#1}}[#2]} % automorphisms

\newcommand{\gr}[1][]{\mathsf{gr}^{#1}}

\newcommand{\Perms}[1]{\mathfrak{S}_{#1}}

\newcommandx{\hlgy}[3][1=\bullet,3=]{\spc{H}_{#1}^{#3}\!\rbrac{{#2}}} % Homology
\newcommandx{\cohlgy}[3][1=\bullet,3=]{\spc{H}^{#1}_{#3}\!\rbrac{{#2}}} % Cohomology
\newcommandx{\scohlgy}[3][1=\bullet,3=]{\mathcal{H}^{#1}_{#3}\!\rbrac{{#2}}} % Cohomology sheaf

\newcommand{\sO}[1]{\shf{O}_{#1}}

\newcommandx{\tb}[2][1=]{\spc{T}_{#1}{#2}} % Tangent bundle
\newcommand{\cvf}[1]{\partial_{{#1}}} 

\newcommand{\injection}{\hookrightarrow}
\newcommand{\defas}{\mathrel{\mathop:}=}
\newcommand{\iu}{\mathrm{i}}
\newcommand{\ipi}{\iu \pi}
\newcommand{\invtate}{\frac{1}{2 \iu \pi}}
\newcommand{\td}[1][]{\mathrm{d}^{#1}}
\DeclareMathOperator{\tdlog}{dlog}

\DeclareMathOperator*{\Res}{Res}

\newcommand{\MZV}[1][]{\mathcal{Z}^{#1}}
\newcommand{\MZVipi}[1][]{\widetilde{\mathcal{Z}}^{#1}}

\newcommand{\mzv}[2][]{\zeta^{#1}(#2)}

\newcommand{\nauty}{\href{http://pallini.di.uniroma1.it/}{\texttt{\textup{nauty}}}}
\newcommand{\HyperInt}{\href{http://bitbucket.org/PanzerErik/hyperint/}{\texttt{\textup{HyperInt}}}}
\newcommand{\DQcode}{\url{http://bitbucket.org/bpym/starproducts/}}
\newcommand{\Kontsevint}{\url{http://bitbucket.org/PanzerErik/kontsevInt}}
\newcommand{\Maple}{\texttt{Maple}}
\newcommand{\Sage}{\texttt{SageMath}}

\newcommand{\crossrat}[4]{\left(#1,#2;#3,#4\right)}

\newcommand{\sCinf}[1]{C^\infty_{#1}}

\newcommand{\KW}[1]{c_{#1}}

\newcommand{\sv}{\mathrm{sv}}

\newcommand{\T}[2][]{\ifthenelse{\equal{#1}{}}{\ZZ\abrac{#2}}{\ZZ{#2}^{#1}}}

\newcommandx{\PerT}[3][1=,3=]{
\ifthenelse{\equal{#3}{}}{\ifthenelse{\equal{#1}{}}{}{W_{#1}}\mathscr{P}\langle#2\rangle^{#3}}{\mathscr{P}{#2}^{#3}}
}

\newcommand{\Per}[2][]{	\ifthenelse{\equal{#1}{}}{}{W_{#1}} \mathscr{P}\ifthenelse{\equal{#2}{}}{}{( #2 )} }

\newcommandx{\HL}[3][1=,3=]{ \ifthenelse{\equal{#1}{} \AND \equal{#3}{}}{}{W_{#1}^{#3}} \mathcal{L}_{#2}} % Hyperlogs

\newcommandx{\V}[2][2=]{\ifthenelse{\equal{#2}{}}{}{W_{#2}}\mathcal{V}(#1)}

\newcommandx{\VA}[3][1=\bullet,3=]{ \ifthenelse{\equal{#3}{}}{}{W_{#3}}\mathcal{U}^{#1}(#2)}
\newcommandx{\Vr}[3][1=,3=]{\ifthenelse{\equal{#1}{} \AND \equal{#3}{}}{}{W_{#1}^{#3}} \mathcal{V}_{{#2}}}
\newcommandx{\VAr}[3][1=\bullet,3=]{	\ifthenelse{\equal{#3}{}}{}{W_{#3}} \mathcal{U}^{#1}_{#2}}

\newcommandx{\A}[3][1=\bullet,2=]{\mathcal{A}^{#1}_{#2}(#3)}

\newcommand{\chn}[2][\bullet]{\mathsf{C}^{#1}\ifthenelse{\equal{#2}{}}{}{(#2)}}

\newcommand{\xb}{\overline{x}}
\newcommand{\yb}{\overline{y}}
\newcommand{\zb}{\overline{z}}
\newcommand{\pb}{\overline{p}}

\newcommand{\pt}{\vec{p}}
\newcommand{\qt}{\vec{q}}

\newcommand{\st}{\vec{s}}
\newcommand{\ttt}{\vec{t}}
\newcommand{\tinf}{\vec{\infty}}

\newcommand{\Pb}{\overline{P}}

\newcommand{\talpha}{\widetilde{\alpha}}

\newcommand{\U}{\mathsf{U}}

\newcommand{\X}{\mathsf{X}}

\newcommand{\Xo}{\mathsf{X}^\circ}
\newcommand{\cyc}{\mathsf{\Sigma}}
\newcommand{\W}{\mathsf{W}}

\newcommand{\D}{\mathsf{D}}
\newcommand{\bD}{\overline{\mathsf{D}}}
\newcommand{\Do}{\mathsf{D}^\circ}
\newcommand{\Spec}{\mathsf{Spec}}

\newcommand{\sect}[1]{\mathsf{\Gamma}\rbrac{#1}} % Global sections of a sheaf

\newcommand{\HH}{\mathbb{H}}
\newcommand{\bHH}{\overline{\HH}}

\newcommand{\hlog}[1]{L_{#1}}
\newcommand{\bhlog}[1]{\overline{L}_{#1}}

\newcommand{\forms}[2][\bullet]{\Omega^{#1}(#2)}
\newcommand{\sforms}[2][\bullet]{\Omega^{#1}_{#2}}

\newcommand{\M}[1]{\mathfrak{M}_{#1}}
\newcommand{\uX}[1]{\mathfrak{X}_{#1}} % universal curve
\newcommand{\uXo}[1]{\mathfrak{X}^\circ_{#1}} % universal punctured curve
\newcommand{\uD}[2][]{\mathfrak{D}_{#2}^{#1}} % universal disk
\newcommand{\uDo}[1]{\mathfrak{D}_{#1}^{\circ}} % universal disk
\newcommand{\arc}[1]{\sigma_{#1}^\epsilon}
\newcommand{\intvl}[1]{(#1,#1')^\epsilon}
\newcommandx{\buD}[3][2=\epsilon,1=]{\partial_{#1}\mathfrak{D}^{#2}_{#3}} % universal disk
\newcommandx{\outD}[1]{\partial_{\mathrm{out}}\mathfrak{D}^{\epsilon}_{#1}} % outer boundary of universal disk

\newcommand{\bM}[1]{\overline{\mathfrak{M}}_{#1}}
\newcommand{\bbM}[2][]{\partial_{#1}\overline{\mathfrak{M}}_{#2}}
\newcommand{\Conf}{\mathsf{Conf}}
\newcommand{\C}[1]{\mathfrak{C}_{#1}}

\usepackage{thmtools}

\newcommand{\fg}[1]{\pi_1(#1)}
\newcommand{\FGpd}[1]{\Pi_1(#1)}

\newcommand{\tmt}{\tau} % t(1-t)
\renewcommand{\emptyset}{\varnothing}

\begin{document}
\vspace{-15cm}
\title{Multiple zeta values in deformation quantization}
\date{}
\author{
Peter Banks
	\thanks{University of Oxford, \url{peter.banks@maths.ox.ac.uk}}
\and Erik Panzer
	\thanks{University of Oxford,  \url{erik.panzer@all-souls.ox.ac.uk}}
\and Brent Pym
	\thanks{University of Edinburgh, \url{brent.pym@mcgill.ca}}}
\maketitle
\vspace{-1.5cm}
\begin{abstract}
Kontsevich's 1997 formula for the deformation quantization of Poisson brackets is a Feynman expansion involving volume integrals over moduli spaces of marked disks.  We develop a systematic theory of integration on these moduli spaces via suitable algebras of polylogarithms, and use it to prove that Kontsevich's integrals can be expressed as integer-linear combinations of multiple zeta values.  Our proof gives a concrete algorithm for calculating the integrals, which we have used to produce the first software package for the symbolic calculation of Kontsevich's formula.
\end{abstract}

\setcounter{tocdepth}{2}
\begin{small}
\tableofcontents
\end{small}

%%%%%%%%%%%%%%%%%%%%%%

\section{Introduction}
\subsection{Motivation and overview}

In 1997, Kontsevich solved a long-standing problem in mathematical physics, by showing that every Poisson manifold can be quantized to obtain a noncommutative algebra~\cite{Kontsevich2003}.  He gave an explicit formula that takes a Poisson bracket on $\RR^k$ as input and produces a ``star product'', i.e.~a noncommutative deformation of the product on $C^{\infty}(\RR^k)$ as a formal power series in the deformation parameter $\hbar$.  He did so by constructing his ``formality morphism''---an explicit homotopy equivalence between the differential graded Lie algebras of polyvector fields and Hochschild cochains on the affine space $\RR^k$.

The star product (and more generally, the formality morphism) is expressed as a sum over a suitable collection of  graphs.  Each graph $\Gamma$ contributes a term involving a differential operator defined by an elementary combinatorial procedure.  This operator is then weighted by a constant
\begin{equation}
\KW{\Gamma} \defas \int_{\C{n,m}}\omega_{\Gamma} \in \RR \label{eq:KW}
\end{equation}
defined by integrating an explicit volume form $\omega_{\Gamma}$ over a suitable moduli space $\C{n,m}$ of marked holomorphic disks.  As explained by Cattaneo--Felder~\cite{Cattaneo2000}, Kontsevich's formula can be interpreted as a perturbative expansion in an appropriate topological string theory, which has the expressions~\eqref{eq:KW} as its Feynman integrals.

The coefficients \eqref{eq:KW} are universal, in the sense that they are independent of the Poisson bracket one seeks to quantize; hence one only needs to compute each integral once and remember its value in perpetuity.  However, the integrals are notoriously difficult to evaluate, and their precise values remain unknown.  Thus, even with the aid of a computer, it has so far been impossible to calculate the terms in the quantization formula explicitly beyond $\hbar^3$ for many of the most basic examples of Poisson brackets, such as the simple ``log canonical'' bracket
\begin{align}
\{x,y\} = xy \label{eq:log-canonical}
\end{align}
on the $(x,y)$-plane (so named because the logarithms $u = \log x$ and $v = \log y$ are canonical variables, i.e.~$\{u,v\} = 1$).

Nevertheless, there are strong expectations about the arithmetic properties of the coefficients.  In his 1999 paper \cite{Kontsevich1999}, Kontsevich made a series of inspiring conjectures, based on connections with the theories of motives, and informed by Tamarkin's independent construction of quantizations~\cite{Tamarkin1999} via formality of the little disks operad.  One of the important consequences of those conjectures is that the integrals \eqref{eq:KW} should be expressible in terms of a much smaller collection of constants, which are better understood (but still quite mysterious):
\begin{conjecture}\label{conj:kontsevich}
The coefficients $\KW{\Gamma}$ can be expressed as $\QQ[(2 \pi \iu)^{- 1}]$-linear combinations of multiple zeta values.
\end{conjecture}
We recall that a \defn{multiple zeta value (MZV)} of weight $n$ is a real number defined by a convergent sum of the form
\begin{equation}
	\mzv{n_1,\ldots,n_d} 
	= \sum_{0 < k_1 < \cdots < k_d} \frac{1}{k_1^{n_1} \cdots k_d^{n_d}}
	\in \RR
	,
	\label{eq:mzv}%
\end{equation}
where $n_1,\ldots,n_d$ are positive integers such that $\sum_j n_j =n$ and $n_d \ge 2$.  Thus MZVs generalize the special values $\zeta(n)$ of the Riemann zeta function.

In this paper, we develop a systematic theory of integration on the moduli spaces of marked disks via suitable algebras of polylogarithms, and prove \autoref{conj:kontsevich} as an application.  In fact, we prove a  stronger statement, which gives precise control over the weights of the MZVs that appear at a given order in $\hbar$, and which bounds the denominators of the rational coefficients.  In other words, after accounting for natural normalization factors, we show that the coefficients of the linear combinations are actually integers.  See \autoref{thm:formality-MZV} below for the precise statement.

Our approach applies not only to Kontsevich's original ``harmonic'' formality morphism, but also to the other known explicit formality morphisms, namely the ``logarithmic'' formality morphism of Kontsevich/Alekseev--Rossi--Torossian--Willwacher~\cite{Alekseev2016,Kontsevich1999} and the ``interpolating'' formality morphism of Rossi--Willwacher~\cite{Rossi2014}.  We treat them uniformly in this paper.  Our approach also applies with minimal modification to the integrals appearing in various extensions of the formality morphism~(e.g.~\cite{Shoikhet2003,Willwacher2016}), to the calculation of the coefficients of the Alekseev--Torossian connection/associator~\cite{Alekseev2010a,Furusho:AT}, and presumably also to the calculation of correlation functions in other two-dimensional field theories, although  we have not made an effort to explicitly pursue these further applications here.

Perhaps most importantly, our results are constructive. The proof gives an effective algorithm to calculate the integrals in terms of MZVs, which we have implemented to obtain the first software for the symbolic calculation of the terms in the quantization formula.  For instance, our software gives the following terms in the ``harmonic'' star product associated to the Poisson bracket \eqref{eq:log-canonical}, obtained by summing contributions from thousands of individual graphs:
\begin{equation*}
	x \star y = g( \hbar) xy \qquad\text{and}\qquad
	y \star x = g(-\hbar) xy,
\end{equation*}
where
\begin{align*}
g(\hbar) &= 1 + \frac{\hbar}{2}  + \frac{\hbar^2}{24} - \frac{\hbar^3}{48}  - \frac{\hbar^4}{1440}  + \frac{\hbar^5}{480} + \rbrac{\frac{251  \zeta(3)^{2}}{2048 \pi^{6}} -  \frac{17}{184320}  }\hbar^6 + \cdots 
\end{align*}
Despite the presence of the conjecturally irrational number $\mzv{3}^2/\pi^6$, this expression is consistent with the expected relation \[
x\star y = e^\hbar\, y \star x,
\]
obtained by formally imposing the canonical commutation rule $[u,v]=\hbar$ on the logarithms $u=\log x$ and $v=\log y$. A similar result holds for the other formality morphisms mentioned above, although the exact expressions for $g(\hbar)$ are different.

Our software is open source and is freely available at the following URL:
\begin{quote}
\DQcode
\end{quote} 
It includes a complete database of coefficients for the $50821$ graphs that appear up to $\hbar^6$ in the star product.  The software can easily compute individual coefficients at higher order in $\hbar$, but since the number of graphs grows factorially, the computation of the full star product becomes quite resource-intensive; already at $\hbar^7$ there are more than a million graphs to consider.

We refer the reader to \autoref{sec:software} for more information about this package, including a summary of numerous previously known properties of the coefficients $\KW{\Gamma}$ that we used to test our software implementation, and by extension our main theorems.    For the rest of the introduction, and indeed most of the paper, we focus on the mathematics.

\subsection{Multiple zeta values}

To give a precise statement of our results, we must first set some notation and recall some standard facts about MZVs; see e.g.~the surveys~\cite{Waldschmidt:LecturesMZV,Zudilin:AlgebraicRelationsMZV}.  Let $\MZV[0] = \ZZ$ be the ring of integers, let $\MZV[1] = \{0\}$, and for $n\ge 2$, let $\MZV[n] \subset \RR$ be the $\ZZ$-submodule generated by the MZVs of weight $n$, as defined in \eqref{eq:mzv}.  We shall work throughout with the $\ZZ$-submodule
\begin{equation}
\MZVipi[n] \defas \frac{\MZV[n]+\ipi\MZV[n-1]}{(2\ipi)^n} \subset \CC. \label{eq:normalized-mzv}
\end{equation}
of \defn{normalized MZVs of weight $n$}.  We write $\Re\MZVipi[n]$ and $\Im\MZVipi[n]$ for the real and imaginary parts of $\MZVipi[n]$, so that
\[
\MZVipi[n] = \Re\MZVipi[n] \oplus \iu \Im \MZVipi[n]
\]
as $\ZZ$-submodules of $\CC$.  Note the parity in \eqref{eq:normalized-mzv}: only even weight MZVs appear in $\Re\MZVipi$, while only odd weight MZVs appear in $\Im\MZVipi$.

Thanks to the shuffle product on MZVs and Euler's identity $(i\pi)^2 = -6\mzv{2}$, we have $\MZVipi[m]\MZVipi[n]\subset\MZVipi[m+n]$.  Meanwhile $1 \in \MZVipi[1]=\tfrac{1}{2} \ZZ$ and hence $\MZVipi[n]\subset\MZVipi[n+1]$, so that the union of these modules gives a filtered subring
\[
	\MZVipi \defas \bigcup_{n \ge 0} \MZVipi[n] \subset \CC.
\]

There is an established upper bound on the rank of the $\ZZ$-modules $\MZV[n]$, which is conjectured to be tight~\cite{Terasoma:MixedTateMotivesAndMZV}. Using proven reduction formulae for MZVs from \cite{Datamine}, one obtains a list of $\ZZ$-module generators for $\MZVipi[n]$ for small $n$; see \autoref{tab:period-gens} below.  For example, $\MZVipi[2]$ is generated by the number $-\mzv{2}/(2\ipi)^2 = 1/24$.  We remark that these generators are conjectured, but not known, to be $\QQ$-linearly independent (unless $n\leq 4$).  In particular, the known irrationality results for odd Riemann zeta values (e.g.\ \cite{Apery:3,BallRivoal:Infinite}) do not address relations with powers of $\pi$, and to date the possibility $\MZVipi \subseteq \QQ[\iu]$ has not been excluded.  However, such transcendence issues will not play any role in this paper.

\begin{table}[h]
\caption{Generators for the $\ZZ$-modules $\MZVipi[n] = \Re\MZVipi \oplus \iu \Im\MZVipi \subset \CC$ for $n \le 6$.}%
\begin{tabular}{rccccccc}
\toprule
$n$ & 0 & 1 & 2 & 3 & 4 & 5 & 6 \\
\midrule
generators for $\Re \MZVipi$ &
	$1$ &
	$\frac{1}{2}$ &
	$\frac{1}{24}$ &
	$\frac{1}{48}$ &
	$\frac{1}{5760}$ &
	$\frac{1}{11520}$ & 
	$\frac{1}{2903040}, \frac{\mzv{3}^2}{128\pi^6}$ \\[2mm]
generators for $\iu\Im\MZVipi$ &
	&
	&
	&
	$\frac{\iu\mzv{3}}{8\pi^3}$ & 
	$\frac{\iu\mzv{3}}{16\pi^3}$ &
	$\frac{\iu \mzv{3}}{192\pi^3},\frac{\iu \mzv{5}}{64\pi^5}$ &
	$\frac{\iu \mzv{3}}{384\pi^3},\frac{\iu \mzv{5}}{128\pi^5}$ \\
\bottomrule
\end{tabular}%
\label{tab:period-gens}%
\end{table}

Note that the denominators in \autoref{tab:period-gens} grow rapidly with $n$.  Indeed, using known identities for MZVs, we show in \autoref{sec:arithmetic-appendix} that $\tfrac{1}{(n+1)!} \in \MZVipi[n]$.  This arithmetic fact allows us to constrain the weight of certain period integrals (\autoref{lem:loop-weight}), which in turn plays a key role in establishing the main result.

\subsection{Statement of the main theorem}
For $n,m \in \ZZ_{\ge 0}$ with $2n+m \ge 2$, let $\C{n,m}$ denote the moduli space parametrizing isomorphism classes of compact complex disks with $n$ marked points in the interior and $m+1$ marked points on the boundary.  It is a real manifold of dimension $2n+m$ that is a classifying space for the pure braid group on $n$ strands; see \autoref{sec:cycles} for details.  Any cross ratio of marked points and their complex conjugates gives a function $f \colon \C{n,m}\to\CC$, and we consider the subring of differential forms
\[
	\A{\C{n,m}} \defas \ZZ\abrac{ \frac{\td f}{2\ipi f} \middle| f \textrm{ is a cross ratio} } \subset \forms{\C{n,m}}
\]
generated by the logarithmic differentials of cross ratios. Our main result is
\begin{theorem}\label{thm:disk-period}
Suppose $\omega \in \A{\C{n,m}}$ is a volume form such that the integral
\[
I \defas \int_{\C{n,m}}\omega
\]
converges absolutely.  Then we have
\[
I \in \begin{cases}
\MZVipi[n-1] & \text{if $m = 0$ and} \\
\MZVipi[n+m-2] & \text{for $m > 0$.}
\end{cases}
\]
\end{theorem}

Up to some conventional coefficients, the integrals appearing in the harmonic, logarithmic, and interpolating formality morphisms are all of this type; see \autoref{sec:formality}.  In particular, the integrals appearing at order $\hbar^n$ in the star product are integrals over $\C{n,2}$, which immediately implies the following:

\begin{corollary}\label{thm:formality-MZV}
Suppose that $\Gamma$ is a graph appearing at order $\hbar^n$ in the star product, for some $n\ge 1$.  Then its coefficient $c_{\Gamma}$ with respect to the known explicit formality morphisms has the following form:
\begin{itemize}
\item For the harmonic formality morphism,
\[
\KW{\Gamma} \in 2^{-2n}\Re\MZVipi \subset \RR.
\]
\item For the logarithmic formality morphism,
\[
\KW{\Gamma} \in \MZVipi[n] \subset \CC.
\]
\item For the interpolating formality morphism (depending on a parameter $t$),
\[
\KW{\Gamma} \in \MZVipi[][t] \subset \CC[t]
\]
is a polynomial in $t$ of degree at most $2n-2$ whose coefficients are normalized MZVs of weight at most $n$.
\end{itemize}
\end{corollary}
The interpolating formality morphism reduces to the harmonic and logarithmic ones when $t=\tfrac{1}{2}$ and $t=0$, respectively.  Hence the first two statements in \autoref{thm:formality-MZV} follow from the third.

In particular, for the original harmonic formality morphism of Kontsevich, we see that only even weight MZVs can occur, and moreover that the coefficients appearing at order $\hbar^n$ for $n\le 5$ are all rational.  However, as we saw in the example \eqref{eq:log-canonical} above, one can find conjecturally irrational coefficients in the star product at order $\hbar^6$ and beyond, even after summing up the contributions from all of the individual graphs.

\begin{figure}[t]
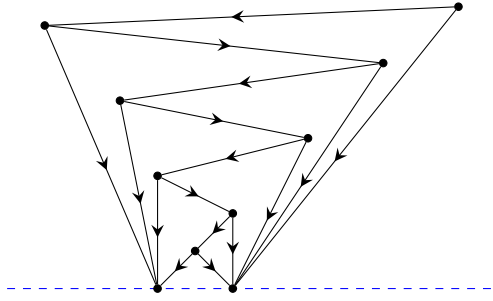

\centering \intrograph%
\caption{A graph appearing at order $\hbar^8$ in the star product.}\label{fig:h8-graph}%
\end{figure}

Let us remark that it is not clear from our theorem which MZVs will appear in the coefficient of a given graph.  For example, consider the graph shown in \autoref{fig:h8-graph}. Our software gives the following value of its coefficient $\KW{\Gamma} \in {\MZVipi}[t]$ in the interpolating formality morphism:
\begin{align*}
\KW{\Gamma}
&=\frac{5040+25059\tmt-23887\tmt^2}{870912000} 
+ \tmt(51+316 \tmt) \frac{\mzv{3}\mzv{5}}{512 {\pi}^{8}}
+ 3\tmt(81+232\tmt) \frac{\mzv{3,5}}{2560{\pi}^{8}} \\[2pt]
&+\iu(1-2t)\tmt\Big(10240(1+ 168 \tmt)\frac{\mzv{3}}{\pi^3}
 - \frac{( 59 +687 \tmt )\mzv{5}}{3072\pi^5}
 +3(247+495 \tmt)\frac{\mzv{7}}{\pi^7} \Big)
\end{align*}
where $\tmt=t(1-t)$.  Note that no rational multiples of $\frac{\mzv{3}^2}{\pi^6}$ appear, although they are allowed by our theorem and conjecturally independent from the other MZVs in this expression.  When $t \in \QQ$ is generic (including the harmonic case $t=\tfrac{1}{2}$) the expression above is conjecturally irrational, but for the logarithmic formality morphism ($t=0$), all of the MZVs drop out and we are left with a rational number:
\begin{align*}
\KW{\Gamma}|_{t=0} =  \frac{1}{172800}.
\end{align*}
On the other hand, there are graphs where the opposite behaviour occurs, e.g.~the harmonic formality coefficient is zero while the logarithmic one involves nontrivial MZVs.  Indeed, rationality is far from the expected generic behaviour for the coefficients of the logarithmic formality morphism.  In fact, we conjecture that \autoref{thm:formality-MZV} is sharp in that case:
\begin{conjecture}
The integrals appearing at order $\hbar^n$ in the logarithmic star product generate $\MZVipi[n]$ as a $\ZZ$-module.
\end{conjecture}
We have verified this conjecture using our software for $n \le 9$. As further evidence for the conjecture, let us recall that  there is a canonical equivalence between the set of homotopy classes of formality morphisms and the set of Drinfeld associators~\cite{Dolgushev2011,Willwacher2014,Willwacher2015}, as torsors over the Grothendieck--Teichm\"{u}ller group.  Under this equivalence, the logarithmic formality morphism corresponds to the Knizhnik--Zamolodchikov associator, which is the generating series for MZVs.   Note, however, that for a given formality morphism, knowledge of the coefficients of the corresponding associator does not imply any arithmetic information about the formality morphism itself or its associated star product---only the homotopy class. 

\subsection{Outline of the proof and layout of the paper}

The proof of \autoref{thm:disk-period} rests on the development of a suitable function theory on the moduli space of marked disks that facilitates the integration of differential forms.  To develop this theory we relate the problem to a more familiar setting in the theory of period integrals: polylogarithms on the moduli space of genus zero curves, as studied by Brown~\cite{Brown2009} and Goncharov~\cite{Goncharov2004}.

The basic observation is that a marked disk $\D$ can be ``doubled'' by gluing it to its complex conjugate along the boundary, yielding a genus zero curve $\X = \D \cup_{\partial\D} \bD \cong \PP[1]$.  In the process, each marked point in $\partial\D$ maps to a unique marked point in $\X$ while the interior marked points in $\D$ produce pairs of marked points in $\X$ that are interchanged by complex conjugation.  In this way we obtain an embedding
\[
\C{n,m} \hookrightarrow \M{N}
\]
where $\M{N} = \M{0,N+1}(\CC)$ denotes the moduli space of smooth complex curves of genus zero with $N +1$ marked points, where $N = 2n+m$.   The forms on $\C{n,m}$ we wish to integrate are then given by the  restriction of holomorphic forms on $\M{N}$.  We can view $\C{n,m}$ as a connected component of the fixed point set of an antiholomorphic involution on $\M{N}$, as studied by Ceyhan~\cite{Ceyhan2007a,Ceyhan2007}, i.e.~$\C{n,m}$ is a component of the real locus for a real structure on $\M{N}$.  For $n > 0$, this real structure is different from the standard one obtained by viewing $\M{N}$ as a scheme over $\Spec(\ZZ)$.

In the case $n=0$, our theorem reduces to the study of integrals over a cell $\C{0,m} \subset \M{N}$ in the standard real locus, whose relation with MZVs was treated in the seminal work~\cite{Brown2009} of Brown (with $\QQ$-coefficients); see also \cite{Bogner2015,Brown2013,Brown:ModuliSpacesFeynmanIntegrals}.  Brown constructed a natural sheaf of differential forms on $\M{N}$ whose coefficient functions are certain multivalued transcendental functions---the ``multiple polylogarithms''.  The closure of $\C{0,m}$ in the Deligne--Knudsen--Mumford (DKM) compactification $\bM{N}$ is a manifold with corners (a copy of the Stasheff associahedron), and Brown describes an inductive algorithm in which Stokes' theorem is applied in the algebra of polylogarithmic forms to reduce the problem to integration over boundary strata of successively smaller dimension, eventually reducing the integral to special values of polylogarithms at points, which give MZVs.  A key achievement in Brown's work is to explain how certain natural coordinates on the standard real locus of the moduli space can be used to regulate the divergences of polylogarithms along the boundary strata, so that Stokes' theorem can be applied in spite of the apparent singularities.

Our proof of \autoref{thm:disk-period} uses Brown's work in an essential way, but we also introduce several new techniques in order to handle various complications in the geometry that have a nontrivial effect on the arithmetic properties of the integrals.  In particular, our result shows that when $n > 0$, the MZVs that appear have a weight that is always strictly less than the dimension of the integration cycle.  Establishing this weight drop is one of the most subtle aspects of our argument; it involves a careful analysis of the asymptotic behaviour of polylogarithms on our integration cycles. We hope to give a motivic explanation for this weight drop in future work.

The first difficulty we encounter in the proof of \autoref{thm:disk-period} is that the natural dihedral coordinates used in Brown's logarithmic regularization scheme are ill-adapted to our integration cycles. To overcome this, we describe an alternative construction of the sheaf of polylogarithms in \autoref{sec:polylogs}, which emphasizes their interpretation as periods of the fundamental groupoid of the universal curve (the numerical shadow of Goncharov's motivic polylogarithms~\cite{Goncharov2004}).  This allows us to give an exposition that reduces the explicit dependence on coordinates and base points and is therefore more readily applied to our cycles.  The construction also allows us to work systematically with coefficients in $\ZZ$, giving an integral lattice in Brown's $\QQ$-linear sheaf of polylogarithms.

The second challenge is that when $n > 1$, the manifold $\C{n,m}$ is not simply connected, and hence polylogarithms on $\C{n,m}$ may have nontrivial monodromy. This causes problems when we try to apply Stokes' theorem, which can only be applied to differential forms that are globally defined (i.e.~single-valued).  To overcome this, we establish in \autoref{sec:monodromy} some cohomological properties of the sheaf of polylogarithms.  For any cycle $\cyc \subset \M{N}$ that is homotopy equivalent to an iterated fibration by punctured disks in universal curves, we give a constructive algebraic model for the polylogarithms that are single-valued on $\cyc \subset \M{N}$, and we establish the vanishing of the higher sheaf cohomology along the fibres; see \autoref{thm:families-push}.

In \autoref{sec:cycles}, we apply these developments to the case $\cyc = \C{n,m}$.  Our main technical result (\autoref{thm:disk-moduli-pushforward}) describes the functorial behaviour of polylogarithms along the natural projections
\[
f \colon \C{n',m'} \to \C{n,m}
\]
that forget a collection of marked points, which we summarize as follows. 

The differential forms with  polylogarithmic coefficients give a complex of locally constant subsheaves $(\VAr{n,m},\td) \subset (\sforms{\C{n,m}},\td)$ in the de Rham complex.
This complex comes equipped with a canonical increasing weight filtration $\VAr{n,m}[\bullet]$. The constants in $\VAr[0]{n,m}$ are exactly the normalized MZVs and a version of the Poincar\'e lemma states that the inclusion of the constant sheaf $\underline{\MZVipi} \hookrightarrow \VAr{n,m}$ is a quasi-isomorphism of filtered sheaves.   Using the results of \autoref{sec:monodromy}, we show that these sheaves are acyclic along the fibres of any projection $f$, and hence by taking hypercohomology of $(\VAr{n,m},\td)$ we obtain an arithmetic refinement of the classical de Rham isomorphism. 

Since the top-degree cohomology of the fibres vanishes, this refined de Rham isomorphism implies, in particular, that fibrewise polylogarithmic volume forms always have polylogarithmic primitives that are single-valued along the fibres. Hence for a global polylogarithmic form 
\[
\omega \in \VA{\C{n',m'}} \defas \sect{\C{n',m'},\VAr{n',m'}}
\]
we may apply Stokes' theorem to compute the pushforward $f_*\omega \in \forms{\C{n,m}}$, defined by integrating $\omega$ over the fibres (provided that these integrals converge).  However, the primitive may have singularities on the boundary that require a more general regularization scheme than the ones developed in \cite{Alekseev2016,Brown2009}.  This could in principle be done globally on the fibres for any projection $f$, but since the pushforward is functorial, it is actually sufficient to treat the simplest cases, in which a single marked point is forgotten.  A careful analysis of the asymptotics in this case allows us to calculate the integrals using contour deformation and a suitable residue theorem, giving the following:

\begin{theorem*}(see \autoref{thm:disk-moduli-pushforward})
If $\omega \in \VA{\C{n',m'}}[j]$ is such that the fibre integrals defining the pushforward $f_*\omega$ converge absolutely, then
\[
f_*\omega \in \VA{\C{n,m}}[j-k],
\]
where the weight of the form has dropped by
\[
k = \begin{cases}
n'-n+1  & m=0 \textrm{ and }m' > 0 \\
n' - n & \textrm{otherwise}.
\end{cases}
\]
\end{theorem*}
\autoref{thm:disk-period} above then follows easily by taking $f$ to be the projection to a point; see \autoref{cor:arnold-integral}.  With the general theory in place, we  explain in \autoref{sec:formality} how to apply the theory to the Feynman integrals~\eqref{eq:KW} appearing in the formality morphism.  (This is simply a matter of recalling the definitions.)  We also give a detailed example that illustrates how our algorithm works in practice.  We close the main body of the paper in \autoref{sec:software} with an overview of our software package.  Finally, the appendices contain proofs of some elementary facts about polylogarithms and MZVs that are needed to establish our result with $\ZZ$-coefficients but which we felt might distract from the core narrative of the paper.

In closing, let us remark that our approach to the fibre integration is very much inspired by Schnetz' method of ``single-valued integration'', introduced in his study of graphical functions in $\phi^4$ theory~\cite{Schnetz2014}.  There he gives a residue theorem for calculating  volume integrals of two-forms on $\PP[1]$ whose coefficients lie in Brown's algebra~\cite{Brown2004,Brown2004a} of univariate single-valued polylogarithms.  The works \cite{Brown2018,DDDDDDMPV:MR,SchlottererSchnetz:ClosedStringGenus0,StiebergerTaylor:ClosedSV} build on this idea to study volume integrals over $\M{N}$ itself, with applications to superstring amplitudes.  Note however, that such integrals give rise to a restricted class of MZVs (the ``single-valued MZVs'', see \cite{Brown:SingleValuedMZV}) which exclude, for example, the generator $\mzv{3,5} \in \MZV[8]$ that appears in our work.  This difference stems from the subtle contributions made by the boundaries of disks in our setting.  Our cohomological viewpoint on the cancellation of monodromy in the single-valued integration process seems to be new, and it is straightforward to apply it in other settings, including the volume integrals on $\M{N}$ studied in \emph{op.~cit.} (see, e.g.~\autoref{ex:MS-sv} below).

\paragraph{Acknowledgements:} We thank Anton Alekseev, Francis Brown, Ricardo Buring, Damien Ca\-laque, Ricardo Campos, Pierre Cartier, Cl\'ement Dupont, Lionel Mason, Sergei Merkulov, Maxim Kontsevich, Oliver Schnetz, Pavol \v{S}evera, Thomas Willwacher and Federico Zerbini for their interest in our work, and for helpful conversations and correspondence.
P.~B.\ was supported by a Vacation Bursary from the United Kingdom Engineering and Physical Sciences Research Council (EPSRC). 
E.~P.\ was supported by All Souls College (Oxford) and thanks the Hausdorff Research Institute for Mathematics for hospitality during the program on Periods in Number Theory, Algebraic Geometry and Physics.
B.~P.\ was supported at various stages by EPSRC Grant EP/K033654/1, by Jesus College (Oxford), and by the Seggie Brown Bequest at the University of Edinburgh.  Finally, we thank the Mathematical Institute at the University of Oxford and its IT staff for support and the provision of computing facilities.

%%%%%%%%%%%%%%%%%%%%%%

\section{Universal polylogarithms}
\label{sec:polylogs}
\subsection{Regularized integrals on genus zero curves}

We begin by recalling some basic constructions involving iterated integrals of one-forms on marked genus zero curves \cite{Brown2013,Hain:LecturesHodgeDeRham,Brown2009,Goncharov:MplMixedTateMotives,LappoDanilevsky,Panzer:PhD}.

\subsubsection{Marked curves}

Throughout this paper, $S$ will be a finite set of size $\abs{S} \geq 2$, and we will denote by $\pInf{S} \defas S \sqcup \sset{\infty}$ its extension by an additional, distinguished element $\infty$.

By an \defn{$S$-marked curve} $\X_S$ we shall mean a smooth complex projective curve $\X$ of genus zero equipped with an embedding $\pInf{S} \injection \X$. The elements of $\pInf{S}$ are called \defn{marked points} and we will typically not distinguish them from their image in $\X$.  We denote by 
\[
	\Xo_S \defas \X \setminus \pInf{S}
\]
the punctured curve obtained by removing all of the marked points.

For every pair $q,r \in S$ of distinct marked points there is a unique isomorphism from $\X$ to the complex projective line $\PP[1]$ that sends $(r,q,\infty)$ to $(0,1,\infty)$. It maps a point $p \in \X$ to the cross ratio $\crossrat{p}{q}{r}{\infty}$ defined by
\begin{equation}
	\crossrat{p}{q}{r}{s}
	\defas
	\frac{(p-r)(q-s)}{(p-s)(q-r)}
	\in \PP[1] 
	,
	\label{eq:cross-ratio}
\end{equation}
and this identifies $S$ with a subset of $\CC$ containing $\sset{0,1}$, so that $\Xo_S \cong \CC\!\setminus\! S$.  In this way we obtain a global holomorphic coordinate $z$ on $\Xo_S$.  We call  such a coordinate chart \defn{admissible}.

\subsubsection{Tangential base points and the fundamental groupoid}
\label{sec:tangents}
If we fix an admissible identification $z\colon \X \rightarrow \PP[1]$, we may equip an $S$-marked curve $\X_S$ with a collection of non-zero tangent vectors at marked points, called its \defn{tangential base points}, as follows.  At every finite marked point $s\in S$, we take the two unit vectors $\pm \cvf{z}$.  At the infinite marked point $\infty$, we take the unit tangent vectors $\pm \cvf{w}$ in the coordinate $w = z^{-1}$.  We will often denote a tangential base point at some $s \in \pInf{S}$ by the symbol $\st$.

\begin{example}
	Suppose $\abs{S}=2$.  Then $\Xo_S$ can be identified with $\PP[1]\setminus\sset{0,1,\infty}$. We get six tangential base points overall, and they form a torsor for the symmetric group $\Aut{\Xo_S} \cong \Perms{3}$ which permutes the marked points $\{0,1,\infty\}$. \qed
\end{example}

Suppose that $\st$ is a tangential base point on $\X_S$.  A \defn{path starting from $\st$} is a smooth path $\gamma\colon [0,1] \to \X$ such that $\gamma'(0) = \st$ and $\gamma(t)$ lies in the punctured curve $\Xo_S$ when $t > 0$. Similarly, we may speak of paths that end at $\st$, where they have final velocity $\gamma'(1) = - \st$, or paths which start and end at tangential base points (but otherwise lie in $\Xo_S$).

There is an obvious notion of homotopy between such paths, and the homotopy classes have a natural composition. The resulting fundamental groupoid
\[
	\FGpd{\X_S} = \frac{\sset{\textrm{paths between tangential base points}}}{\sset{\textrm{homotopies that fix tangential base points}}}
\]
is finitely generated, e.g.\ by a collection of loops around each marked point $s\in S$ together with a collection of paths that connects all of the tangential base points.
The group of loops based at any fixed tangential base point $\st$ is isomorphic to the fundamental group of the punctured curve $\Xo_S$.
\begin{remark}
	Our definition of $\FGpd{\X_S}$ depends on the choice of an admissible coordinate $z \colon \X \cong \PP[1]$, but below we will use $\FGpd{\X_S}$ to define algebras of periods and hyperlogarithms that will turn out to be independent of this choice. \qed
\end{remark}

\subsubsection{Regularized limits}

Let $\X_S$ be an $S$-marked curve with a tangential base point $\st$.  A \defn{neighbourhood of $\st$} is an open set in the punctured curve $\Xo_S$ that contains a sector $\set{z}{0<\abs{z}<\varepsilon\ \text{and}\ \arg z\in (-\varepsilon,\varepsilon)}$ for some $\varepsilon>0$, and some holomorphic coordinate $z$ that identifies $\st$ with $\partial_z|_{z=0}$. 

Suppose that $f$ is a holomorphic function defined in a contractible neighbourhood $\W \subset \Xo_S$ of $\st$.  We say that $f$ has \defn{logarithmic singularities at $s$} if in some (and hence any) holomorphic coordinate $z$ such that $\partial_z|_{z=0} = \st$, the function $f$ admits an expansion of the form
\begin{equation}
	f = \sum_{j=0}^n f_j \cdot (\log z)^j
	\label{eq:log-singularities}
\end{equation}
where the functions $f_j$ are holomorphic at $s$ and $\log z$ denotes the principal branch of the logarithm on $\W$. Such an expansion is necessarily unique. 

If $f$ has logarithmic singularities, we may define its \defn{regularized limit} at the tangential base point $\st$ by formally setting all positive powers of the logarithm to zero, and evaluating at $s$:
\[
	\Reglim{z}{0}  f \defas f_0(s) \in \CC.
\]
The result is independent of the coordinate used in its definition, i.e.~it depends only on the tangential base point.

The operation of taking regularized limits at a tangential base point commutes with sums and products (i.e.~it is an algebra homomorphism from the expressions of the form \eqref{eq:log-singularities} to the complex numbers).  Moreover, it coincides with the actual limit $f$ at $s$ whenever the latter exists; this occurs precisely when $f_j(0)=0$ for all $j>0$. In particular, such convergent limits do not depend on the tangent vector $\st$ and are invariant under switching the branch of $f$ by analytic continuation around a small loop at $s$.

\subsubsection{Differential forms}

The homology classes of simple counter-clockwise loops around the points of $S$ give a canonical basis for the first homology of the punctured curve:
\[
	\hlgy[1]{\Xo_S;\ZZ} \cong \ZZ^S.
\]
Via the residue theorem, the first cohomology is isomorphic to the $\ZZ$-module
\begin{equation}
	\A[1]{\X_S}  \subset \forms[1]{\Xo_S}
	\label{eq:XS-A1}%
\end{equation}
consisting of holomorphic forms on $\Xo_S$ with at most simple poles at the marked points, whose residues are integer multiples of $\tfrac{1}{2\ipi}$. Under this isomorphism, the basis element $\dlog{s} \in \A[1]{\X_S}$ dual to the loop around $s \in S$ is given by the unique form with residue $\tfrac{1}{2\ipi}$ at $s$, residue $\tfrac{-1}{2\ipi}$ at $\infty$ and zero residue at all other marked points. In any admissible coordinate $z$, we have
\begin{equation*}
	\dlog{s} = \invtate\tdlog(z-s) = \invtate\frac{\td z}{z-s}
	.
\end{equation*}
for $s \in S$.  It will also be convenient to adopt the convention that
\[
\dlog{\infty} \defas 0.
\]

\subsubsection{Regularized iterated integrals}

For a ring $R$ and a finite set $S$, we denote by $R S^n$ the free $R$-module generated by the set $S^n$ of words of length $n$ in the alphabet $S$.  We denote by 
\[
R\abrac{S} \defas \bigoplus_{n\ge 0} RS^n
\]
the free $R$-algebra generated by $S$.
In light of the isomorphism $\A[1]{\X_S} \cong \ZZ^S$, the tensor algebra of the $\ZZ$-module $\A[1]{\X_S}$ is canonically identified with the free algebra $\T{S}$.  We use bar notations for tensors:
\[
	[\alpha_1|\cdots|\alpha_n] \defas \alpha_1\otimes \cdots\otimes \alpha_n \in \A[1]{\X_S}^{\tp n}
\]
for $\alpha_1,\ldots,\alpha_n \in \A[1]{\X_S}$, so that the isomorphism $\T[n]{S} \to \A[1]{\X_S}^{\otimes n}$ is given concretely by
\[
	s_1\!\cdots s_n \mapsto [\dlog{s_1}|\cdots|\dlog{s_n}].
\]
If $\gamma$ is a smooth path in the punctured curve $\Xo_S$, recall that Chen's \cite{Chen:IteratedPathIntegrals} \defn{iterated integral} of $\alpha_1,\ldots,\alpha_n \in \A[1]{\X_S}$ along $\gamma$ is defined by
\begin{equation}
	\int_\gamma [\alpha_1|\cdots|\alpha_n] 
	= \int_{0}^1 (\gamma^*\alpha_1)(t_1) \int_0^{t_1} (\gamma^\ast \alpha_2) (t_2)\ \cdots \int_0^{t_{n-1}} (\gamma^*\alpha_n)(t_n).
	\label{eq:II}%
\end{equation}
It depends only on the homotopy class of $\gamma$ and may thus be viewed as a multivalued holomorphic function of the endpoints $\gamma(0)$ and $\gamma(1) \in \Xo_S$.

This function has logarithmic singularities at the marked points \cite[Proposition~2.14]{Goncharov:MplMixedTateMotives}, so for a path $\gamma$ that starts and/or ends at a tangential base point, we may define the regularized integral
\[
	\int_\gamma [\alpha_1|\cdots|\alpha_n] 
	\defas \Rlim_{\epsilon \to 0} \int_{\gamma|_{[\epsilon,1-\epsilon]}} [\alpha_1|\cdots|\alpha_n].
\]
It depends only on the homotopy class of $\gamma$, considered as a path between fixed (tangential) base points. If $\alpha_1$ has no pole at $\gamma(1)$ and $\alpha_n$ has no pole at $\gamma(0)$, then this reduces to the ordinary iterated integral, which converges absolutely in this case and is independent of the choice of tangent vectors at $\gamma(0)$ and $\gamma(1)$. This follows from \cite[Proposition~2.14]{Goncharov:MplMixedTateMotives} or \cite[Lemma~3.3.16]{Panzer:PhD}.

Let us give some examples of regularization in the simplest case $n=1$, where iterated integrals are just ordinary integrals:
\begin{example}\label{ex:convergent-wt1}
	Suppose that $p,q,r,s \in \pInf{S}$ are four distinct marked points and $\gamma$ denotes any path from $p$ to $q$. In any admissible coordinate $z$, the integral
	\begin{equation}
		\int_\gamma (s-r) = \int_{\gamma} [\dlog{s} - \dlog{r}] = \left.\frac{\log(z-s)-\log(z-r)}{2\ipi}\right|^{z=q}_{z=p} = 
		 \frac{\log \crossrat{p}{q}{r}{s}}{2\ipi}
		\label{eq:log-cross-ratio}%
	\end{equation}
	is absolutely convergent and determines a branch of the logarithm of the cross ratio.  The branch depends on the homotopy class of $\gamma$, and a change in homotopy class shifts the expression \eqref{eq:log-cross-ratio} by an integer. \qed
\end{example}

\begin{example}\label{ex:reg-one-end}
	The differential form $\dlog{s}$ has a pole at $s$, and hence its integral along a path $\gamma$ from $\st$ to $q \in S \setminus \sset{s}$ requires regularization.  Let us identify $\X$ with $\PP[1]$ by choosing an admissible chart $z$ such that $\st = \cvf{z}$.  We then calculate
	\begin{equation}
		\int_{\gamma} [\dlog{s}]
		= \Big(\Reglim{z}{q} - \Reglim{z}{s}\Big) \frac{\log(z-s)}{2\ipi}
		= \frac{\log(q-s)}{2\ipi} 
		\label{eq:log-reg0}%
	\end{equation}
where once again the branch of the logarithm is determined by the homotopy class of $\gamma$. Observe that since
	\[
	\log(q-s) = \log (q,1;s,\infty) + \log (1,\infty;s,0),
	\]
	the integral \eqref{eq:log-reg0} can also be expressed in terms of cross ratios. \qed 
\end{example}

\begin{example}\label{ex:reg-half-loop}
In a similar vein, consider a path $\gamma$ that starts at a tangential base point $\st$ and wraps around $s$, traversing a total angle $n\pi$ before returning to $s$ along the tangent vector $(-1)^n \st$. The regularized integral
\[
	\int_\gamma [\dlog{s}] = \frac{n\ipi}{2\ipi} = \frac{n}{2} \in \tfrac{1}{2}\ZZ
\]
measures the change in  $\log(z-s)$ along $\gamma$, which is $\ipi$ for each half-loop. \qed
\end{example}
Because regularized iterated integration along a path $\gamma$ is defined for an arbitrary word in the forms $\A[1]{\X_S}$, we can view it as a $\ZZ$-linear map
\[
	\int_\gamma \colon \T{S} \longrightarrow \CC.
\]
It has the following important properties \cite{LappoDanilevsky,Goncharov:MplMixedTateMotives}:
\begin{enumerate}
\item 	For any two words $u,v \in \T{S}$, we have
\begin{equation}
	\rbrac{\int_{\gamma} u} \cdot \rbrac{\int_\gamma v}  = \int_{\gamma}(u\shuffle v), 
\label{eq:shuffle}%
\end{equation}
	where $\shuffle \colon \T{S} \otimes \T{S} \to \T{S}$ is the commutative shuffle product.

\item 	If $\gamma = \gamma_1\PathConc \gamma_2$ is the concatenation of two paths with $\gamma_1'(0)=\gamma_2'(1)$, then 
\begin{equation}
	\int_\gamma s_1\cdots s_k
	= \sum_{j=0}^k
	\rbrac{ \int_{\gamma_1}s_1 \cdots s_j }
	\rbrac{ \int_{\gamma_2}s_{j+1}\cdots s_k }.
	\label{eq:path-concat}%
\end{equation}

\item The integral along the reversed path is given by
\begin{equation}
	\int_{\gamma^{-1}} s_1\cdots s_n
	= (-1)^n \int_{\gamma} s_n\cdots s_1.
	\label{eq:antipode}%
\end{equation}

\item If $\gamma$ is an identity element of $\FGpd{\X_S}$, i.e.~a contractible loop based at some tangential base point, then 
	\begin{equation}
		\int_{\gamma} \T[>0]{S} = 0.
		\label{eq:II-identity}%
	\end{equation}
\end{enumerate}

If one or both ends of a path $\gamma$ is at a tangential base point, the regularized iterated integral along $\gamma$ can be computed algorithmically using the principle of \defn{shuffle regularization}, as follows.   We explain the case of a path $\gamma \in \FGpd{\X_S}$ from a tangential base point $\pt$ to another tangential base point $\qt$; the case in which one end lies in $\Xo_S$ is treated similarly.

 Let us choose forms $\alpha_p,\alpha_q \in \A[1]{\X_S}$ that are dual to simple loops around $p$ and $q$, i.e.
\[
\begin{pmatrix}
\Res_q \alpha_q &\Res_p \alpha_q \\
\Res_q \alpha_p & \Res_p \alpha_p
\end{pmatrix} = \frac{1}{2\ipi}\begin{pmatrix}
1 & 0 \\ 0 & 1
\end{pmatrix}.
\]
Thus if $p,q \in S$ we may simply take $\alpha_q = \dlog{q}$ and $\alpha_p=\dlog{p}$,  whereas if $q = \infty$, we may choose an additional marked point $r \ne p,q$ arbitrarily and take $\alpha_q = -\dlog{r}$ and $\alpha_p = \dlog{p}-\dlog{r}$.  Then every element $v\in \T[n]{S}$ can be written uniquely in the form
\begin{equation}
	v = \sum_{j+k\leq n} \alpha_q^j \shuffle v_{jk} \shuffle \alpha_p^k
	\label{eq:shreg-expansion}%
\end{equation}
where $v_{jk} \in \T[n-j-k]{S}$ is a linear combination of tensor monomials that start with a form that is smooth at $q$ and end with a form that is smooth $p$.  From the shuffle product formula \eqref{eq:shuffle}, we have
\begin{equation}
	\int_{\gamma} v 
	= \sum_{j+k \leq n}
	\frac{1}{j!} \left( \int_{\gamma} \alpha_q \right)^j
	\frac{1}{k!} \left( \int_{\gamma} \alpha_p \right)^k
	\int_{\gamma} v_{jk}.
	\label{eq:shreg-expansion-int}%
\end{equation}
Note that the last factor is a convergent iterated integral and thus independent of the tangential base points at $p$ and $q$.  In this way we reduce the calculation to convergent iterated integrals, and regularized logarithms.  The latter can then be evaluated directly as in \autoref{ex:reg-one-end} and \autoref{ex:reg-half-loop} above.

\subsection{Periods of marked curves}
\begin{definition}
	Let $\X_S$ be an $S$-marked curve.  The \defn{ring of periods of $\X_S$} is the subring
	\begin{equation}
	\Per{} = \Per{\X_S} \defas \abrac{ \int_{\FGpd{\X_S}} \T{S} } \subset \CC	\label{eq:periods}%
	\end{equation}
	generated by iterated integrals between pairs of tangential base points.
\end{definition}
For every $n\geq 0$, there is a $\ZZ$-submodule $\Per[n]{} \subset \Per{}$ of \defn{periods of weight $n$}, spanned by monomials $\prod_i \int_{\gamma_i} u_i$ where the elements $u_i \in \T[n_i]{S}$ have total weight $\sum_i n_i = n$. By definition, we have $\Per[0]{} = \ZZ$ and $\Per[m]{} \cdot \Per[n]{} \subset \Per[n+m]{}$.  By \autoref{ex:reg-half-loop}, we have that $1 \in \Per[1]{}$, and therefore $\Per[n]{} \subset \Per[n+1]{}$ for all $n$.  Thus the ring $\Per{}$ is filtered by the weight. It receives a homomorphism
\[
	\int_\gamma \colon \T{S} \longrightarrow \Per{}
\]
for any path $\gamma \in \FGpd{\X_S}$, where as above $\T{S}$ is viewed as a commutative algebra equipped with the shuffle product.
By the path concatenation and inversion formulae \eqref{eq:path-concat} and \eqref{eq:antipode}, the ring $\Per{}$ is generated by periods $\int_{\gamma} u$ where $\gamma$ is drawn from a finite collection of paths that generate $\FGpd{\X_S}$. It follows that each $\Per[n]{}$ is a $\ZZ$-module of finite rank.

All periods of weight one can be computed along the lines of Examples \ref{ex:convergent-wt1} through \ref{ex:reg-half-loop} above, giving the following elementary description:
\begin{lemma}
The $\ZZ$-submodule $\Per[1]{} \subset \CC$ is generated by $\tfrac{1}{2}$ and the numbers $\frac{\log(c)}{2 \ipi}$, where $c$ runs over all possible cross ratios of marked points on $\X_S$.
\label{lem:periods-weight-one}%
\end{lemma}

In particular, we see that the $\ZZ$-module $\Per[1]{}$ depends only on the isomorphism class of $\X_S$, and not on the admissible chart used to define the tangential base points and the fundamental groupoid. It is also insensitive to the distinction of the label $\infty$, and thus completely determined by the isomorphism class of the complex manifold $\Xo_S$ without any reference to labels.

By shuffle regularization~\eqref{eq:shreg-expansion-int}, every period of higher weight is a polynomial in weight one periods and convergent iterated integrals. Since the latter are also insensitive to the tangential base points, we have the following:
\begin{lemma}
	The weight-filtered ring $\Per{\X_S}$ of periods depends only on the isomorphism class of the punctured curve $\Xo_S$.
	\label{lem:V-XS-canonical}%
\end{lemma}

If $\X_S$ is a marked curve and $S' \subset S$, let $\X_{S'}$ be the $S'$-marked curve obtained from $\X_S$ by forgetting all marked points in $S \setminus S'$.  Then we have an obvious inclusion $\Per{\X_{S'}} \subset \Per{\X_{S}}$.  In particular, taking $\abs{S'} = 2$, we see that the periods of any marked curve include the periods of the projective line $\PP[1]$ equipped with the marked points $S' = \sset{0,1}$ and $\infty$.  It is well known that such periods can be written in terms of MZVs, see \cite{LeMurakami:KontsevichKauffman}.  More precisely, we have the following
\begin{theorem}\label{thm:M03-periods}
If $|S| = 2$, then the periods of $\X_S$ are exactly the normalized multiple zeta values, that is, we have an equality of filtered rings
\[
\Per{\X_S} = \MZVipi \subset \CC.
\]
\end{theorem}

\begin{proof}
The result is standard when the coefficients are taken in $\QQ$, but a comment is required because we work with integer coefficients.

By path concatenation, path inversion, and automorphisms of $\PP[1]\setminus\{0,1,\infty\}$, we can reduce every period to a polynomial in iterated integrals along two paths: the straight line path $\dch$ from $0$ to $1$, and a half-loop at zero as in \autoref{ex:reg-half-loop}.

The convergent integrals along $\dch$ give the MZVs \cite[Proposition~A.4]{LeMurakami:KontsevichKauffman}:
\[
	\frac{\mzv{n_1,\ldots,n_d}}{(2\ipi)^{n_1+\cdots+n_d}} 
	= \int_{\dch} [
		\underbrace{\dlog{0}|\cdots|\dlog{0}|{\scriptstyle-}\dlog{1}}_{n_d}\!|
		\cdots
		|\!\underbrace{\dlog{0}|\cdots|\dlog{0}|{\scriptstyle-}\dlog{1}}_{n_1}]
	\in \MZVipi[n_1+\cdots+n_d].
\]
Hence by shuffle regularization, it is enough to show that $(\Per[1]{})^n/n! \subset \Per[n]{}$.  But since there are only three marked points, there are no possible cross ratios, and hence $\Per[1]{} = \tfrac{1}{2}\ZZ$ by \autoref{lem:periods-weight-one}.  Therefore
\begin{equation*}
	\Per[n]{} = \sum_{k=0}^n \frac{1}{2^k k!} \frac{\MZV[n-k]}{(2\ipi)^{n-k}}.
\end{equation*}
It remains to observe that $1/(2^{k}k!) \in (\MZV[k]+\ipi\MZV[k-1])/(2\ipi)^k$, so that the sum on the right collapses to the definition \eqref{eq:normalized-mzv} of $\MZVipi$.  This follows from the identity
\begin{equation}
	\frac{\pi^{2m}}{(2m+1)!}
	= \mzv{\underbrace{2,\ldots,2}_{m \textrm{ times}}},
	\label{eq:Hoffman-formula}%
\end{equation}
which is due to Hoffman~\cite[Corollary~2.3]{Hoffman:MultipleHarmonicSeries} and valid for all $m\geq 1$.
\end{proof}
\begin{example}\label{ex:reg-MZV}
	The iterated integral of $w=[\dlog{1}|\dlog{0}|\dlog{0}|\dlog{1}]$ along $\dch$ is divergent, due to the singularity $\int_{0}^{1-\epsilon} [\dlog{1}] = \log\epsilon$ produced by the first letter. But since
	\begin{equation*}
		w = [\dlog{1}] \shuffle [\dlog{0}|\dlog{0}|\dlog{1}]
		- [\dlog{0}|\dlog{1}|\dlog{0}|\dlog{1}]
		-2[\dlog{0}|\dlog{0}|\dlog{1}|\dlog{1}]
	\end{equation*}
	and the iterated integration respects the shuffle product, we conclude that
	\begin{align*}
		\int_{\dch} w
		&= \left(\Reglim{\epsilon}{0} \log \epsilon\right) \int_{\dch} [\dlog{0}|\dlog{0}|\dlog{1}]
		- \int_{\dch}[\dlog{0}|\dlog{1}|\dlog{0}|\dlog{1}]
		- 2\int_{\dch}[\dlog{0}|\dlog{0}|\dlog{1}|\dlog{1}]
		\\
		&= 0\cdot \frac{-\mzv{3}}{(2\ipi)^3}- \frac{\mzv{2,2}}{(2\ipi)^4} - 2\frac{\mzv{1,3}}{(2\ipi)^4}
		 = -\frac{1}{1152}
		 \in \MZVipi[4]
	\end{align*}
	using the identities $\mzv{1,3}=\frac{1}{4}\mzv{4}$, $\mzv{2,2}=\frac{3}{4} \mzv{4}$ and $\mzv{4}=\frac{\pi^4}{90}$ as in \cite{Waldschmidt:LecturesMZV}. \qed
\end{example}
Later on, we will use the fact that periods of closed loops have lower weight:
\begin{lemma}\label{lem:loop-weight}
Let $\X_S$ be any $S$-marked curve, and suppose that $\gamma \in \FGpd{\X_S}$ is a closed loop at some tangential base point of $\X_S$. Then for any word $s_1\cdots s_n \in \T{S}$ of length $n \geq 1$, we have
\[
	\int_\gamma s_1\cdots s_n \in \Per[n-1]{\X_S}.
\]
\end{lemma}

\begin{proof}
	The fundamental group based at $\st$ is generated by paths of the form $\gamma = \eta^{-1} \PathConc \delta  \PathConc \eta$, where $\eta$ is a path from $\st$ to some marked point $p \in S$ and $\delta$ is a small loop around $p$. 
	By virtue of the path concatenation and inversion formulas, it suffices to consider the periods along such generators.  By the path concatenation formula, we have
	\begin{align}
		\int_{\gamma} s_1\cdots s_n 
		&= \sum_{0\le j \le k \le n}
		\int_{\eta^{-1}} s_1\cdots s_j \!
		\int_{\delta} \!\! s_{j+1}\cdots s_k \!
		\int_{\eta} s_{k+1} \cdots s_n. \label{eq:triple-concat}
	\end{align}
	Note that the terms with $j=k$ (i.e.~with no integrals over $\delta$) combine to give the iterated integral $\int_{\eta^{-1} \PathConc\eta} s_1\cdots s_n$, which vanishes since $\eta^{-1} \PathConc\eta$ is a contractible loop. Hence it is enough to show that each term with $j < k$ has weight at most $n-1$.  By a direct calculation (see e.g.~\cite[Lemma~2.5]{Goncharov:MplMixedTateMotives}), it follows that
	\begin{equation*}
		\int_{\delta} s_{j+1}\cdots s_k
		= \begin{cases}
			1/(k-j)! & \text{if $s_{j+1}=\cdots = s_k=p$ and}\\
			0            & \text{otherwise.} \\
		\end{cases}
	\end{equation*}
	This lies in $\MZVipi[(k-j)-1] \subseteq \Per[(k-j)-1]{\X_S}$ by \autoref{lem:(n+1)!} in the appendix, and the result follows.
\end{proof}

\begin{example}\label{ex:wt2-loop}
Let us fix a pair $p,q \in S$ of distinct marked points, and an additional marked point $t \in S$.   We compute the period
\[
I \defas \int_{\gamma_t} qp
\]
where $\gamma_t \in \FGpd{\X_S}$ is a loop based at $\tinf$ whose homology class is dual to $\dlog{t}$.    Following the recipe in the proof of \autoref{lem:periods-weight-one}, we write $\gamma_t = \eta^{-1} \PathConc\delta  \PathConc\eta$ for a path $\eta$ from $\tinf$ to $\ttt$ and a small loop $\delta$ around $t$.  The path concatenation expansion then collapses to 
\begin{align*}
I = \int_{\eta^{-1}} q \int_\delta p +\int_{\delta}q \int_\eta p
= \begin{cases}
-\tfrac{1}{2\ipi}\log(p-q) & t=p \\
\tfrac{1}{2\ipi}\log(q-p) & t=q \\
0 & t \ne q,p
\end{cases}
\end{align*}
where the branch of logarithm is determined by the homotopy class of $\eta$ (or equivalently, the homotopy class of $\gamma_t$). \qed
\end{example}

\subsection{Hyperlogarithms}
\label{sec:hlogs}

Suppose that $\X_S$ is a marked curve and $\gamma$ is a path from some tangential base point $\pt$ to a point $z \in \Xo_S$ inside the punctured curve. As we vary $z$, the path $\gamma$ lifts locally to a smooth family of paths $\gamma_z$ from $\pt$ to $z$. Every element $v \in \T{S}$ therefore defines a germ
\begin{equation}
	\hlog{v}(z)
	\defas \int_{\gamma_z} v % [\dlog{s_1}|\cdots|\dlog{s_n}]
	\label{eq:def-hlog}%
\end{equation}
of a holomorphic, multivalued function on $\Xo_S$.  Such a function is called a hyperlogarithm \cite{LappoDanilevsky:algorithmique,LappoDanilevsky,Poincare1884}. 

Observe that if $\widetilde{\gamma}$ is another path to $z$ starting from some tangential base point $\qt$, then we have $\widetilde\gamma = \gamma \PathConc \xi$ for some path $\xi$ from $\qt$ to $\pt$, and hence by the path concatenation formula, the corresponding hyperlogarithms are related by
\begin{equation}
	\widetilde{\hlog{s_1\cdots s_n}}
	= \hlog{s_1\cdots s_n}
	+ \hlog{s_1\cdots s_{n-1}} \cdot \int_\xi s_n
	+ \hlog{s_1\cdots s_{n-2}} \cdot \int_\xi s_{n-1}s_n
	+ \cdots 
\label{eq:hyperlog-monodromy}%
\end{equation}
The coefficients $\int_\xi s_j\cdots s_n \in \Per{}$ are periods of the curve $\X_S$.  Therefore the $\Per{}$-linear combinations of hyperlogarithms generate a canonically defined subsheaf
\[
	\HL {}  \subset \sO{\Xo_S}
\]
of analytic functions on $\Xo_S$.  It is a locally constant sheaf of algebras over the ring $\Per{}$ of periods.  By definition, any path $\gamma$ as above defines a surjective algebra homomorphism to the stalk
\begin{equation}
	\int_\gamma\colon (\PerT{S},\shuffle) \longrightarrow \HL{z}. \label{eq:hyperlog-germ}
\end{equation}
In fact, this map is also injective \cite[Corollary~5.6]{Brown2009}:
\begin{proposition}\label{prop:hlog-iso}
	The map \eqref{eq:hyperlog-germ} is an isomorphism of $\Per{}$-algebras.
\end{proposition}

Note that since $\Per{}$ is a filtered ring, the free algebra $\PerT{S}$ carries two natural filtrations.  More precisely, let us consider the $\ZZ$-submodule
\[
W_n^k\PerT{S} \defas \sum_{j \le k} \Per[n-j]{} \cdot S^{j} \subset \PerT{S}
\]
consisting of elements of total weight $n$ that involve monomials in the free algebra of length at most $k$.  Using \autoref{prop:hlog-iso}, we can transfer these filtrations to the sheaf $\HL{}$, giving locally constant subsheaves
\[
\HL[n]{}[k]  \subset \HL{}
\]
such that $\HL[n]{}[k] \subset \HL[n+1]{}[k] \cap \HL[n]{}[k+1]$.  It follows from \eqref{eq:hyperlog-monodromy} that these filtrations are independent of the path $\gamma$ used to define the isomorphism of $\PerT{S}$ with the stalks of $\HL{}$.  We refer to the filtration $\HL[\bullet]{}$ and $\HL{}[\bullet]$ as the filtrations by the \defn{total weight} and the \defn{relative weight}, respectively.  
We denote by
\begin{equation}
\gr \HL{} = \bigoplus_{k\ge 0} \HL{}[k]/\HL{}[k-1] \cong \underline{\PerT{S}}
\label{eq:hlog-gr}%
\end{equation}
the constant sheaf obtained by taking the associated graded with respect to the relative weight filtration.

The monodromy of a hyperlogarithm is computed by \eqref{eq:hyperlog-monodromy} as the special case when $\xi$ is a closed loop based at $\pt = \qt$.
The weight drop for the corresponding periods (\autoref{lem:loop-weight}) implies that the monodromy representation is unipotent with respect to both filtrations:
\begin{lemma}\label{lem:mon-unipotent}
	If $\hlog{} \in \HL[n]{z}[k]$ is the germ of a hyperlogarithm at $z\in\Xo_S$, then its analytic continuation $\gamma \cdot \hlog{}$ along a loop $\gamma \in \fg{\Xo_S,z}$ has the form
\[
	\gamma \cdot \hlog{} \equiv \hlog{} \mod  \HL[n-1]{z}[k-1].
\]
\end{lemma}

It follows immediately from the definition \eqref{eq:II} of iterated integrals that the de Rham differential acts on hyperlogarithms by
\begin{equation}
	\td \hlog{s_1\cdots s_n} = \hlog{s_2\cdots s_n} \cdot \dlog{s_1} \in \HL{}\otimes \A[1]{\X_S}
	\label{eq:hlog-diff}%
\end{equation}
so that we have a subcomplex
\[
	\VAr{\X_S} \defas \HL{} \otimes \A{\X_S} \subset \forms{\Xo_S}
\]
of the holomorphic de Rham complex of the punctured curve. It is filtered by the total and relative weights, where elements of $\A[1]{\X_S}$ are assigned weight one.
From \eqref{eq:hlog-diff} it is clear that the stalk complex of $\VAr{\X_S}$ at any point admits a contracting homotopy defined on generators by
\begin{align}
L_{s_1\cdots s_n} \dlog{s} \mapsto L_{ss_1\cdots s_n}. \label{eq:hyperlog-primitive}
\end{align}
We therefore have the following refinement of the Poincar\'e lemma:
\begin{lemma}\label{lem:hlog-Poincare}
	The inclusion $\underline{\Per{}} \hookrightarrow \VAr{\X_S}$ of the constants is a quasi-isomorphism of sheaves of filtered differential graded algebras.
\end{lemma}

Hyperlogarithms based at a tangential base point $\st=\cvf{z}$ have logarithmic singularities at the punctures $\pInf{S}$. The expansion \eqref{eq:log-singularities} at the base point $\st$ can be obtained explicitly using  shuffle regularization as in \eqref{eq:shreg-expansion}.  Indeed, we may write any $v \in \T{S}$ as a shuffle polynomial $v = \sum_k \dlog{s}^{\tp k} \shuffle v_k$, where $v_k \in \T{S}$ is a linear combination of monomials whose last factor has no singularities at $s$. Therefore
\begin{equation}
	\hlog{v}(z)
	= \sum_{k=0}^n \frac{\hlog{v_k}(z)}{k!} \rbrac{\frac{\log(z-s)}{2\ipi}}^k
	\label{eq:hlog-zero-expansion}%
\end{equation}
where the hyperlogarithms $L_{v_k}$ are holomorphic at $s$.

\subsection{Universal polylogarithms}

We now recall the construction of polylogarithms on the moduli space of marked curves, following Brown~\cite{Brown2009} and Goncharov~\cite{Goncharov2004}.  The main results in this section are due to them, although our exposition differs.  The primary technical difference is that we work with coefficients in $\ZZ$ rather than $\QQ$, so we justify at some points why this is possible.

\subsubsection{Moduli spaces and the universal curve}
Let us denote by  $\M{S}$ the moduli space parametrizing isomorphism classes of $S$-marked curves.  Since the set of marked points is $\pInf{S} = S \cup \sset{\infty}$, this is an abbreviation for the standard notation $\M{0,S_\infty}$ or $\M{0,N}$, where $N = |S|+1$.  Recall that $\M{S}$ is a complex manifold of dimension $|S|-2$ and it comes equipped with a universal family of marked curves
\[
\uX{S} \to \M{S}
\]
such that the fibre $\uX{S,y}$ over any point $y \in \M{S}$ is a marked curve whose isomorphism class is $y$.  We denote by $\uXo{S} \subset \uX{S}$ the universal punctured curve. 

If $S \subset T$ is an inclusion of finite sets, there is a fibration $\M{T} \to \M{S}$ defined by forgetting the marked points labelled by $T \setminus S$.  The fibres have dimension $|T\setminus S|$. In particular if $|T|=|S|+1$, then this fibration is canonically identified with the universal punctured curve $\uXo{S}$.

\subsubsection{Cross ratios and differential forms}
For any quadruple $i,j,k,l \in \pInf{S}$ of markings, the cross ratio of the corresponding marked points gives a natural function
\begin{equation*}
	f= \crossrat{i}{j}{k}{l}\colon
	\M{S} \longrightarrow \CC \setminus\!\sset{0,1}.
\end{equation*}
  If $T =S \sqcup \sset{z}$ and we choose an embedding $\sset{0,1} \hookrightarrow S$ then the cross ratios of the form $(z,1;0;\infty)$ give admissible charts on the fibres of the universal curve $\uX{S} \supset \M{T}\to \M{S}$, and we obtain a trivialization
\[
	\uX{S} \cong \M{S} \times \PP[1]
\]
of the underlying fibration, so that the marked points in $\uX{S}$ are indicated by the cross ratios $\crossrat{s}{1}{0}{\infty}\colon \M{S} \to \PP[1]$ for each $s\in \pInf{S}$.

The logarithmic differentials of cross ratios generate a subring we denote by
\[
\A{\M{S}} \defas \abrac{\frac{\tdlog f}{2\ipi}\,\middle|\, f \textrm{ is a cross ratio}} 
\subset \forms{\M{S}},
\]
with the constants $\A[0]{\M{S}} \defas \ZZ \subset \forms[0]{\M{S}}$.

\begin{example}
	If $\abs{S}=2$, then $\M{S}$ is a point and $\A{\M{S}} = \ZZ$.  \qed
\end{example}

\begin{example}
If $\abs{S}=3$, any cross ratio $z \colon \M{S} \to \PP[1]\setminus\{0,1,\infty\}$ is an isomorphism.   The other five independent cross ratios can be expressed in terms of $z$ as $z^{-1},1-z,(1-z)^{-1},(z-1)z^{-1}$ and $z(z-1)^{-1}$, so that 
	\[
	\A[1]{\M{S}} =  \ZZ \frac{\td z}{2 \ipi z} \oplus \ZZ \frac{\td z}{2\ipi(z-1)}
	\]
	is the $\ZZ$-module of forms on $\PP[1]\setminus\{0,1,\infty\}$ considered above.\qed
\end{example}
As explained by Getzler~\cite[Section~3.8]{Getzler1995} and Brown~\cite[Section~6.1]{Brown2009}, Arnold's calculation~\cite{Arnold:ColoredBraid} of the cohomology of the planar configuration spaces implies that $\A{\M{S}}$ projects isomorphically onto the integral cohomology ring of the moduli space:
\begin{equation}
	\A{\M{S}}
	\cong
	\cohlgy{\M{S};\ZZ},
	\label{eq:M-cohomology}%
\end{equation}
and this allows for the explicit presentation of $\A{\M{S}}$ in terms of generators and relations (called Arnold's relations).

In particular, if $|T| = |S|+1$ and we identify $\M{T}$ with the universal punctured curve, then the restriction of an element $\alpha \in \A[1]{\M{T}}$ to a fibre $y$ gives an element $\alpha|_{\uX{S,y}} \in \A[1]{\uX{S,y}} \cong \ZZ^S$.  Done in families, this gives rise to an exact sequence
\[
\xymatrix{
0 \ar[r]& \A[1]{\M{S}} \ar[r]& \A[1]{\M{T}} \ar[r] & \A[1]{\uX{S}} \ar[r] & 0
}
\]
of free $\ZZ$-modules, where $\A[1]{\uX{S}} \cong \ZZ^S$ is the natural $\ZZ$-module of relative differentials on the universal curve.  A choice of splitting of this exact sequence gives an isomorphism of graded $\ZZ$-modules
\begin{equation}
\A{\M{T}} \cong \A{\M{S}} \otimes \A{\uX{S}}
\label{eq:universal-relative}%
\end{equation}
which one can use inductively to construct a basis for $\A{\M{S}}$. 

More generally, for an arbitrary projection $f \colon \M{T} \to \M{S}$ we may define the relative differentials
\[
\A{\M{T}/\M{S}} = \frac{\A{\M{T}}}{f^*\A[1]{\M{S}}\wedge \A{\M{T}}}
\]
and we obtain a non-canonical isomorphism
\begin{equation}
\A{\M{T}} \cong \A{\M{S}} \otimes \A{\M{T}/\M{S}} \label{eq:arnold-tensor}
\end{equation}
of $\ZZ$-modules that becomes canonical once we pass to the associated graded of the usual fibrewise Hodge filtration
\[
F^j \A{\M{T}} = f^*\A[j]{\M{S}}\wedge\A{\M{T}}.
\]

\subsubsection{Polylogarithms on moduli spaces}
Given an embedding $\sset{0,1} \injection S$, we obtain an admissible chart on each fibre of the universal curve $\uX{S}$, and the corresponding tangential base points give  global sections of the tangent bundle $\tb{\uX{S}}$.  Since the fibration is topologically locally trivial, the fundamental groupoids of the fibres assemble into a locally constant sheaf $\FGpd{\uX{S}}$ of groupoids on $\M{S}$. Given a local section $\gamma \in \FGpd{\uX{S}}$ and an element $v \in \T{S}$, we can compute the corresponding period integrals in the fibres to obtain a holomorphic function $\int_\gamma v$ on $\M{S}$.

\begin{definition}
	The \defn{sheaf of polylogarithms on $\M{S}$} is the subsheaf of rings
	\[
	\Vr{S} \defas \abrac{ \int_{\FGpd{\uX{S}}} \T{S} } \subset \sO{\M{S}}
	\]
	generated by period integrals on the universal curve.
\end{definition}

The sheaves $\Vr{S}$ are locally constant by analytic continuation and come equipped with an exhaustive increasing weight filtration
\[
\underline{\ZZ} = \Vr[0]{S}\subset\Vr[1]{S} \subset \cdots
\]
defined as above for periods. 
  By \autoref{lem:periods-weight-one}, $\Vr[1]{S}$ is the subsheaf of $\sO{\M{S}}$ generated by the constant $\tfrac{1}{2}$ and the branches of the functions $\tfrac{1}{2\ipi}\log f$ where $f$ is a cross ratio. One can write all germs of higher weight as special functions known as multiple polylogarithms \cite[Section~5.4]{Brown2009} but we will not need the explicit details of this description.  Note that while the construction of $\Vr{S}$ relies on a choice of labels $\{0,1\}\hookrightarrow S$, the same argument as in \autoref{lem:V-XS-canonical} shows that the subsheaf $\Vr{S} \subset \sO{\M{S}}$ is, in fact, independent of this choice.  Similarly, the argument in \autoref{lem:loop-weight} shows that periods of the form $\int_\gamma s_1\cdots s_n \in \Vr{S}$ have weight $n-1$ if $\gamma$ is a loop.

\begin{lemma}
The monodromy representation of $\fg{\M{S}}$ on the stalks of $\Vr{S}$ is unipotent with respect to the weight filtration $\Vr[\bullet]{S}$.
\end{lemma}

\begin{proof}
Consider the analytic continuation of a period integral $L = \int_\gamma v \in \Vr[n]{S}$ around a loop $\eta \in \fg{\M{S}}$.  It is given by the integral $\widetilde L \defas \int_{\eta \cdot \gamma} v$  where $\eta \cdot \gamma \in \FGpd{\uX{S}}$ is the fibrewise path obtained by parallel transport of $\gamma$ around $\eta$.  Note that $\eta\cdot \gamma$ and $\gamma$ have the same tangential base points.  Therefore $\eta \cdot \gamma = \gamma \star \tau$ is the composition of $\gamma$ with some loop $\tau$ in the fibre.  Hence the result follows from the path concatenation formula and \autoref{lem:loop-weight}.
\end{proof}

Any polylogarithm $L = \int_\gamma v$ satisfies a Picard--Fuchs equation (the dihedral Knizhnik--Zamolodchikov equation of~\cite{Brown2009}), which implies that
\[
\td L  \in \Vr{S}\otimes \A[1]{\M{S}} \subset \sforms[1]{\M{S}}.
\]
Note that by shuffle regularization, it is enough to compute the differential of an absolutely convergent integral. In the case of a single monomial $v=s_n^{k_n}\cdots s_1^{k_1}$, written minimally so that $s_i \ne s_{i+1}$ and $k_i \ge 1$, we see from \cite[Theorem~2.1]{Goncharov:MplMixedTateMotives}:
\begin{equation}
	\td L
	= \sum_{i=1}^n \frac{\td \log \crossrat{s_i}{\infty}{s_{i+1}}{s_{i-1}}}{2\ipi}
	\cdot \int_{\gamma} s_n^{k_n}\cdots s_{i+1}^{k_{i+1}}s_i^{k_i - 1}s_{i-1}^{k_{i-1}}\cdots s_1^{k_1}
	\label{eq:td-hlog}%
	,
\end{equation}
where $s_0 \defas \gamma(0)$ and $s_{n+1} \defas \gamma(1)$ denote the endpoints of integration.

As a result, the sheaf of \defn{polylogarithmic $k$-forms}
\begin{equation}
		\VAr{S}
		\defas \Vr{S} \otimes \A{\M{S}}
		\subset \sforms{\M{S}}
		\label{eq:VAr-M}%
	\end{equation}
is a subcomplex of the de Rham complex of $\M{S}$.  It carries a canonical filtration $\VAr{S}[\bullet]$ by the total weight:
\[
\VAr[k]{S}[j] \defas \Vr[j-k]{S}\otimes \A[k]{\M{S}} \subset \VAr[k]{S}.
\]

\subsubsection{Boundary divisors and regularized restrictions}

We now recall and refine some key results on the asymptotic behaviour of polylogarithms on $\M{S}$.

Recall that $\M{S}$ has a natural compactification, obtained by adding a simple normal crossings divisor $\bbM{S} = \bM{S} \setminus \M{S}$ that parametrizes stable nodal curves of  genus zero, or equivalently, nested linear collisions of collections of marked points.  Each boundary stratum of codimension $k$ is isomorphic to a product of the form $\M{A_0}\times \cdots \times \M{A_k}$ for some finite sets $A_j$ and hence it carries a natural sheaf of polylogarithms $\Vr{A_0}\cdots \Vr{A_k} \subset \sO{\M{A_0}\times\cdots\times \M{A_k}}$ which we equip with the natural filtration by the total weight.  

For any stratum of codimension $k$, we can find a collection of cross ratios $t_1,\ldots,t_k \colon \bM{S} \to \PP[1]$ that give a system of normal crossings coordinates transverse to the stratum. Over any sufficiently small open set $\U \subset \M{A_0}\times \cdots \times \M{A_k}$ they give transverse coordinates on an analytic tubular neighbourhood
\[
\U \times \D^k  \hookrightarrow \bM{S}
\]
where $\D \subset \PP[1]$ is a small disk centred at zero. 

Let $\W \subset \D$ be a contractible neighbourhood of the unit tangential base point at $0$.  Then we have the following $\ZZ$-linear refinement of a theorem of Brown:

\begin{theorem}[\cite{Brown2009}]\label{prop:moduli-log-sing}
Let $L \in \Vr[j]{S}$ be a polylogarithm defined on the open set $\U\times \W^k \subset \M{S}$.  Then there exists a unique finite collection of holomorphic functions $L_{j_1\cdots j_k}$ on $\U \times \D^k$ such that
\begin{align}
L = \sum_{j_1,\ldots,j_k} \frac{L_{j_1\cdots j_k}}{j_1!\cdots j_k!} \rbrac{\frac{\log t_1}{2\ipi}}^{j_1} \cdots \rbrac{\frac{\log t_k}{2\ipi}}^{j_k} \label{eq:boundary-expansion}
\end{align}
over $\U \times \W^k$.
Moreover, the coefficients restrict to polylogarithms on the boundary:
\begin{equation}
L_{j_1\cdots j_k}|_\U \in W_{j-j_1-\cdots-j_k} (\Vr{A_0}\cdots\Vr{A_k}). \label{eq:polylog-restrict}
\end{equation}
\end{theorem}

\begin{proof}[Sketch of proof]
	As explained by Brown~\cite{Brown2009}, the existence of an expansion of the form~\eqref{eq:boundary-expansion} follows from the differential equation~\eqref{eq:td-hlog} via a multivariate analogue of Fuchs' theorem.  The statement \eqref{eq:polylog-restrict} about the restrictions (with coefficients in $\QQ$) follows from a corresponding factorization of the differential equation on the boundary strata, and an analysis of the solutions along one-dimensional boundary components.

In \autoref{sec:appendix-mpl}, we outline an alternative proof of \eqref{eq:polylog-restrict} with integer coefficients, using our definition in terms of periods.  The basic idea is as follows:  given a family $\X_t$ of smooth $S$-marked curves that degenerates to a nodal curve $\X_0$ when $t = 0$, we can analyze the asymptotics of periods on $\X_t$ as $t \to 0$ using shuffle regularization.  We find that in the regularized limit, a period of $\X_t$ breaks into a polynomial in periods of the irreducible components of $\X_0$. 
\end{proof}

\subsubsection{Fibrations of moduli spaces and the Poincar\'e lemma}

Evidently if $f\colon \M{T} \to \M{S}$ is the projection defined by an inclusion $S \subset T$ of finite sets, the pullback of functions along $f$ gives an embedding of sheaves $f^*\Vr{S}\subset \Vr{T}$. Using such projections, we can decompose polylogarithms in terms of hyperlogarithms on curves as follows; cf.~\cite[Corollary~6.17]{Brown2009}.

Let us suppose that $T = S \sqcup \sset{z}$ so that we can identify the natural projection $f \colon \M{T} \to \M{S}$ with the universal punctured curve $\uXo{S}$.  Suppose that $x \in \M{T}$ and choose a germ of a section $\gamma_z \in \FGpd{\uX{T}}_x$ consisting of paths in $\uX{T}$ that end at $z \in T$.  Then period integrals of the form $\int_{\gamma_z} v$ where $v \in \T{S}$ are, by definition, hyperlogarithms on the fibres of $\uX{S}$, and taking $\Vr{S}$-linear combinations of such integrals we obtain a map
\begin{equation}
\int_{\gamma_z} \colon \Vr{S,f(x)}\abrac{S}  \to \Vr{T,x}. \label{eq:fibration-basis}
\end{equation}
\begin{proposition}[{\cite[Proposition~6.15]{Brown2009}}]\label{prop:fibration-basis}
	The map \eqref{eq:fibration-basis} is an isomorphism of filtered $\ZZ$-modules.
\end{proposition}
\begin{proof}
This was stated with $\QQ$-coefficients in \emph{op.~cit.} but it holds integrally as follows.  The proposition is obvious in weight zero, so by induction, we can assume that it holds up to weight $n$ and prove it in weight $n+1$.  Given an element $L \in \Vr[n+1]{T}$, its differential $\td L$ along the fibres lies in $\Vr[n]{T}\otimes\A[1]{\uX{S}}$ by \eqref{eq:td-hlog}.  Hence by induction, $\td L$ is a hyperlogarithmic one-form on the fibres.  Applying the Poincar\'e lemma for hyperlogarithms fibrewise using \eqref{eq:hyperlog-primitive}, we construct fibrewise hyperlogarithm $\widetilde L \in \int_\gamma \Vr{S}\abrac{S}$ such that the relative differential $\td(L-\widetilde L)$ vanishes on the fibres, i.e.~$L-\widetilde L$ is fibrewise constant.  It therefore suffices to show that $\Reglim{z}{\st} (L-\widetilde L) \in \Vr{S}$, but this is the special case of \autoref{prop:moduli-log-sing} for the boundary hypersurface $(k=1)$ corresponding to the collision of the pair $A_0 = \{s,z\}$. 
\end{proof}

\begin{example}\label{ex:fib-basis}
	Let $S = \{0,1\}$ and $T = \{0,1,t\}$ and use the cross ratio $z = \crossrat{t}{1}{0}{\infty}$ as an admissible admissible chart on $\uXo{S}\cong\M{T} \cong \PP[1]\setminus\{0,1,\infty\}$. Let $\gamma_z$ denote the straight path from $0$ to $z$, and consider the function $L \in \Vr[1]{T}$, a local section defined for $z\in\CC\setminus [1,\infty)$, given by the period integral
\begin{equation*}
	L 
	\defas \int_{\gamma_z} [\dlog{0}|\dlog{t}|\dlog{1}].
\end{equation*}
As written, it is not immediately obvious that this function is a hyperlogarithm, because the the integrand depends on $z$. But the differential equation \eqref{eq:td-hlog} shows
\begin{equation}
	\td L = (\dlog{0}-\dlog{1})\int_{\gamma_z} [\dlog{0}|\dlog{1}]
	+\dlog{1} \int_{\gamma_z} [\dlog{0}|\dlog{t}],
	\tag{$\ast$}%
	\label{ex-eq:fib-basis-td}%
\end{equation}
where the first period is the hyperlogarithm $\hlog{01}$. The second period does not depend on the puncture labelled $1$ and is thus a function on $\M{\sset{0,t}}$, that is, a constant. Indeed, rescaling by $1/t$ identifies it as the normalized zeta value
\begin{equation*}
	\int_{\gamma_z}[\dlog{0}|\dlog{t}]
	= \int_{\dch}[\dlog{0}|\dlog{1}]
	= \frac{-\mzv{2}}{(2\ipi)^2}
	= \frac{1}{24}
	\in \MZVipi[2]
\end{equation*}
according to \autoref{thm:M03-periods}. We can therefore integrate \eqref{ex-eq:fib-basis-td} to conclude that $L=\hlog{001}-\hlog{101} + \frac{1}{24} \hlog{1} + C$ is indeed a hyperlogarithm. In this case, the constant of integration $C$ is zero, because the hyperlogarithms vanish at zero by definition and therefore
\begin{equation*}
	C
	= \lim_{z\rightarrow 0} L
	= \int_{\gamma_0} [\dlog{0}|\dlog{0}|\dlog{1}]
	= 0
\end{equation*}
is convergent and zero due to \eqref{eq:II-identity}. \qed
\end{example}

The isomorphism of stalks \eqref{eq:fibration-basis} depends on the choice of a path in the fibre.  However, as with hyperlogarithms on a single curve, we obtain a canonical relative weight filtration $\Vr{T}[\bullet]$ whose associated graded is the fibrewise constant sheaf
\[
\gr \Vr{T} = \bigoplus_{k\ge 0} \Vr{T}[k]/\Vr{T}[k-1] \cong f^{-1}\Vr{S} \abrac{S}
\]
of free algebras with coefficients given by polylogarithms on the base.

This process can be iterated.  For an arbitrary projection $f \colon \M{T} \to \M{S}$ we can choose flag of subsets $S=S_0 \subset S_1 \subset \cdots \subset S_N = T$ where $|S_j|=|S|+j$.  We obtain a multi-index filtration $\Vr{T}[\bullet,\cdots,\bullet]$ where the successive superscripts account for the relative weight in the successive fibres.  If we choose a succession of base points and paths in fibres as above, we obtain an identification of the stalk of $\Vr{T}$ at $x \in \M{T}$ with an iterated free algebra:
\begin{equation}
\Vr{T,x} \cong \Vr{S,f(x)}\abrac{S_0}\abrac{S_1}\cdots \abrac{S_{N-1}}
\label{eq:partial-fibration-basis}%
\end{equation}
which descends to a canonical isomorphism on the associated multigraded. In the extremal case $|S|=2$, we obtain a complete description of the stalk:
\[
\Vr{T,x} \cong \MZVipi\abrac{S_0}\abrac{S_1}\cdots \abrac{S_{N-1}},
\]
showing that the only constant polylogarithms are the normalized MZVs.

Applying the Poincar\'e lemma for hyperlogarithms iteratively along the fibres of an arbitrary projection $f \colon \M{T} \to \M{S}$, we see that the complex
\begin{align}
\VAr{f} \defas \Vr{T} \otimes \A{\M{T}/\M{S}} \subset \sforms{\M{S}/\M{T}} \label{eq:relative-forms-MS}
\end{align}
of relative polylogarithmic forms is exact in positive degrees (c.f.\ {\cite[Theorem~9.1]{Brown2009}} or \cite[Theorem~1]{Bogner2015}):

\begin{proposition}\label{prop:moduli-poincare}
For an arbitrary projection $f\colon \M{T} \to \M{S}$, the inclusion
\[
f^{-1}\Vr{S} \hookrightarrow \VAr{f}
\]
of the fibrewise constant sheaf is a quasi-isomorphism of filtered complexes.
\end{proposition}
\begin{proof}
We proceed by induction on the number $n$ of marked points forgotten by $f$. When $n=1$, the claim follows from \autoref{prop:fibration-basis} by applying the Poincar\'e lemma for hyperlogarithms (\autoref{lem:hlog-Poincare}) along the fibres of the universal curve $\uX{S}$.  For $n > 1$, we can factor the projection as 
\[
\xymatrix{
\M{T} \ar[r]^-{g} & \M{T'} \ar[r]^{f'} & \M{S}
}
\]
where $|T| = |T'| + 1$.  Consider the decreasing filtration by subcomplexes $F^j \VAr{f} = g^*\VAr[j]{T'} \wedge \VAr{f}$. By \eqref{eq:arnold-tensor}, the associated graded sheaf is the tensor product $\gr \VAr{f} \cong   g^{-1}\VAr{f'}\otimes_{g^{-1}\Vr{T'}}\VAr{g}$ which, by induction, is a resolution of the fibrewise constant sheaf $g^{-1}{f'^{-1}\Vr{S}}  \otimes_{g^{-1}\Vr{T'}} g^{-1}{\Vr{T'}} \cong f^{-1}\Vr{S}$.  Hence the result follows by considering the corresponding spectral sequence.
\end{proof}

\begin{corollary}
The polylogarithmic de Rham complex $\VAr{S} \subset \sforms{\M{S}}$ is a resolution of the constant sheaf $\underline{\MZVipi}$, compatible with the weight filtration.
\end{corollary}

%%%%%%%%%%%%%%%%%%%%%%

\section{Cohomology of polylogarithm sheaves}
\label{sec:monodromy}
We now turn to the following question: suppose that
\[
	\cyc \subset \M{S}
\]
is a topological subspace.  Which polylogarithms on $\M{S}$ become single-valued when restricted to $\cyc$?
In sheaf-theoretic terms, we would like to understand the weight-filtered space of global sections
\[
	\sect{\cyc,\Vr{S}|_\cyc} = \cohlgy[0]{\cyc,\Vr{S}|_\cyc}.
\]
More generally, we may ask about the higher sheaf cohomology groups, and for instance seek criteria for their vanishing.

For a general subspace $\cyc$, the problem seems difficult, but it is quite tractable in the case where $\cyc$ has the structure of an iterated fibration by punctured disks in universal curves.  Such a situation can be reduced inductively to the one-dimensional case of hyperlogarithms, so we begin with a discussion in the case of a single curve.

\subsection{Cohomology of hyperlogarithms}
\label{sec:hyperlog-cohlgy}

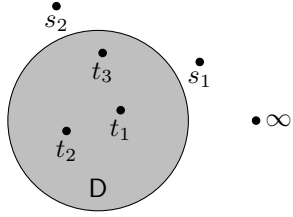
\begin{figure}
	\centering
\begin{tikzpicture}[scale=0.6]
\draw[fill=lightgray] (0,0) circle (2);
\drawvertex[below]{t_1}{0.5,0.23}
\drawvertex[below]{t_2}{-0.7,-0.23}
\drawvertex[below]{t_3}{0.1,1.5}
\drawvertex[below]{s_1}{30:2.6}
\drawvertex[below]{s_2}{110:2.7}
\drawsilentvertex{inf}{0:3.5}
\draw (0:4) node {$\infty$};
\draw (-90:1.5) node {$\D$};
\end{tikzpicture}%
\caption{A disk enclosing marked points $T = \sset{t_1,t_2,t_3}$ so that $S\setminus T = \sset{s_1,s_2}$.}\label{fig:disk-T}%
\end{figure} 

Suppose we are given an $S$-marked curve $\X$ and a subset $T \subset S$.  Let $\D \subset \X$  be a contractible open set such that $\pInf{S} \cap \D = T$ as illustrated in \autoref{fig:disk-T}, and let $\Do \defas \D \setminus T \subset \Xo$ be the corresponding open set.  Let $\HL{}$ be the sheaf of hyperlogarithms on $\Xo$ from \autoref{sec:hlogs}.

In this section we explain how to compute the sheaf cohomology $\cohlgy{\Do,\HL{}}$ as a module over the ring $\Per{}$ of periods of $\X$. We begin with the following observation:
\begin{lemma}\label{lem:H0-hol}
Any global section $L \in \cohlgy[0]{\Do,\HL{}}$ extends to a holomorphic function on the whole disk $\D$.
\end{lemma}

\begin{proof}
This follows easily from the expansion \eqref{eq:hlog-zero-expansion} of $L$ in a neighbourhood of a marked point  $t\in T \subset \D$, and the fact that $L$ has no monodromy around $t$.
\end{proof}

To describe the cohomology algebraically, let us choose a tangential base point $\pt$ and a path $\eta$  from $\pt$ to a point $q \in \Do$ gives an identification of the stalk $\HL{q}$ with the free algebra $\PerT{S}$.  Taking germs of global sections, we obtain an embedding
\begin{equation}
\phi_\eta \colon \cohlgy[0]{\Do,\HL{}} \hookrightarrow \PerT{S}. \label{eq:H^0(L)}
\end{equation}
that is inverse to the iterated integral $\int_\eta$.  In light of \autoref{lem:H0-hol}, this embedding is independent of the endpoint $q \in \Do$, and in fact depends only on the homotopy class of $\eta$ as a path from $\pt$ to the contractible set $\D$ in the punctured curve $\Xo \cup \D = \X \setminus (\pInf{S}\setminus T)$.  Note that if $p \in T$, there is a preferred choice of such a homotopy class, given by the trivial path.  

To get a feel for this embedding, it is helpful to consider some examples:
\begin{example}\label{ex:sv-log}
Take an admissible coordinate $z$ and  an element $s \in S \setminus T$. Since the only branch point of the corresponding hyperlogarithm $\hlog{s}(z)=\log\frac{z-s}{p-s}$ is at the point $z=s \notin \D$, any branch of $\hlog{s}$ is single-valued  on $\D$, giving an element of $\cohlgy[0]{\Do,\HL{}}$.  Thus $s \in \phi_\eta(\cohlgy[0]{\Do,\HL{}}) \subset \PerT{S}$. \qed
\end{example}

\begin{example}
In contrast, suppose $t \in T$, and consider a hyperlogarithm of the form $\hlog{ut}$ for some element $u \in \PerT{S}$ of total weight $n$.  By the path concatenation formula \eqref{eq:hyperlog-monodromy}, the monodromy of $\hlog{ut}$ around a simple loop at $t$ is non-trivial:
\[
	\hlog{ut} \mapsto \hlog{ut} + \hlog{u} \mod \HL[n-1]{}.
\]
Hence $\hlog{ut}$ is not single-valued in $\Do$, so that $ut \notin \phi_\eta(\cohlgy[0]{\Do,\HL{}})$. \qed
\end{example}
More generally, every non-zero element of the left ideal $\PerT{S}T \subset \PerT{S}$ in the free algebra corresponds to a hyperlogarithm with nontrivial monodromy in $\Do$.  (Note that the left ideal is taken with respect to the standard noncommutative product on $\PerT{S}$, not the commutative shuffle product.)  In a certain sense, this is the only way that a hyperlogarithm can fail to be single-valued in $\Do$.  More precisely, let us write
\[
	\PerT{S|T} \defas \PerT{S}/\PerT{S}T
\]
for the quotient $\Per{}$-module. We will prove the following:

\begin{theorem}\label{thm:cohomology-hyperlogs}
The natural composition
\begin{align}
\xymatrix{
\cohlgy[0]{\Do,\HL{}}\,\, \ar@{^(->}[r]^-{\phi_\eta} & \PerT{S} \ar@{->>}[r] & \PerT{S|T}
} \label{eq:sv-iso}
\end{align}
is an isomorphism of bifiltered $\Per{}$-modules, and the induced isomorphism
\[
\gr W^\bullet\cohlgy[0]{\Do,\HL{}} \cong \PerT{S|T}
\]
is independent of the path $\eta$.  Moreover, we have the vanishing
\[
\cohlgy[j]{\Do,\HL{}} = 0
\]
for all $j > 0$.  In fact, the complex computing $\cohlgy{\Do,\HL{}}$ admits a contracting homotopy of total weight one.
\end{theorem}

\begin{remark}
Note that while $\cohlgy[0]{\Do,\HL{}}$ is a commutative algebra over $\Per{}$, the theorem only describes its $\Per{}$-module structure; not the multiplication.\qed
\end{remark}

Applying the theorem to the case $\D = \X\setminus\{\infty\}$, so that $S = T$, we obtain the following immediate corollary:
\begin{corollary}\label{cor:cohlgy-whole-curve}
For any marked curve $\X$, we have
\[
\cohlgy[j]{\Xo,\HL{}} = \begin{cases}
\Per{} & j = 0 \\
0 & j > 0.
\end{cases}
\]
\end{corollary}

The rest of this section is devoted to the proof of \autoref{thm:cohomology-hyperlogs}.   The basic idea is as follows: note that the projection $\PerT{S} \to \PerT{S|T}$ has a natural splitting, whose image is the submodule of $\PerT{S}$ generated by  words that do not end with an element of $T$.  To a first approximation, the inverse of the composition \eqref{eq:sv-iso} is defined by taking the hyperlogarithms associated to such elements.   While such hyperlogarithms may still have nontrivial monodromy in $\Do$, we will see that the monodromy can be cancelled inductively through the addition of terms that lie in the ideal $\PerT{S}T$. 

To begin the proof, let us observe that $\Do$ is homotopy equivalent to a bouquet of circles. Hence the universal cover of $\Do$ is contractible, i.e.~$\Do$ is a $K(\pi,1)$ space, so that the cohomology of the local system $\HL{}$ may be computed as the group cohomology of the fundamental group of $\Do$ with coefficients in the monodromy representation on any stalk $\HL{q}\cong \PerT{S}$.

For each $t \in T$ choose a loop $\gamma_t \in \fg{\Do,\vec{p}}$ so that the elements  $\{\gamma_t\}_{t\in T}$ give a collection of free generators for the fundamental group that are dual to the corresponding differential forms:
\begin{equation}
	\int_{\gamma_t} \dlog{s}
	=\begin{cases}
		0 & \text{if $t\neq s$ and} \\
		1 & \text{if $t=s$.} \\
	\end{cases}
	\label{eq:duality-subset}%
\end{equation}
We may then compute the group cohomology using the standard two-term complex for the cohomology of a free group:
\[
	\chn{\HL{}} \defas \rbrac{\xymatrix{
	\PerT{S} \ar[r]^-{\delta} & {\PerT{S}}^{\oplus T}} }
\]
where the differential $\delta$ computes the monodromy around the generators:
\[
	\delta \hlog{u} = (\gamma_t \cdot \hlog{u} - \hlog{u})_{t \in T}.
\]
We will construct an explicit contracting homotopy for the complex $\chn{\HL{}}$.

To this end, let us decompose the differential $\delta$ into homogeneous components
\[
	\delta = \delta_1 + \delta_2+\cdots 
\]
where
\[
\mapdef{\delta_j}{\PerT{S}}{{\PerT{S}}^{\oplus T}}
{s_1\cdots s_n}{\left( s_1\cdots s_{n-j} \int_{\gamma_t} s_{n-j+1}\cdots s_n \right)_{t \in T}}
\]
gives the component in the path concatenation formula \eqref{eq:hyperlog-monodromy} that drops the relative weight of a hyperlogarithm by $j$.  We will view $\delta$ as a perturbation of the simpler differential $\delta_1$, which is easier to understand:
\begin{lemma}
The projection $\chn[0]{\HL{}} = \PerT{S}  \to \PerT{S|T}$ and its natural splitting $\PerT{S|T} \to \PerT{S}$ extend to a homotopy equivalence
\[
	\rbrac{\chn{\HL{}},\delta_1} \cong \PerT{S|T} 
\]
where $\PerT{S|T}$ is viewed a complex concentrated in degree zero.
\end{lemma}

\begin{proof}
From the path concatenation formula \eqref{eq:path-concat} or \eqref{eq:hyperlog-monodromy}, we see that
\[
	\delta_1s_1\cdots s_n = \rbrac{ s_1\cdots s_{n-1} \int_{\gamma_t} \dlog{s_n}}_{t \in T}
\]
for any word $s_1\cdots s_n \in \PerT{S}$.  Consider the cochain homotopy operator  $h \colon \chn{\HL{}}\to\chn[\bullet-1]{\HL{}}$ defined by the formula
\[
	\mapdef{h}{{\PerT{S}}^{\oplus T}}{\PerT{S}}
	{(u_t)_{t \in T}}{\sum_{t \in T} u_t t.}
\]
Then clearly the image of $h$ is the ideal $\PerT{S}T \subset \PerT{S}$.  Moreover, the identity $\delta_1 h = 1$ holds because of the duality \eqref{eq:duality-subset}. Hence $h$ is our desired contracting homotopy, and the result follows.
\end{proof}

Note that the contracting homotopy $h : (\chn{\HL{}},\delta_1)\to(\chn[\bullet-1]{\HL{}},\delta_1)$ constructed in the proof has biweight $(1,1)$ with respect to the natural filtrations on hyperlogarithms, i.e.
\[
h\chn[1]{\HL[n]{}[j]} \subset \chn[0]{\HL[n+1]{}[j+1]}.
\]
We will use this operator to construct a contracting homotopy for the complex $(\chn{\HL{}},\delta)$ itself.  (This is a simple instance of the perturbation lemma of homological algebra~\cite{Brown1965,Crainic2004a,Gugenheim1972}.)

Let $\delta' \defas \delta - \delta_1$ be the difference of the two differentials; it is a sum of terms that decrease the relative weight by two or more.  By the unipotence of monodromy (\autoref{lem:mon-unipotent}), we have
\[
\delta'\chn[0]{\HL[n]{}[j]} \subset \chn[1]{\HL[n-1]{}[j-2]}
\]
and therefore
\[
h\delta'  \chn{\HL[n]{}[j]} \subset \chn{\HL[n]{}[j-1]}.
\]
Thus, the endomorphism $h\delta'$ is nilpotent with respect to the relative weight filtration on hyperlogarithms, so that the following operator is well-defined and invertible:
\[
	\Psi \defas (1+h\delta')^{-1} = 1 - h\delta' +  (h\delta')^2 - \cdots 
\]
Moreover $\Psi$ preserves both weight filtrations.  An elementary calculation shows that $\Psi$ gives a commutative diagram of bifiltered complexes
\[
\xymatrix{
(\chn{\HL{}},\delta_1) \ar[rr]^{\Psi}\ar[rd] && (\chn{\HL{}},\delta) \ar[ld] \\
& \PerT{S|T}
}
\]
which establishes \autoref{thm:cohomology-hyperlogs}.

\begin{remark}
	For an element $u \in \PerT{S|T}$, let us denote by $L_{u|T} \in \cohlgy[0]{\Do,\HL{}}$ the corresponding global section.  Tracing through the proof, one finds that for any $s \in S$, we have the identity
	\[
	\td L_{su|T} = \begin{cases}
	L_{u|T} \cdot \dlog{s} & \text{if $s\in S\setminus T$ and} \\
	 (L_{u|T}-L_{u|T}(s)) \cdot \dlog{s} & \text{if $s \in T$.} \\
	\end{cases}
	\]
This identity can be used recursively to give an explicit description of the global section associated to any element of $\PerT{S|T}$ in terms of iterated integrals along the path $\eta$. \qed
\end{remark}

\begin{remark}\label{rem:monodromy-cancel}
Given any one-cochain $\lambda = (\lambda_t)_{t \in T} \in {\PerT{S}}^{\oplus T}$, let
\begin{align}
\Lambda \defas (1+h\delta')^{-1}h\lambda  = (h-h\delta'h+h\delta'h\delta'h- \cdots)\lambda \in \PerT{S}. \label{eq:monodromy-cancellation}
\end{align}
Then $\lambda = \delta\Lambda$, i.e.~the monodromy of the hyperlogarithm $\hlog{\Lambda}$ around each  marked point $t \in T$ is equal to $L_{\lambda_t}$.\qed
\end{remark}

\begin{example}\label{ex:monodromy-coboundary}
Given $p \in T$ and $q \in S \setminus T$, consider the one-cocycle $(\lambda_t)_{t \in T}$ defined by
\[
\lambda_t = \begin{cases}
q & t = p \\
0 & t \ne p.
\end{cases}
\]
Let us compute a coboundary $\Lambda$ for $\lambda$ using \eqref{eq:monodromy-cancellation}.  We have
\[
h \lambda = \sum_{t \in T}\lambda_t t = qp,
\]
and therefore
\[
h\delta'h \lambda = h \rbrac{\int_{\gamma_t} qp}_{t \in T} = \sum_{t \in T}\rbrac{\int_{\gamma_t}qp} \cdot t = -\tfrac{\log(p-q)}{2\ipi}\cdot  p \in \PerT{S}
\]
by \autoref{ex:wt2-loop}, for an appropriate branch of the logarithm.  Meanwhile  $(h\delta')^jh \lambda = 0$ for $j > 1$ by weight considerations and hence
\[
\Lambda = h\lambda - h\delta'h\lambda =  qp + \tfrac{\log(p-q)}{2\ipi}p
\]
Hence the hyperlogarithm $L_{qp}+\tfrac{\log(p-q)}{2\ipi}L_p$ has monodromy $L_q$ around $p$ and trivial monodromy around the other points in $T$.\qed
\end{example}

\subsection{Families of disks}\label{sec:families-cycles}

We now consider a version of the previous construction ``in families''.  Suppose given a commutative diagram
\begin{equation}\vcenter{%
\xymatrix{
\cyc' \ar[r]\ar[d]_{f} & \uXo{S}\ar[d] \\
\cyc \ar[r] & \M{S}
}}
\label{eq:family-of-cycles}%
\end{equation}
where $f \colon \cyc' \to \cyc$ is a locally trivial fibration of topological spaces, such that the induced map on each fibre is homotopy equivalent to the inclusion of an open punctured disk $\Do$ encircling a collection $T \subset S$ of marked points as in the previous section. 

Let us denote by $\Vr{\cyc'}$ and $\Vr{\cyc}$ the pullbacks of the locally constant sheaves of polylogarithms on $\uXo{S}$ and $\M{S}$ to the spaces $\cyc'$ and $\cyc$, respectively.  Choosing a suitable homotopy classes of paths in the fibres, we may use \autoref{prop:fibration-basis} to express $\Vr{\cyc'}$ in terms of fibrewise hyperlogarithms.  Arguing exactly as in \autoref{thm:cohomology-hyperlogs}, we may compute the fibrewise cohomology of the sheaf $\Vr{\cyc'}$ (i.e.~the derived direct image $Rf_*\Vr{\cyc'}$), giving the following:
\begin{lemma}\label{lem:family-cycles-push}
There is a canonical isomorphism
\begin{equation*}
	\gr f_* \Vr{\cyc'} \cong \Vr{\cyc}\abrac{S|T} 
\end{equation*}
and all higher direct images vanish: $R^jf_*\Vr{\cyc'} = 0$ for $j > 0$.
\end{lemma}

Suppose now that we are given a tower of diagrams \eqref{eq:family-of-cycles} as above:
\begin{equation}\vcenter{%
\xymatrix{
\cyc_{n}\ar[d]_-{f_n} \ar@/_3pc/[ddd]_-{f} \ar[r] & \uXo{S_{n-1}} \cong \M{S_{n}}\ar[d] \\
\cyc_{n-1} \ar[d]_-{f_{n-1}} \ar[r] & \uXo{S_{n-2}} \cong \M{S_{n-1}}  \ar[d] \\
\vdots \ar[d]_-{f_1} &\vdots \ar[d] \\
\cyc_{0}  \ar[r] & \M{S_{0}} 
}}\label{eq:fibration-tower}%
\end{equation}

We then have the following:
\begin{theorem}\label{thm:families-push}
For the iterated relative weight filtration $W^{\bullet,\cdots,\bullet}\Vr{\cyc_n}$ obtained from the tower \eqref{eq:fibration-tower}, we have
\[
\gr W^{\bullet,\cdots,\bullet}f_*\Vr{\cyc_n} \cong \Vr{\cyc_0}\abrac{S_0|T_0} \abrac{S_1|T_1} \cdots \abrac{S_{n-1}|T_{n-1}},
\]
and all of the higher direct images vanish: $R^jf_*\Vr{\cyc_n} = 0$ for $j > 0$.
\end{theorem}

\begin{proof}
Combining \autoref{lem:family-cycles-push} with a standard fact about filtered complexes (\autoref{lem:cohlgy-inc-filt} below), we see that for the sheaves in question, the associated graded construction commutes with the derived pushforward $Rf_{j*}$.  Hence the result follows easily by induction on $n$.
\end{proof}

\begin{lemma}\label{lem:cohlgy-inc-filt}
Suppose $\chn{}$ is a cochain complex with an exhaustive increasing filtration $0=W^{-1}\chn{} \subset W^0\chn{} \subset W^1\chn{} \subset \cdots$ such that $\cohlgy[j]{\gr \chn{}} = 0$ for $j \ne 0$.  Then $\cohlgy[j]{\chn{}} = 0$ for $j \ne 0$, and $\gr \cohlgy[0]{\chn{}} = \cohlgy[0]{\gr \chn{}}$.
\end{lemma}

\begin{proof}
Since $\chn{} = \bigcup_{j \ge 0} W^j\chn{}$, we have $\cohlgy{\chn{}} = \colim{j}\cohlgy{W^j\chn{}}$, so the problem reduces to studying the individual filtered pieces.  These can be dealt inductively using the long exact sequences associated to the successive graded quotients $\gr[j]\chn{} = W^j\chn{}/W^{j-1}\chn{}$.
\end{proof}

\begin{example}
    We may apply the results above to the case in which $\cyc_j = \M{S_j}$, so that $S_j = T_j$ for all $j$. (This is a relative version of \autoref{cor:cohlgy-whole-curve}.) We find that $Rf_*\Vr{S_n} =  \Vr{S_0}$
    in the derived category of sheaves on $\M{S_0}$.    Projecting all the way down to a point, we obtain the following computation of the cohomology of the sheaf of polylogarithms on the moduli space $\M{S}$:
	\[
	\cohlgy[j]{\M{S},\Vr{S}}
	= \begin{cases} 
		\MZVipi & \text{if $j=0$ and}\\
		0       & \text{if $j > 0$.}
	\end{cases}
	\]
In particular, the only polylogarithms on $\M{S}$ that are globally single-valued are the constants. \qed
\end{example}

\begin{example}\label{ex:MS-sv}
Consider the subsheaf $\Vr{S}\overline{\Vr{S}} \cong \Vr{S}\otimes_{\MZVipi}\overline{\Vr{S}} \subset \sCinf{\M{S}}$ generated by polylogarithms and their complex conjugates.  Choosing a full flag of subsets $S_2\subset S_3 \subset \cdots \subset S_N = S$ with $|S_j| = j$, we may use the relative weight filtrations on the antiholomorphic factors to obtain a multi-index filtration on $\Vr{S}\overline{\Vr{S}}$ whose associated graded is
\[
\gr W^{\bullet,\cdots,\bullet} (\Vr{S}\overline{\Vr{S}}) = \Vr{S}\abrac{S_2}\abrac{S_3}\cdots \abrac{S_N}.
\]
Combining the result of the previous example with \autoref{lem:cohlgy-inc-filt}, we obtain an isomorphism
	\[
	\gr W^{\bullet,\cdots,\bullet}\cohlgy[j]{\M{S},\Vr{S}\overline{\Vr{S}}} = \begin{cases} \MZVipi\abrac{S_2}\cdots\abrac{S_N} & j=0\\
	0 & j > 0 .
	\end{cases}
	\]
	The elements of $\cohlgy[0]{\M{S},\Vr{S}\overline{\Vr{S}}}$ are called \defn{single-valued polylogarithms}; see, e.g.~\cite{Brown2004,Brown2004a} for the one-dimensional case.  Note that the degrees of freedom for the global sections match the degrees of freedom for the stalks of $\Vr{S}$. Hence for any germ of a holomorphic polylogarithm, we can cancel its monodromy by adding sections of $\Vr{S}\overline{\Vr{S}}$ whose holomorphic parts have lower multi-weight, resulting in a globally single-valued function. An explicit construction is discussed in \autoref{rem:sv-int}.\qed
\end{example}

%%%%%%%%%%%%%%%%%%%%%%

\section{Integration on moduli spaces of disks}
\label{sec:cycles}

\subsection{Moduli spaces of marked disks}
\label{sec:disk-moduli-spaces}

Let $P$ be a finite set, and let $Q$ be a totally ordered finite set. We assume throughout that $2|P|+|Q| \ge 2$.  A \defn{$(P,Q)$-marked disk} consists of the following data:
\begin{itemize}
\item a compact complex analytic disk $\D$ with boundary $\partial \D$,
\item a special marked point $\infty \in \partial \D$,
\item an embedding of $P$ in the interior $\D \setminus \partial \D$,
\item an embedding of $Q$ in $\partial \D \setminus \{\infty\}$, such that the total order on $Q$ is compatible with the natural orientation of the boundary.
\end{itemize}

Any $(P,Q)$-marked disk $\D$ is evidently isomorphic to the closed upper half-plane $\bHH \subset \PP[1]$ in such a way that the special marked point $\infty \in \partial \D$ is sent to the point at infinity $\infty \in \bHH$. Thus the moduli space $\C{P,Q}$ of $(P,Q)$-marked disks is a connected real analytic manifold of dimension $2\abs{P}+\abs{Q}-2$ that is canonically isomorphic to the quotient
\[
	\C{P,Q} \cong \rbrac{\Conf_{P}(\HH) \times \Conf_{Q,+}(\RR)}/G,
\]
where
%of the configuration space of $\HH$ times the ordered configuration space of $\RR$,
\begin{align*}
	\Conf_P(\HH) 
	&\defas \set{ (x_p)_{p \in P} }{x_{p} \ne x_{p'} \textrm{ if }p \ne p'} \subset \HH^P
	\quad\text{and}
\\
	\Conf_{Q,+}(\RR)
	&\defas \set{(x_q)_{q\in Q}}{x_{q} < x_{q'} \textrm{ if }q < q'} \subset \RR^Q
\end{align*}
is the configuration space of $\HH$ and the ordered configuration space of $\RR$, respectively, and where
%by the group $G$ of conformal transformations of $\bHH$ that preserve $\infty$. As
\[
G \cong \RR_{>0}\ltimes \RR,
\]
is the group of conformal transformations of $\bHH$ that preserve $\infty$.
It acts by $(a,b) \cdot y = ay+b$ on $y \in \bHH$, and diagonally on the configuration spaces. Note that $\C{P,Q}$ has a canonical orientation, which descends from the natural orientation of the configuration spaces.

The moduli space $\C{P,Q}$ comes equipped with a universal marked disk
\[
\uD{P,Q} \to \C{P,Q},
\]
whose total space can be realized concretely as the quotient of the product $\Conf_P(\HH) \times \Conf_{Q,+}(\RR) \times \bHH$ by the diagonal action of $G$.  The marked points in the fibres are indicated by the sections $\C{P,Q} \to \uD{P,Q}$ obtained from the natural evaluation maps $P \times \Conf_P(\HH) \to \HH$ and $Q \times \Conf_{Q,+}(\RR) \to \partial \bHH$.  We denote by
\[
\partial\uD{P,Q} \subset \uD{P,Q}
\]
the boundary, induced by the boundary $\RR \cup \sset{\infty}$ of $\bHH \subset \PP[1]$, and we denote by
\[
\uDo{P,Q} \subset \uD{P,Q}
\]
the universal punctured disk, obtained by removing the boundary $\partial\uD{P,Q}$ and all interior marked points in the fibres.  Note that the total space of $\uDo{P,Q}$ is canonically identified with $\C{P\sqcup\{*\},Q}$, where $*$ is an extra  label.

A marked disk can be ``doubled'' by gluing it to its complex conjugate along its boundary, giving an $S$-marked curve, where 
\[
	S = P \sqcup \Pb \sqcup Q.
\]
In this way we obtain canonical real analytic maps
\begin{equation}
	\iota       \colon  \C{P,Q}   \injection     \M{S} \qquad\text{and}\qquad
	\tilde\iota \colon \uD{P,Q}  \hookrightarrow \uX{S} 
	\label{eq:moduli-embedding}%
\end{equation}
that embed $\C{P,Q}$ as a connected component of the fixed point set of an antiholomorphic involution on $\M{S}$, so that $\M{S}$ can be viewed as a complexification of $\C{P,Q}$; see~\cite{Ceyhan2007a,Ceyhan2007}.  The map $\tilde \iota$ then embeds the fibres of $\uD{P,Q}\cong \C{P,Q} \times \bHH$ as closed disks in the universal curve $\uX{S}\cong \M{S} \times \PP[1]$.

When $|Q| \ge 2$, one can construct a canonical pair of complex conjugate coordinates on the fibres.  Indeed, there is a preferred embedding $\{0,1\} \hookrightarrow Q$: the one that sends that sends 0 and 1 to the minimal and maximal elements of $Q$, respectively.  In this way, we trivialize the universal curve over $\C{P,Q}$, and the cross ratios $z \defas \crossrat{\ast}{1}{0}{\infty}$ and $\zb \defas \crossrat{\bar{\ast}}{1}{0}{\infty}$ give the desired coordinates on the disk fibres. Note that this identifies the boundary of the fibres with the real line $\RR \cup \sset{\infty}$.

When $|Q| \le 1$ we must have $|P| \geq 1$ and can similarly construct coordinates by sending an interior marked point to $\iu \in \HH$, but those coordinates are not cross ratios, so a different convention for tangential base points is required in that case.  We avoid treating this case since it turns out to be unnecessary for our purposes.

%
%In \cite{Kontsevich2003}, Kontsevich compactified  $\C{P,Q}$ to a real analytic manifold with  corners $\bC{P,Q}$, using iterated blowups in the style of Fulton--Macpherson~\cite{Fulton1994} (but using real oriented blowups as in~\cite{Axelrod1994}).  The strata of the boundary $\partial \bC{P,Q}$  parametrize marked disks with extra sphere and disk bubbles attached, with some extra tangent decorations at certain marked points~\cite{Devadoss2012,Hoefel2009,Voronov1999}.  We warn the reader that while the inclusion map \eqref{eq:moduli-embedding} extends to a map of pairs
%\begin{equation*}
%%	(\bC{P,Q},\bbC{P,Q}) \longrightarrow (\bM{S},\bbM{S}),
%\end{equation*}
%this extension is neither injective nor immersive over $\bbC{P,Q}$, so that $\bC{P,Q}$ differs from the closure of $\C{P,Q}$ in $\bM{S}$.  This failure of immersivity can be corrected by making suitable blowups of $\bbM{S}$ along boundary strata.

%It is precisely this difference which makes the integration theory on $\C{P,Q}$ more interesting. In particular, the integrals we will consider in sections~\ref{sec:sv-integration} and \ref{sec:formality} are not mapped to relative periods of a pair $(\bM{S}\!\setminus\! A,\bbM{S}\!\setminus\!B)$ with divisors $A,B \subset \bbM{S}$ of the type studied in \cite{Brown2009}.\erik{Mention this in introduction?}

\subsection{Polylogarithms on moduli spaces of disks}
\label{sec:disk-moduli-cohlgy}

The smooth manifold $\C{P,Q}$ has a natural real analytic structure, and we will say that a complex-valued function on $\C{P,Q}$ is \defn{analytic} if its real and imaginary parts are real analytic.  We denote by $\sO{\C{P,Q}} \subset \sCinf{\C{P,Q}}$ the sheaf of analytic functions on $\C{P,Q}$ and by $\sforms{\C{P,Q}}$ the locally free $\sO{\C{P,Q}}$-module of analytic differential forms.  Equivalently, since $\M{S}$ is a complexification of $\C{P,Q}$, we can identify $\sforms{\C{P,Q}}$ with the subsheaf of the $C^\infty$ differential forms generated by restrictions of holomorphic forms.  Note that the restriction is an isomorphism on stalks, so that $\iota^*\sforms{\M{S}}$ is isomorphic to the abstract inverse image sheaf $\iota^{-1}\sforms{\M{S}}$. 

We adopt the following notations for the sheaves of polylogarithms and associated forms:
\begin{align*}
\A{\C{P,Q}} &\defas \iota^*\A{\M{S}} \subset \forms{\C{P,Q}},  \\
\Vr{P,Q} &\defas \iota^* \Vr{S} \subset \sO{\C{P,Q}}\quad\text{and} \\
\VAr{P,Q} &\defas \iota^* \VAr{S}\subset \sforms{\C{P,Q}}.
\end{align*}
Evidently the sheaves $\Vr{P,Q} \cong \iota^{-1}\Vr{S}$ and $\VAr{P,Q} \cong \iota^{-1}\VAr{S}$ are locally constant and filtered by the total weight. Similarly, for a projection $f \colon \C{P',Q'} \to \C{P,Q}$ defined by forgetting a collection of marked points, we have a projection $\widetilde f \colon \M{S'} \to \M{S}$ and we define the sheaves of relative differentials
\[
\VAr{f} \defas \iota^*\VAr{\widetilde f} \subset \sforms{\C{P',Q'}/\C{P,Q}}
\]
by pulling back the relative forms for $\widetilde f$ as defined in \eqref{eq:relative-forms-MS}.  We also introduce the following notations for the global sections:
\[
\V{\C{P,Q}} \defas \sect{\C{P,Q},\Vr{P,Q}}
\qquad\text{and}\qquad
\VA{\C{P,Q}} \defas \sect{\C{P,Q},\VAr{P,Q}}.
\]
Our main result is then the following:
\begin{theorem}\label{thm:disk-moduli-pushforward}
Let $f\colon \C{P',Q'} \to \C{P,Q}$ be a projection defined by forgetting $n = |P'|-|P|$ interior marked points and $m = |Q'|-|Q|$ boundary marked points.  Then the following statements hold:
\begin{enumerate} 
\item \label{thm:disk-moduli-pushforward:acyc} The sheaves $\Vr{P',Q'}$ are acyclic along the fibres, i.e.
\[
	R^jf_*\Vr{P',Q'} =0
\]
for all $j > 0$.
\item \label{thm:disk-moduli-pushforward:deRham}  The relative polylogarithmic forms compute the singular cohomology of the fibres.  More precisely, there are canonical quasi-isomorphisms
\[
	f_*(\VAr{f},\td) \cong Rf_*(\VAr{f},\td)  \cong  \Vr{P,Q} \otimes Rf_* \underline{\ZZ}
\]
of complexes of sheaves on $\C{P,Q}$.
\item \label{thm:disk-moduli-pushforward:integral} If $\omega \in \VA[d]{\C{P',Q'}}[j]$ is such that the fibre integrals defining the pushforward $f_*\omega$ converge absolutely, then
\[
	f_*\omega \in \VA[d-2n-m]{\C{P,Q}}[j-k]
\]
is a polylogarithmic form, but the weight has dropped at least by
\[
k \defas  \begin{cases}
	n+1  & \text{if $|Q|=0$ and $|Q'| > 0$, or} \\
	n & \text{otherwise.}
	\end{cases}
\]
\end{enumerate}
\end{theorem}

Note that if $|P| = 1$ and $|Q|=0$, or if $|P|=0$ and $|Q|=2$, then the base $\C{P,Q}$ is a point.  Specializing the theorem to such projection, we obtain 
\begin{corollary}\label{cor:arnold-integral}
For any $P$ and $Q$, the following statements hold:
\begin{enumerate}
\item The sheaves $\Vr{P,Q}$ are acyclic.
\item The complex $(\VA{\C{P,Q}},\td)$  computes the singular cohomology of $\C{P,Q}$ with coefficients in the filtered ring $\MZVipi$.
\item Suppose that $\omega \in \A[\dim \C{P,Q}]{\C{P,Q}}$ is a volume form whose integral over $\C{P,Q}$ converges absolutely. Then
\[
\int_{\C{P,Q}} \omega \in \begin{cases}
\MZVipi[|P|-1] & \text{if $Q=\emptyset$ and} \\
\MZVipi[|P|+|Q|-2] & \text{otherwise}. \\
\end{cases}
\]
\end{enumerate}
\end{corollary}

\begin{proof}
The first to statements are immediate.  For the third, note that by definition, the weight of $\omega$ is equal to $\dim \C{P,Q} = 2|P|+|Q|-2$.  Now apply \autoref{thm:disk-moduli-pushforward:integral} of \autoref{thm:disk-moduli-pushforward} for a suitable projection to a point.  If $|Q|\ge 2$, use a projection that forgets all elements of $P$ and all but two elements of $Q$, so the weight drops by $|P|$.   If $|Q| = 1$, use a projection that forgets the unique element of $Q$ and all but one of the elements of $P$; again the weight drops by $1 + (|P|-1) = |P|$.  Finally if $|Q|=0$, forget all but one of the elements of $P$; the weight drops by $|P|-1$.
\end{proof}
%
%If $|Q| \ge 2$ and we choose $q \ne q' \in Q$, then we have a projection $f \colon \C{P,Q} \to \C{\emptyset,\{q,q'\}} \cong \{*\}$ which forgets $m= |P|$ interior marked points and $n=|Q|-2$ boundary points; hence by the theorem, $f_*\omega \in \MZVipi$ has weight $d-m = |P|+|Q|-2$. 
%
%On the other hand, if $Q < 2$, then we must have $|P| \ge 1$ and hence a choice of point $p \in P$ gives a projection to $\C{\{p\},\emptyset} \cong \{*\}$ which forgets $m=|P|-1$ interior marked points, so that $f_*\omega$ has weight $d-m=|P|-1$ if $Q=\emptyset$ and weight $d-m-1=|P|-1$ when $|Q|=1$.
%\end{proof}

Let us outline the proof of \autoref{thm:disk-moduli-pushforward}.  First, observe that every projection can be factored as a sequence of elementary projections in which only a single boundary or interior marked point is forgotten, i.e.~where either $P'=P$ and $Q'= \Qq$, or $P' = \Pp$ and $Q' = Q$.  By functoriality, it is sufficient to prove the theorem in these two cases separately, which we do in \autoref{sec:forget-bdry} and \autoref{sec:forget-interior} below.

As we shall see, \autoref{thm:disk-moduli-pushforward:acyc}  follows immediately  from the general results \autoref{sec:monodromy}.  But also \autoref{thm:disk-moduli-pushforward:deRham} follows from \autoref{thm:disk-moduli-pushforward:acyc}.  Indeed, since the sheaves $\VAr[k]{f}$ of relative differentials are obtained by tensoring $\Vr{P',Q'}$ with a constant sheaf of free $\ZZ$-modules, their higher direct images also vanish.  Hence the hyperderived pushforward $Rf_*(\VAr{f},\td)$ coincides with the ordinary pushforward $f_*(\VAr{f},\td)$ of this sheaf of dg algebras.  On the other hand, it follows from \autoref{prop:moduli-poincare} that $\VAr{f}$ is a resolution of the locally constant sheaf $f^{-1}\Vr{P,Q}$; hence we have an isomorphism $Rf_*(\VAr{f},\td) \cong Rf_*(f^{-1}\Vr{P,Q}) \cong \Vr{P,Q} \otimes Rf_*\underline{\ZZ}$ by the projection formula.

Finally, for \autoref{thm:disk-moduli-pushforward:integral} of \autoref{thm:disk-moduli-pushforward}, we recall that the pushforward of differential forms along the projection  $f \colon \C{P',Q'} \to \C{P,Q}$ is a $\forms{\C{P,Q}}$-linear operation, and annihilates any form of degree less than the fibre dimension $2m+n$.  In light of the decomposition \eqref{eq:arnold-tensor} of the $\ZZ$-module of forms into vertical and horizontal pieces, it is sufficient to establish \autoref{thm:disk-moduli-pushforward:integral} for a form whose degree is equal to the fibre dimension.  We do so by constructing a suitable fibrewise primitive for the differential form, and using Stokes theorem; the existence of such a primitive follows from \autoref{thm:disk-moduli-pushforward:deRham}, but to establish that the result is a polylogarithm on the base---and especially to control its weight---we need to undertake a careful analysis of the geometry and the asymptotics of the primitive on the boundary.

% so in \autoref{
%
%explain why \autoref{lem:projection-tower} holds for these elementary projections.  We do so in  \autoref{sec:acyc-bdry} and \autoref{sec:acyc-int} below, where we also give an explicit description of the sections of $f_*\Vr{P',Q'}$.
%
%
%
%
%More generally, given inclusions $P \subset P'$ and $Q \subset Q'$, we obtain projection maps
%\[
%f : \C{P',Q'} \to \C{P,Q}
%\]
%and the pushforward $f_*\Vr{P',Q'}$ is the sheaf of polylogarithms that are single-valued along the fibres of $f$.  To understand their structure we make use of the following observation:
%\begin{lemma}\label{lem:projection-tower}
%The projection $f$ can be factored as a tower of families of cycles encircling suitable collections of marked points.
%\end{lemma}
%
%We may therefore apply \autoref{thm:families-push} to obtain an algebraic model for $f_*\Vr{P',Q'}$.  In particular, we have the following:
%\begin{corollary}\label{cor:VPQ-acyc}
%For any projection $f$, the higher direct images $R^jf_*\Vr{P',Q'}$ vanish.  Hence $\Vr{P',Q'}$ is an acyclic sheaf on $\C{P',Q'}$.
%\end{corollary}
%
%Note that since any projection can be factored as a sequence of elementary projections in which a single boundary or interior marked point is forgotten, it is enough to explain why \autoref{lem:projection-tower} holds for these elementary projections.  We do so in  \autoref{sec:acyc-bdry} and \autoref{sec:acyc-int} below, where we also give an explicit description of the sections of $f_*\Vr{P',Q'}$.

\subsection{Forgetting a boundary marked point}
\label{sec:forget-bdry}

\subsubsection{Geometry of the projection}
\label{sec:acyc-bdry}
Consider a projection $f \colon \C{P,\Qq} \to \C{P,Q}$ in which a single boundary  marked point called $q$ is forgotten.  Considering the embeddings in the moduli space of curves, we obtain a commutative diagram
\[
\xymatrix{
\C{P,\Qq} \ar@{^(->}[r]\ar[d]_{f} & \M{\Sq}\ar[d] \cong \uXo{S} \\
\C{P,Q}\ar@{^(->}[r]  & \M{S}
}
\]
of the form \eqref{eq:family-of-cycles}.  The fibres of $f$ are open intervals, corresponding to marked disks in which the locations of the points labelled by $P$ and $Q$ are fixed, and the point $q$ moves between its neighbouring marked points $q_-,q_+ \in \pInf{Q}$. 

Note that if $|Q| > 0$ then $q_- \ne q_+$, but if $|Q| = 0$, then $q_- = q_+ = \infty$ so the fibres give a path between  opposite tangential base points at $\infty$.

%Concretely, if $|Q|\geq 2$, we can define a cross ratio $x=\crossrat{q}{q_+}{q_-}{q'}$ with a fourth boundary point $q'$. It sends $q_-$ to $0$ and $q_+$ to $1$ and takes values in the interval $x\in(q_-,q_+)=(0,1)$. Thus $(f,x)$ trivializes the fibration $f$:
%\[
%	\C{P,\Qq} \cong \C{P,Q} \times (0,1) \subset \uXo{S}.
%\]
%If $|Q|\leq 1$ we can pick a point $p\in P$ and consider the cross ratio $x=\crossrat{q}{q_-}{p}{\bar{p}}$. Restricted to $\C{P,\Qq}$, it takes values in the unit circle and it identifies the fibres of $f$ with counter-clockwise arcs $(q_-,q_+)=(1,e^{i\phi})$ parametrized by
%\begin{equation*}
%	\gamma\colon (0,1) \rightarrow S^1 \setminus \sset{1},\quad
%	t \mapsto e^{i t\phi}.
%\end{equation*}
%When $|Q|=1$ we have $q_-\neq q_+$ and the endpoint $\crossrat{q_+}{q_-}{p}{\bar{p}}=e^{i\phi}$ of the arc, where $q=q_+$, defines an analytic function $\phi\colon \C{P,Q}\rightarrow(0,2\pi)$.
%The case $Q=\emptyset$ is special since $q_-=q_+=\infty$. Then $\phi=2\pi$ and the arc $S^1\setminus\sset{1}$ are constant.

\subsubsection{Structure of polylogarithms}
\label{sec:boundary-polylogs}
The fibres of $f$ are contractible, and hence any polylogarithm is single-valued on the fibres, and the higher cohomology of any locally constant sheaf along the fibres vanishes, which gives \autoref{thm:disk-moduli-pushforward:acyc} and hence \autoref{thm:disk-moduli-pushforward:deRham} in this case.  Note that by using $q_-$ as a base point for hyperlogarithms, we get a canonical isomorphism
\[
f_*\Vr{P,\Qq} \cong \Vr{P,Q} \abrac{S}.
\]
This is a special case of  \autoref{lem:family-cycles-push} in which $T = \emptyset$.

% gives an isomorphism
%\[
%\gr f_{*} \Vr{P,\Qq} \cong \Vr{P,Q}\abrac{S}.
%\]
%In analytic terms, this means that \emph{every} germ of a polylogarithm on $\C{P,\Qq}$ extends to a single-valued function on the fibres. Indeed, recall its expansion as an $f^{-1}\Vr{P,Q}$-linear combination of hyperlogarithms $\hlog{u}(x)$ of the fibre coordinate $x$, where the latter may have singularities at the marked points $S\cup\infty \subset \PP[1]$.
%But $x \in \PP[1]$ is restricted to the interval $(0,1)$ or the arc $\gamma_y$, which are contractible and disjoint from the marked points. Hence, the fundamental group of the fibre is trivial and the hyperlogarithm $\hlog{u}(x)$ is single-valued and analytic on the whole fibre, i.e.~it defines a section of $f_*\Vr{P,\Qq}$.
%
%Actually, much more is true: the triviality of the fibration $f$ implies that the monodromy action of $\fg{\C{P,Q}}$ on the fibres is trivial. Therefore, each $u\in \T{S}$ defines a global section $\hlog{u} \in \V{\C{P,\Qq}} \cong \V{\C{P,Q}} \abrac{S}$.

\subsubsection{Fibre integration}
\label{sec:interval-integration}
Now suppose that $\omega \in \VA[1]{\C{P,\Qq}}[j]$ is a one-form whose pushforward $f_*\omega$ converges absolutely.  Its restriction to the fibres of $f$ is a hyperlogarithmic one-form and hence the  integral along the fibre is given by a period of the universal curve.  It therefore lies in $\V{\C{P,Q}}[j]$ by definition.

To complete the proof of \autoref{thm:disk-moduli-pushforward} in this case, we need to establish the additional weight drop when $Q = \emptyset$.  To see this, we use the fact the fibres give a loop in the universal curve, based at $\infty$, and apply \autoref{lem:loop-weight} on the weight drop for closed loops.  An important subtlety here is that \autoref{lem:loop-weight} only applies to loops that are based at some fixed tangential base point, whereas the loop in question goes between the opposite base points $\pm \tinf$.  However, we are assuming that the integral converges absolutely.  Therefore the result is independent of the tangential base point and \autoref{lem:loop-weight} still applies.

%
%
%In more detail, the pushforward of a one-form can be calculated algorithmically as follows.  Note that the differential forms $\dlog{s}$ for $s \in S$ give a basis for the relative differentials of the projection.  Thus by using $q_-$ as a base point for hyperlogarithms in the fibres, we obtain a unique expansion
%\[
%\omega = \sum_{s \in S} L_{u_s} \dlog{s} \mod f^*\forms[1]{\C{P,Q}}
%\]
%where $L_{u_s} \in \V{\C{P,\Qq}}[j-1]$ denotes the fibrewise hyperlogarithm associated to an element $u_s \in \V{\C{P,Q}}\abrac{S}$. According to \eqref{eq:hlog-diff}, we can construct a primitive for $\omega$ by prepending the symbol $s$ to each element $u_s$, giving the hyperlogarithm
%\[
%\alpha \defas \sum_s L_{su_s} \in \V{\C{P,\Qq}}[j]
%\]
%such that $\td \alpha = \omega$ as functions on the fibre. Then by Stokes theorem,
%\begin{equation*}
%f_*\omega = \int_{(q_-,q_+)} \omega
% =  (\lim_{q\to q_+} -\lim_{q \to q_-}) \alpha
% = \int_{(q_-,q_+)} \sum_s s u_s
% \in \V{\C{P,Q}}[j],
%% \label{eqn:int-push-limits}
%\end{equation*}
%where the latter integral sign denotes an iterated integral to make apparent that $f_{\ast} \omega$ is a particular period of the marked curve in the fibre. Note that, by assumption, the integral is absolutely convergent, so it does not require regularization and tangential base points.
%
%Recall that for $Q=\emptyset$, the path of integration $(q_-,q_+)=\gamma$ becomes a closed circle with $\gamma(0)=\gamma(1)=1$, so the weight drops down to $j-1$ by \autoref{lem:loop-weight}.
%
%

\subsection{Forgetting an interior marked point}
\label{sec:forget-interior}

\subsubsection{Geometry of the projection}
Now consider the more interesting projection $g\colon \C{\Pp,Q} \to \C{P,Q}$ which forgets a single interior marked point $z$.  As explained in \autoref{sec:disk-moduli-spaces}, it can be naturally identified with the punctured universal marked disk $\uDo{} = \uDo{P,Q} \to \C{P,Q}$, so to simplify the notation we shall write
\[
	 \Vr{\uD{}} = \Vr{\Pp,Q} \qquad\text{and}\qquad
	\VAr{\uD{}} = \VAr{g}
\]
for the sheaves of polylogarithmic functions and relative differential forms.

Note that under the maps that double a disk to get a genus zero curve, the projection $g$ corresponds to a projection of the form $\M{S\cup\sset{z,\zb}} \to \M{S}$ that forgets both $z$ and its conjugate $\zb$.   In what follows it will be convenient to break the symmetry between $z$ and $\zb$ by factoring the projection $\M{\Spp} \to \M{S}$ into a pair of projections, where we first forget $\zb$, then forget $z$. The first of these projections gives an isomorphism of the universal punctured disk $\uDo{}$ onto its image: indeed, for a marked genus zero curve obtained by doubling a disk, the location of the conjugate point $\zb$ is uniquely determined by the location of $z$.  Hence we have a commutative diagram
\begin{equation}\vcenter{%
\xymatrix{
\uDo{} \ar@{^(->}[r]\ar@{=}[d] & \M{\Spp}\ar[d] \cong \uXo{\Sp} \\
\uDo{}\ar@{^(->}[r] \ar[d]_-{g}  & \M{\Sp} \cong \uXo{S} \ar[d]\\
\C{P,Q}\ar@{^(->}[r]  & \M{S}
}}\label{eq:forget-interior-tower}%
\end{equation}
where the middle horizontal arrow embeds $\uDo{}$ as a punctured disk in the universal curve $\uXo{S}$ as in \autoref{sec:disk-moduli-spaces}.  Note that this diagram is exactly of the type considered in \autoref{thm:families-push}.  We conclude immediately that the higher direct images of $\Vr{\uD{}}$ vanish, so that \autoref{thm:disk-moduli-pushforward:acyc} and \autoref{thm:disk-moduli-pushforward:deRham} of \autoref{thm:disk-moduli-pushforward} hold for the projection $g$.

%In order to shorten the discussion, we will throughout this section assume that $\abs{Q}\geq 2$, such that $z$ and $\bar{z}$ are cross ratios as discussed in \autoref{sec:disk-moduli-spaces}:

In order to shorten the proof of \autoref{thm:disk-moduli-pushforward:integral}, we will assume throughout this section that $\abs{Q}\geq 2$, so that $z$ and $\bar{z}$ are cross ratios as discussed in \autoref{sec:disk-moduli-spaces}:
\begin{lemma}\label{lem:CPQ-base-change}
In proving \autoref{thm:disk-moduli-pushforward:integral} of \autoref{thm:disk-moduli-pushforward} for the projection $g$, we may assume without loss of generality that $|Q| \ge 2$.
\end{lemma}
\begin{proof}
If $|Q| < 2$, choose an embedding $Q \hookrightarrow \{0,1\}$.  We then obtain a Cartesian diagram of projections
\[
\xymatrix{
\uDo{P,\{0,1\}} \ar[r]^-{\tilde f} \ar[d]^-{\tilde g} & \uDo{} \ar[d]^-{g} \\
\C{P,\{0,1\}} \ar[r]^-{f} & \C{P,Q}.
}
\]
The pullback $f^* \colon \V{\C{P,Q}} \to \V{\C{P,\{0,1\}}}$ is an injection of filtered rings whose image is exactly the polylogarithms on $\C{P,\{0,1\}}$ that are constant along the fibres of $f$. Hence the desired statement for the pushforward $g_*$ follows from the corresponding statement for $\widetilde g_*$  by the base change formula $f^*g_*\omega = \tilde g_* \tilde f^*\omega$.
\end{proof}

\subsubsection{Structure of polylogarithms}
\label{sec:disk-polylogs}
 Note that while we view the locations of $z$ and $\zb$ as complex conjugate coordinates on the fibres of $\uD{}$, the corresponding cross ratios give genuinely independent \emph{holomorphic} coordinates on the moduli space $\M{\Spp}$.  As a consequence, the sections of the sheaf $\Vr{\uD{}} \defas \Vr{\Pp,Q}$ of polylogarithms are $\Vr{P,Q}$-linear combinations of certain analytic functions $L(z,\zb)$ that may have singularities whenever $z$ or $\zb$ collide with the other marked points, or when $z$ collides with $\zb$; i.e.~on the real line $z=\zb$ in our coordinates.

To describe these functions algebraically we use the diagram \eqref{eq:forget-interior-tower}.  Taking the relative weights of hyperlogarithms along the two projections, we obtain a bifiltration $\Vr{\uD{}}[\bullet,\bullet]$ such that the functions in $\Vr{\uD{}}[\bullet,0]$ are independent of $\zb$ (hence holomorphic in the natural complex structure on the disk fibres), and the functions in $\Vr{\uD{}}[0,0] = g^*\Vr{P,Q}$ are constant on the fibres.

After a suitable choice of tangential base points and paths, we may write a section $L \in \Vr{\uD{}}$ in the form
\begin{equation}
L = \sum_{j} a_j \hlog{u_j}\bhlog{v_j} \label{eq:disk-fibration-basis}
\end{equation}
where the individual factors are as follows:
\begin{itemize}
\item The coefficients $a_j \in g^*\Vr{P,Q}$ are fibrewise constant polylogarithms.
\item The functions $\hlog{u_j}$ with $u_j \in \T{S}$ are the restrictions of hyperlogarithms on the universal curve $\uX{S}$, constructed by integrating the forms
\[
	\dlog{s} = \frac{1}{2\ipi} \frac{\td z}{z-s}, \qquad s \in S.
\]
\item The functions $\bhlog{v_j}$ with $v_j \in \T{S,z}$ are the restrictions of hyperlogarithms on the universal curve $\uX{S\cup\sset{z}}$, constructed integrating the forms
\[
	\bdlog{s} = \frac{1}{2\ipi} \frac{\td \zb}{\zb-s}, \qquad s \in S \cup \sset{z}.
\]
Note that in general, such functions are \emph{not} antiholomorphic on the fibres, because they may involve the differential form $\bdlog{z} = \frac{1}{2\ipi} \frac{\td \zb}{\zb -z}$. 
\end{itemize}

Note that only a proper subset of the expressions of the form \eqref{eq:disk-fibration-basis} will define functions that are single-valued along the fibres of $\uD{}$, i.e.~sections of $g_*\Vr{\uD{}}$.  The degrees of freedom can be determined by applying~\autoref{thm:families-push} to the diagram \eqref{eq:forget-interior-tower}. 
\begin{lemma}\label{lem:disk-gr}
The direct image sheaf $g_*\Vr{\uD{}}$ on $\C{P,Q}$ has a double filtration with associated bigraded
\begin{equation}
	\gr W^{\bullet,\bullet} g_*\Vr{\uD{}} \cong  \Vr{P,Q} \otimes \underline{\T{S|P}} \otimes \underline{\T{S,z}}, \label{eq:disk-gr}
\end{equation}
and $R^kg_*\Vr{\uD{}} = 0$ for all $k > 0$.
\end{lemma}

Let us gives some examples of such fibrewise single-valued functions, using $\vec{\infty}$ as the base point for hyperlogarithms:

\begin{example}
If $p \in P$ and $q \in Q$, then the functions 
\begin{equation*}
\hlog{\pb} = \frac{\log(z-\pb)}{2\ipi} ,\quad
\bhlog{p}  = \frac{\log(\zb - p)}{2\ipi} ,\quad
\hlog{q}   = \frac{\log(z-q)}{2\ipi} \quad\text{and}\quad
 \bhlog{q} = \frac{\log(\zb-q)}{2\ipi}
\end{equation*}
are single-valued; in each case the logarithm stays on the standard principal branch (the imaginary part vanishes as $z,\zb \to \infty$ along the positive real axis).
These functions have symbols
\begin{equation*}
1 \otimes \pb\otimes 1,\qquad
1\otimes 1\otimes p,\qquad
1\otimes  q\otimes 1 \qquad\text{and}\qquad
1 \otimes 1 \otimes q 
\end{equation*}
in the associated bigraded. \qed
\end{example}

\begin{example}
For $p \in P$, the functions $\hlog{p} = \frac{\log(z-p)}{2\ipi}$ and $\bhlog{\pb} = \frac{\log(\zb - \pb)}{2\ipi}$ are not single-valued; they have monodromy $+1$ and $-1$ around $p$, respectively.  However, the linear combination
\[
	\hlog{p} + \bhlog{\pb} = \frac{\log|z-p|^2}{2\ipi} \in \V{\C{P\sqcup\sset{z},Q}}[1]
\]
is single-valued.  Its symbol in the associated bigraded is $1 \otimes \pb$. \qed
\end{example}

\begin{example}
The hyperlogarithm
\begin{equation}
\bhlog{z} 
= \frac{\log(\zb-z)}{2\ipi}
= \frac{\log(-2\iu \Im z)}{2\ipi} 
= -\frac{1}{4} + \frac{\log(2\Im z)}{2\ipi}
\label{eq:logIm}%
\end{equation} is single-valued on the disk.  Its symbol is $1 \otimes 1 \otimes \zb$.
\qed 
\end{example}

\begin{example}\label{ex:disk-dilog}
	For $q\in Q$ and $p \in P$, we can construct a section of $g_*\Vr{\uD{}}$ whose symbol is $1 \otimes q \otimes \pb$ as follows.  Start with the obvious fibrewise hyperlogarithm
\[
\Lambda_0 \defas \hlog{q}\bhlog{\pb}
\]
with symbol $1 \tp q \otimes \pb$.  In light of the examples above, we see that the monodromy of $\Lambda_0$ around $p$ is equal to  $-\hlog{q}$, while the monodromy around the other interior marked points is trivial.  As explained in \autoref{ex:monodromy-coboundary}, we can construct a holomorphic hyperlogarithm having exactly this monodromy, namely
\[
\Lambda = \hlog{qp} + \tfrac{\log(p-q)}{2\ipi}\hlog{p}.
\]
Therefore the function
\[
\Lambda_0+\Lambda = \hlog{q}\bhlog{\pb} +  \hlog{qp} + \tfrac{\log(p-q)}{2\ipi}\hlog{p}
\]
is single-valued along the fibres, but its symbol is still $1 \otimes q\otimes \pb$. \qed
\end{example}

In what follows, we will need some information about the asymptotic behaviour of polylogarithms on the disk fibres.  For this, it is convenient to introduce polar coordinates
\[
	r = |z-s| \qquad\text{and}\qquad \theta = \arg(z-s)
\] 
centred at the marked points $s \in P \cup Q$. For $s=\infty$, we set $r=\abs{z^{-1}}$ and $\theta=\arg(-z^{-1})$. We will only refer to $\theta$ in the cases when $s\in \pInf{Q}$, and we fix the branch of $\arg$ so that $\theta \in [0,\pi]$ in these cases.

\begin{proposition}\label{prop:disk-log-expansion}
Suppose that $|Q| \ge 2$ and let $L \in g_*\Vr[j]{\uD{}}$ be a polylogarithm of weight $j$ on $\uDo{}$ that is single-valued along the fibres. Then we have the following asymptotic behaviour of $L$ in the natural fibre coordinates $(z,\zb)$:
\begin{enumerate}
\item \label{prop:disk-log-expansion:real} The germ of $L$ at an unmarked boundary point $x \in \partial\uD{} \setminus \pInf{Q}$ has the form
\begin{equation}
L = \sum_{n=0}^j  L_n \log^n(2 i \Im z)
\label{eq:disk-log-expansion:real}%
\end{equation}
where the functions $L_n$ are analytic at $x$.  Moreover 
\[
	L_0= L_0^{hol} + \asyO{|\Im z|}
\]
for some hyperlogarithm $L_0^{hol} \in \Vr{\uD{}}[\bullet,0]$ that is holomorphic along the fibres and defined in a neighbourhood of $x$.

\item \label{prop:disk-log-expansion:Q} The germ of $L$ at a boundary marked point $q \in \pInf{Q}$ can be written in the form
\[
L = \sum_{n=0}^N\sum_{m=0}^M L_{mn}(r,\theta) \log^m r \log^n\theta 
\]
where the functions $L_{mn}$ are analytic for $(r,\theta) \in \sset{0} \times [0,\pi)$. An analogous expansion exists where $\log \theta$ is replaced by $\log(\pi-\theta)$ and the functions $L_{mn}$ are analytic for $(r,\theta) \in \sset{0}\times(0,\pi]$.

\item  \label{prop:disk-log-expansion:P} The germ of $L$ at an interior marked point $p \in P$ has the form
\[
L = \sum_{n=0}^j L_n  \log^n r 
\]
for some functions $L_n$ that are analytic at $p$. Moreover,
\[
	L_0 \in g^*\Vr[j]{P,Q} + \asyO{r}.
\]
\end{enumerate}
\end{proposition}

\begin{proof}
All three statements are a straightforward application of \autoref{prop:moduli-log-sing} on the asymptotics of polylogarithms on the boundary strata of $\bM{S\cup\sset{z,\zb}}$. 

For the first statement, we consider the codimension-one boundary stratum given by the collision of the marked points $z$ and $\zb$.  This stratum is isomorphic to $\M{S\cup\sset{z}}$ and the transverse coordinate $t = \crossrat{\zb}{0}{z}{\infty}$ gives an expansion in 
$
\tfrac{1}{2\ipi}\log t 
= \frac{1}{2\ipi} \log \frac{z-\zb}{z}
\equiv \frac{1}{2\ipi}\log(z-\zb)
\mod \Vr{P,Q}[1,0]$, as desired.

For the second statement, we consider the codimension-two stratum of $\M{S}$ corresponding to a nested collision of the sets $\{z,\zb\} \subset \{z,\zb,q\}$.  As transverse coordinates in this case we choose the cross ratios $t_1 = \crossrat{z}{q_1}{q}{q_2}$ and $t_2 = \crossrat{z}{q}{\zb}{\infty}$, where $q_1\ne q_2 \in \Qinf$ are distinct from $q$ (for $q=\infty$, set $t_2=\crossrat{z}{\infty}{\zb}{0}$).  We have
\[
t_1 = \frac{(z-q)(q_1-q_2)}{(z-q_2)(q_1-q)} = r \cdot \frac{q_2-q_1}{q_2(q_1-q)} \mod r,\theta
\]
and 
\[
t_2 = \frac{z-\zb}{q-\zb} = \frac{r(e^{\iu \theta}-e^{-\iu\theta})}{-re^{-\iu \theta}} =-2 \iu \theta \mod \theta
\]
so that the fibrewise change of variables $(t_1,t_2) \to (r,\theta)$ is analytic, and the statement follows from the expansion in $(t_1,t_2)$.

Finally, for the third statement, we consider the codimension-two stratum corresponding to the independent collision of $z$ with $p$ and $\zb$ with $\pb$.  This stratum is isomorphic to $\M{S}$.  Using the cross ratios $t = \crossrat{z}{0}{p}{\infty}$ and  $\bar t = \crossrat{\zb}{0}{\pb}{\infty}$ as transverse coordinates, we obtain an expansion that is a polynomial in the functions $\log(z-p)$ and $\log(\zb -\pb)$, which are multivalued on the fibres.  One can then check that the single-valuedness of $L$ in a neighbourhood of $p$ forces the expansion to collapse to a polynomial in the single-valued logarithm  $\log|z-p|^2 = 2 \log r$.
\end{proof}

\subsubsection{Single-valued primitives}

Now suppose that $\omega \in g_{\ast} \VAr[2]{\uD{}}[j]$ is a section of the sheaf of relative two-forms on $\C{\Pp,Q} \cong \uDo{}$, of total weight $j$. 
We can apply \eqref{eq:universal-relative} twice and write $\omega$ uniquely in the form
%$\omega \in \VAr[2]{\uD{}} \cong \Vr{\uD{}}\tp \underline{\ZZ}^{S\cup\sset{z}}\tp \underline{\ZZ}^S$ of the relative two-forms uniquely as
\begin{equation}
	\omega = \sum_{t\in S\cup\sset{z}} \sum_{s\in S} f_{s,t} \cdot \bdlog{t} \wedge \dlog{s},
	\label{eq:relative-disc-form-basis}%
\end{equation}
where the coefficients $f_{s,t} \in \Vr{\uD{}}$ are local sections of $\Vr{\uD{}}$.

Suppose further that the pushforward $g_*\omega$ converges absolutely.  In what follows we will abuse notation and write
\[
\int_{\uD{}} \omega \defas g_*\omega
\]
for the pushforward.  To complete the proof of \autoref{thm:disk-moduli-pushforward}, we must show that $\int_{\uD{}}\omega$ is a polylogarithm on $\C{P,Q}$ with weight at most $j-1$.

To compute the pushforward of $\omega$, we search for a primitive, and apply Stokes theorem.    Note that, by definition, the cohomology classes of $\omega$ give a section of the sheaf of fibrewise cohomologies:
\[
[\omega] \in \scohlgy[2]{g_*\VAr{\uD{}},\td}.
\]
But by \autoref{thm:disk-moduli-pushforward:deRham} of \autoref{thm:disk-moduli-pushforward}, this sheaf is isomorphic to $\Vr{P,Q}\otimes R^2g_*\cZZ$, and since the fibres are punctured disks, their second cohomologies vanish, i.e.~we have $R^2g_*\cZZ=0$.   Hence over any contractible open set in the base, we can find a fibrewise primitive for $\omega$.  For the purposes of calculation, this primitive can be chosen arbitrarily, but to establish the desired bound on the weight of $g_*\omega$, and to give a concrete algorithm for calculating the integral, it is convenient to construct a primitive that is suitably compatible with the complex structure on the fibres:

\begin{proposition}\label{lem:1-0-primitive}
The form $\omega$ admits a fibrewise primitive $\alpha \in g_*\VAr[1]{\uD{}}[j]$ which has type $(1,0)$ in the Dolbeault decomposition along the fibres.
\end{proposition}
\begin{proof}
	Expand the coefficients $f_{s,t}$ in \eqref{eq:relative-disc-form-basis} into hyperlogarithms as in \eqref{eq:disk-fibration-basis}.
	To construct a primitive we just integrate with respect to $\zb$ according to \eqref{eq:hyperlog-primitive}, that is $\bhlog{u} \bdlog{t} \mapsto \bhlog{tu}$.
	The result is a $(1,0)$-form $\talpha=\sum_{s\in S} \tilde{\alpha}_s \dlog{s} \in \VAr[1]{\uD{}}[j]$ of weight $j$ that is a fibrewise primitive for $\omega$, but $\talpha$ may have monodromy in the fibres. 

	To correct this, choose a contractible open subset of the base $\C{P,Q}$ and a collection of loops $\gamma_p$ for $p \in P$ that freely generate the fundamental groups of the fibres.  Note that since $\omega$ has no monodromy, and analytic continuation commutes with differentiation, the monodromy
\[
\delta \talpha = ((\gamma_p-1)\talpha)_{p\in P}
\]
around any generating loop $\gamma_p$ must be a closed $(1,0)$-form, and hence holomorphic.  We may therefore think of the collection $\delta\talpha$ as a fibrewise one-cocycle for the sheaf of fibrewise holomorphic hyperlogarithmic forms.  Applying the vanishing of cohomology for hyperlogarithm sheaves as in \autoref{sec:monodromy}, we conclude that there exists a fibrewise holomorphic form $\talpha'$ whose monodromy in the fibres is equal to that of $\talpha$.   But since $\talpha'$ is holomorphic, it is also closed along the fibres, so $\alpha \defas \talpha - \talpha'$ gives the desired primitive of $\omega$.
\end{proof}
Note that the holomorphic form $\tilde{\alpha}'=\sum_s \tilde{\alpha}'_s \dlog{s}$ can be constructed explicitly using %by $\tilde{\alpha}'_s = (1+h\delta')^{-1}h\delta \talpha_s \in g^{\ast} \Vr{P,Q}\tp \T{S}P$ 
\eqref{eq:monodromy-cancellation}. The resulting fibrewise single-valued polylogarithms $\alpha_s=\tilde{\alpha}_s-\tilde{\alpha}'_s \in g_*\Vr[j-1]{\uD{}}$ are the coefficients of
\begin{equation}
	\alpha = \bigg( \sum_{s \in S}  \frac{\alpha_s(z,\zb)}{(2\ipi)(z-s)} \bigg) \td z.
	\label{eq:disk-1,0}%
\end{equation}
\begin{example}
	Consider a two-form $\omega = \hlog{q}\cdot \bdlog{\bar{p}} \wedge \dlog{s}$. We obtain the $(1,0)$-form $\tilde{\alpha} = \Lambda_0 \dlog{s}$ as a primitive, where $\Lambda_0=\tilde{\alpha}_s=\hlog{q} \bhlog{\bar{p}}$ has holomorphic monodromy $-\hlog{q}$ around $p$. 
	In \autoref{ex:disk-dilog} we constructed a holomorphic polylogarithm $-\Lambda$ with the same monodromy, and hence we obtain a holomorphic form $\tilde{\alpha}'=-\Lambda \dlog{s}$ such that $\alpha = (\Lambda+\Lambda_0) \dlog{s}$ is single-valued primitive of $\omega$. \qed
\end{example}
\begin{remark}[global single-valuedness]\label{rem:sv-int}
	Once we fix a tangential base point $\st$ in $P \cup \pInf{Q}$, the fibration basis \eqref{eq:partial-fibration-basis} in the disk % according to \eqref{eq:forget-interior-tower} 
	globalizes to an embedding
	\begin{equation*}%\label{eq:sv-fibration-disk} \tag{$\ast$}
		g_{\ast} \Vr{\uD{}} \injection \Vr{P,Q} \tp \underline{\T{S}} \tp \underline{\T{S,z}},
	\end{equation*}
	because by single-valuedness in the fibres it does not matter which path $\gamma$ in $\HH$ from $\st$ to $z$ we choose to define the hyperlogarithms. This also means that analytic continuation along a loop $\eta \in \fg{\C{P,Q}}$ in the base acts only on the first factor $\Vr{P,Q}$, since its action on the hyperlogarithms in the fibre is by conjugation of $\gamma$.
	The construction $\sum_t f_{s,t} \bdlog{t} \mapsto \alpha_s$ of the single-valued primitive given in the proof of \autoref{lem:1-0-primitive} thus determines a morphism of sheaves
	\begin{equation*}
		\int_{\sv}\colon g_{\ast} \Vr{\uD{}} \tp \A[1]{\uX{S\cup\sset{z}}} \rightarrow g_{\ast} \Vr{\uD{}}
		\quad\text{such that}\quad
		\frac{\partial}{\partial \zb} \int_{\sv} f \cdot \bdlog{t} = \frac{f}{2\ipi(\zb-t)},
	\end{equation*}
	and for a global section $f \in \V{\C{P\cup\sset{z},Q}}$ also $\int_{\sv} f \bdlog{t} \in \V{\C{P\cup\sset{z},Q}}$ is single-valued (this was observed in \cite{SchlottererSchnetz:ClosedStringGenus0}). An iteration of this map gives rise to a single-valued version of hyperlogarithms, i.e.\ an explicit embedding
	\begin{equation*}
		\T{S,z} \injection \V{\C{P\cup\sset{z},Q}},\quad
		t_1\cdots t_n \mapsto \int_{\sv} \bdlog{t_1} \cdots \int_{\sv} \bdlog{t_n}.
	\end{equation*}
	The restriction of this map to words in $\T{\Pb}\subset \T{S,z}$ reproduces \cite[Theorem~8.1]{Brown2004a} and provides an explicit construction for the global sections in \autoref{ex:MS-sv}. Some further special cases have been discussed in \cite{ChavezDuhr:Triangles,Schnetz:NumbersAndFunctions}. \qed
\end{remark}

\subsubsection{Regularized Stokes theorem on the fibres}

Let $\epsilon > 0$ be given, and let $\uD[\epsilon]{} \subset \uD{}$ be the closed subset obtained by removing the open (half-)disk $\{|z-s| <\epsilon\}$ of radius $\epsilon$ around each $s \in P\cup Q$ and the disk $\{\abs{z}^{-1}<\epsilon\}$ around $\infty$.  As illustrated in \autoref{fig:blowup-disk}, the boundary cycle $\partial\uD[\epsilon]{}$ in a given fibre is, for $\epsilon$ sufficiently small, the sum of several smooth one-manifolds, which we label as follows:

\begin{itemize}
\item For each boundary marked point $q$, let $q'$ be its successor in the cyclic order.  We denote by
\[
\intvl{q} \defas  \buD{} \cap (q,q')
\]
the portion of the boundary lying on the real line $z=\zb$ between $q$ and $q'$.
\item For each marked point $s \in P \cup \Qinf$, we denote by
\[
\arc{s} \subset \buD{}
\]
the (half-)circle of radius $\epsilon$ centred at $s$, with orientation induced from that of the disk. By convention, $\arc{\infty}$ is the half-circle $\abs{z}= \epsilon^{-1}$.
\end{itemize}
Then the connected components of $\partial\uD[\epsilon]{}$ are given by the circles $\arc{p}$ for $p \in P$, and the one-cycle
\[
\outD{} \defas \sum_{q \in \Qinf} \intvl{q}+\arc{q}.
\]

\begin{figure}
\centering
\begin{tikzpicture}[scale=0.8]

\begin{scope}
\def\Rad{4}
\def\rad{0.75}
% Node positions
\node (q1) at (-130:\Rad) {};
\node (q2) at (-70:\Rad) {};
\node (q3) at (-20:\Rad) {};
\node (qinf) at (90:\Rad) {};

% Fill the interior
\draw[fill=lightgrey] (0,0) circle (\Rad);

% Draw real line boundary arcs
\placeoutseg{q1}{-130}{-70}{$(q_1,q_2)^\epsilon$}
\placeoutseg{q2}{-70}{-20}{$(q_2,q_3)^\epsilon$}
\placeoutseg{q3}{-20}{90}{$(q_3,\infty)^\epsilon$}
\placeoutseg{qinf}{90}{230}{$(\infty,q_1)^\epsilon$}

% Cut out interior circles
\placeinterior{1}{30:0.5*\Rad}
\placeinterior{2}{160:0.6*\Rad}
\placeinterior{3}{-50:0.3*\Rad}

% Cut out boundary half-circles
\placeoutpt{q1}{$q_1$}
\placeoutpt{q2}{$q_2$}
\placeoutpt{q3}{$q_3$}
\placeoutpt{qinf}{$\infty$}
\clip (0,0) circle (\Rad);
\placeoutarc{q1}{-130}{$\arc{q_1}$}
\placeoutarc{q2}{-70}{$\arc{q_2}$}
\placeoutarc{q3}{-20}{$\arc{q_3}$}
\placeoutarc{qinf}{90}{$\arc{\infty}$}

\end{scope}
\end{tikzpicture}%
\caption{The surface $\uD[\epsilon]{}$ obtained by removing small disks around the marked points $P = \sset{p_1,p_2,p_3}$ and $\Qinf = \sset{q_1,q_2,q_3,\infty}$.}%
\label{fig:blowup-disk}%
\end{figure}
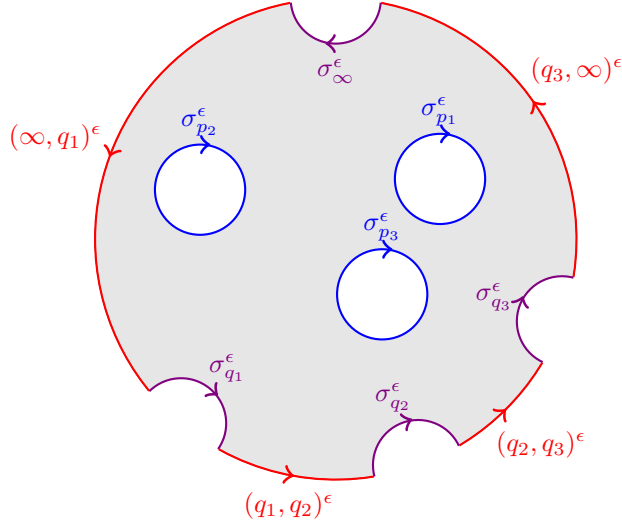

%Since $\omega$ is absolutely integrable, we have
%\[
%\int_{\uD{}} \omega = \lim_{\epsilon \to 0} \int_{\uD[\epsilon]{}} \omega.
%\]
The integral over $\uD[\epsilon]{}$ can be computed using Stokes theorem:
\begin{proposition}\label{prop:disk-stokes}
For $\epsilon$ sufficiently small, the one-form $\alpha$ extends continuously to the boundary $\partial \uD[\epsilon]{}$, and we have
\[
\int_{\uD[\epsilon]{}} \omega = \int_{\partial \uD[\epsilon]{}} \alpha.
\]
\end{proposition}

\begin{proof}
Since $\alpha$ is analytic away from the marked points and the boundary of the disk, the only possible issue comes from the singularities of $\alpha$ along the boundary $\outD{}$ of the disk, i.e.~where $z=\zb$.  But by expanding the coefficients $\alpha_s$ in   \eqref{eq:disk-1,0} using \autoref{prop:disk-log-expansion:real} of \autoref{prop:disk-log-expansion}, we see that $\alpha$ has at worst logarithmic divergences in this region, and no poles.  A local calculation using this expansion then gives the result; we omit the details since this is a special case of a more general version of Stokes' theorem which holds for absolutely integrable differential forms with logarithmic singularities; see~\cite[Theorem 4.11]{Brown2009}.
\end{proof}

In \autoref{sec:inner-contribution} through \autoref{sec:half-circle-contribution} below, we will show that the contribution of each component of the boundary has a regularized limit as $\epsilon \to 0$, so that we may compute
\begin{align*}
\int_{\uD{}}\omega
&= \lim_{\epsilon \to 0} \int_{\uD[\epsilon]{}} \omega \\
&= \lim_{\epsilon \to 0} \int_{\partial\uD[\epsilon]{}} \alpha \\
&= \sum_{p \in P} \rbrac{\Reglim{\epsilon}{0}\int_{\arc{p}} \alpha} + \sum_{q\in\Qinf}\rbrac{\Reglim{\epsilon}{0}\int_{\intvl{q}}\alpha + \Reglim{\epsilon}{0}\int_{\arc{q}}\alpha}.
\end{align*}
We show moreover that each inner circle contributes a polylogarithm on the base $\C{P,Q}$ of weight at most $j-1$.  Meanwhile the intervals and half-circles forming the cycle $\outD{}$ individually contribute polylogarithms of weight at most $j$.  We then show in \autoref{sec:weight-drop} that when all these contributions are added up,  the highest weight parts cancel out, so that the total contribution of $\outD{}$ has weight at most $j-1$, as claimed.

\subsubsection{Contributions from interior circles}
\label{sec:inner-contribution}
The contributions to the integral given by the interior boundary circles can be computed using residues, as follows.   Let us define the residue of $\alpha$ at an interior marked point $p\in P$ to be the (regularized) coefficient of $\tfrac{\td z}{z-p}$:
\begin{align*}
\Res_{p}\alpha & \defas \Rlim_{z \to p}\Rlim_{\zb \to \pb} \frac{(z-p)\alpha}{\td z} \\
&= \Rlim_{z \to p}\Rlim_{\zb \to \pb} \sum_{s \in S} \frac{\alpha_s(z,\zb)(z-p)}{2\ipi(z-s)} \\
&= \Rlim_{z \to p}\Rlim_{\zb \to \pb} \frac{\alpha_p(z,\zb)}{2\ipi} \\
&\in \frac{1}{2\ipi}\Vr[j-1]{P,Q}.
\end{align*}
In the third step we have used \autoref{prop:disk-log-expansion:P} of \autoref{prop:disk-log-expansion} to conclude that the coefficients $\alpha_s$ from \eqref{eq:disk-1,0} have at worst logarithmic divergences at $p$, so that the function $\frac{(z-p)\alpha_s}{(z-s)}$ vanishes as $(z,\zb) \to (p,\pb)$, unless $s = p$.

With this notation in hand, we have the following analogue of the residue theorem, which we learned from a paper of Schnetz \cite[Section~2.8]{Schnetz2014}:
\begin{lemma}
The interior circles contribute convergent limits
\[
\lim_{\epsilon \to 0} \int_{\arc{p}}  \alpha
= -2 \ipi \Res_p\alpha
= -\lim_{(z,\zb)\to(p,\pb)} \alpha_p
\in \Vr[j-1]{P,Q}.
\]
\end{lemma}

\begin{proof}
Considering the expansion of the coefficients of $\alpha$ around $p$ using \autoref{prop:disk-log-expansion:P} of \autoref{prop:disk-log-expansion}, we see that that
\[
	\alpha = f(\log|z-p|) \frac{\td z}{z-p} + \asyO{1}\td z
\]
where $f$ is a polynomial whose constant term is $\Res_p\alpha$, and $\asyO{1}$ denotes terms that are uniformly bounded as $z \to p$.  Since the radius of the circle decreases linearly in $\epsilon$, the integral of the bounded terms vanishes in the limit $\epsilon \to 0$.  Meanwhile, the function $\log|z-p|$ is constant on the circle $\arc{p}$, where its value is the radius $\epsilon$.  Taking account of the orientation of the circle and applying the usual residue theorem for holomorphic forms, we have
\[
\lim_{\epsilon \to 0} \int_{\arc{p}}  \alpha = \lim_{\epsilon \to 0} \int_{\arc{p}} f(\log \epsilon) \frac{\td z}{z-p} = -\lim_{\epsilon \to 0} 2\ipi f(\log \epsilon).
\]
Since $\omega$ is absolutely integrable, we know a priori that the limit must exist, and hence the polynomial $f$ must in fact be equal to its constant term $\Res_p(\alpha)$, giving the result.
\end{proof}

\subsubsection{Contributions from boundary intervals}

Let us choose a boundary marked point $q \in \Qinf$ and let $q'$ be its successor in the cyclic order.  
\begin{lemma}\label{lem:hol-restrict}
In any contractible neighbourhood of $(q,q')$, there exists a unique fibrewise holomorphic form $\beta_q \in \VAr[1]{\uD{}}$  such that
\[
\alpha|_{(q,q')} = \beta_q|_{(q,q')}.
\]
\end{lemma}

\begin{proof}
This follows immediately from the asymptotics of polylogarithms along the real line as in \autoref{thm:disk-moduli-pushforward}, and the fact that $\alpha$ extends continuously to the boundary $(q,q') \subset \buD{}$.  
\end{proof}

\begin{corollary}
For each $q \in \Qinf$, we have the identity
\[
\Reglim{\epsilon}{0} \int_{(q,q')^\epsilon} \alpha 
= %\int_{(q,q')}\beta_q
%\defas
\Reglim{\epsilon}{0} \int_{(q,q')^{\epsilon}}\beta_q \in \Vr[j]{P,Q}.
\]
\end{corollary}

Note that the regularized limit $\Reglim{\epsilon}{0} \int_{(q,q')^\epsilon} \beta_q$ is exactly the definition of the regularized integral of $\beta_q$ along the straight path in the universal curve $\uX{S}$ between the tangential base points $\qt = \cvf{z}|_{q}$ and $-\vec{q'}=-\cvf{z}|_{q'}$.  Hence it lies in $\Vr[j]{P,Q}$, as claimed.

\subsubsection{Contributions from boundary half-circles}
\label{sec:half-circle-contribution}

We now consider the contribution from the half-circle $\arc{q}$ centred at a marked point $q \in \Qinf$.  These too can be computed using the corresponding holomorphic form $\beta_q$, but the proof is rather more involved:

\begin{proposition}\label{prop:half-residue}
We have the identity
\[
\Reglim{\epsilon}{0} \int_{\arc{q}} \alpha =  \int_{\gamma_q}\beta_q \in \Vr[j]{P,Q}
\]
where $\gamma_q$ is a tiny half-loop between opposite tangential base points at $q$.
\end{proposition}

\begin{proof}
We have $\Reglim{\epsilon}{0}\int_{\arc{q}}\beta_q = \int_{\gamma_q} \beta_q$ by a simple contour deformation argument.  Hence it is enough to show that  $\Reglim{\epsilon}{0}\int_{\arc{q}}\lambda = 0$, where $\lambda = \alpha - \beta_q$.

Using polar coordinates $(r,\theta)$ centred at $q$, applying \autoref{prop:disk-log-expansion} to the polylogarithmic coefficients of $\alpha$, and similarly applying the expansion \eqref{eq:hlog-zero-expansion} for holomorphic hyperlogarithms to the coefficients of $\beta_q$, we obtain an expansion
\[
\lambda = \sum_{m=0}^M \sum_{n=0}^N (a_{mn}(\theta) + rb_{mn}(r,\theta) )\log^m r \log^n \theta \rbrac{ \frac{\td r}{r} + \iu \td \theta}
\]
for some functions $a_{mn}$ and $b_{mn}$ that are analytic at $(r,\theta) = (0,0)$.  In these coordinates, the path of integration $\arc{q}$ is given by $r = \epsilon$ and $\theta \in [0,\pi]$, so we immediately find
\[
\Rlim_{\epsilon \to 0} \int_{\arc{q}} \lambda = \sum_{n} \int_0^\pi a_{0n}(\theta) \log^n(\theta) \td \theta.
\]
It therefore suffices to prove that the functions $a_{0n}$ are identically zero.  We will in fact prove that $a_{mn} = 0$ for all $m,n$.

Suppose by way of contradiction that at least one of the functions $a_{mn}$ is nonzero, and let
\[
M' = \max\set{m \in \ZZ_{\ge 0}}{a_{mn} \ne 0 \textrm{ for some }n}.
\]
We may therefore write
\begin{align*}
\lambda &= \sum_{m=0}^{M'}\sum_{n=0}^N a_{mn}\log^m r \log^n\theta \rbrac{r^{-1}\td r  + \iu \td \theta} \\
&\ \ \ +   r \cdot \sum_{m=0}^M\sum_{n=0}^N b_{mn} \log^m r \log^n\theta  \rbrac{r^{-1}\td r + \iu \td \theta}. 
\end{align*}
Consider the integral of $\lambda$ over the boundary of the partial annulus $\W(\epsilon,\Theta)$ shown in \autoref{fig:annulus}, as a function of the inner radius $\epsilon$ and the opening angle $\Theta$.  We assume that the constant outer radius $R$, and the opening angle $\Theta$ lie in some fixed compact neighbourhood of the origin, on which the functions $a_{mn},b_{mn}$ are analytic.

\begin{figure}[t]
\centering
\begin{tikzpicture}
\def\R{4}
\def\r{1.5}
\begin{scope}
\clip  (0,0) -- (0:{\R+1}) -- (70:{\R+1}); 
\draw[fill=lightgray] (0,0) circle (\R);
\draw[fill=white] (0,0) circle (\r);
\end{scope}
\drawsilentvertex{a}{0,0}
\draw[dashed,blue] (-0.5,0) -- ({\R+0.5},0);
\draw[dashed,blue] (70+180:0.5) -- (70:{\R+0.5});
\draw[sxred,very thick,edge] (0:\R) arc (0:70:\R);
\draw[sxred,very thick,edge] (70:\r) arc (70:0:\r);
\draw[sxred,very thick,edge] (0:\r) -- (0:\R);
\draw[sxred,very thick,edge] (70:\R) -- (70:\r);
\draw[sxred] (35:{\R+0.5}) node {$A_R$};
\draw[sxred] (35:{\r+0.5}) node {$A_\epsilon$};
\draw[sxred] ({(\r+\R)/2},-0.3) node {$A_0$};
\draw[sxred] (80:{(\r+\R)/2}) node {$A_\Theta$};
\draw[blue] (0:\R+1) node {$\theta = 0$};
\draw[blue] (70:\R+0.75) node {$\theta = \Theta$};
\draw[thick,dashed,->] (0,0) -- (15:\R);
\draw (20:{0.5*(\r+\R)}) node {$R$};
\draw[thick,dashed,->] (0,0) -- (60:\r);
\draw (45:0.5*\r) node {$\epsilon$};
\end{tikzpicture}%
\caption{The partial annulus $\W(\epsilon,\Theta)$.}\label{fig:annulus}%
\end{figure}
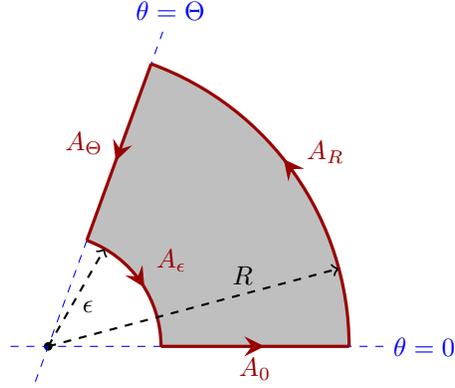

On the one hand, we have $\td \lambda = \td \alpha  = \omega$ since $\beta$ is holomorphic and $\alpha$ is a primitive of our two-form $\omega$.  Therefore
\[
\int_{\partial \W(\epsilon,\Theta)}\lambda = \int_{\W(\epsilon,\Theta)} \omega
\]
by Stokes' theorem.  Note that since $\omega$ is absolutely integrable, the integral is bounded as $\epsilon \to 0$, uniformly in $\Theta$.

On the other hand, we can explicitly compute the integral along the boundary components indicated in \autoref{fig:annulus}.  Firstly, from \autoref{lem:hol-restrict} and the definition of $\lambda$, we know that $\lambda$ is continuous on $\W(\epsilon,\Theta)$, and that it vanishes on the real line segment $\theta = 0$.  It follows that 
\[
\int_{A_0} \lambda = 0.
\]
Moreover, the functions $a_{mn},b_{mn}$ are divisible by $\theta$, and hence the one-forms $a_{mn}\log^n\theta\, \td \theta$ and $b_{mn}\log^n\theta\, \td \theta$ are uniformly bounded in any compact neighbourhood of $(0,0)$.  Therefore
\begin{align*}
\int_{A_\epsilon} \lambda &= \sum_{m=0}^{M'}\sum_{n=0}^N \log^m \epsilon \int_{0}^\Theta a_{mn} \log^n\theta \,\iu\td \theta \\
&\ \ \ +   \epsilon \sum_{m=0}^M\sum_{n=0}^N  \log^m \epsilon \int_{0}^{\Theta}b_{mn} \log^n\theta \, \iu \td \theta \\
&= \asyO{\log^{M'}\epsilon}
\end{align*}
as $\epsilon \to 0$, uniformly in $\Theta$.  Similarly
\[
	\int_{A_R} \lambda = \asyO{1}
\] 
since the outer radius is fixed.  Finally, we have
\begin{align*}
	\int_{A_{\Theta}} \lambda &= \sum_{m=0}^{M'} \sum_{n=0}^N a_{mn}(\Theta)\log^n\Theta \int^\epsilon_R \frac{\log^m r \td r}{r} \\
&\ \ \ + \sum_{m=0}^M \sum_{n=0}^N \log^n\Theta\int_{R}^\epsilon b_{mn}(r,\theta) \log^m r \, \td r \\
&= \frac{\log^{M'+1}\epsilon}{M'+1} \sum_{n=0}^N a_{M'n}(\Theta) \log^n\Theta  + \asyO{\log^{M'}\epsilon}
\end{align*}
Therefore
\[
	\int_{\partial\W(\epsilon,\Theta)} \lambda  = \frac{\log^{M'+1} \epsilon}{M'+1} \sum_{n=0}^N a_{M'n}(\Theta)\log^n\Theta  + \asyO{\log^{M'}\epsilon}
\]
so that the boundedness in the limit $\epsilon \to 0$ implies that 
\[
\sum_{n=0}^N a_{M'n}(\Theta)\log^n\Theta = 0
\]
for all $\Theta$ sufficiently small.  But the functions $\log^n\Theta$ for $0\le n \le N$  are linearly independent over the ring $\CC\{\Theta\}$ of analytic germs at $\Theta = 0$.  Therefore $a_{M'n} = 0$ for all $n$, which contradicts the definition of $M'$.
\end{proof}

\subsubsection{Total contribution of the outer boundary}
\label{sec:weight-drop}

We now complete the proof of \autoref{thm:disk-moduli-pushforward} by proving that the total contribution of the cycle $\outD{}$ has weight $j-1$.

The key point is the following lemma relating the various holomorphic forms $\beta_q$ associated to $\alpha$:

\begin{lemma}
If $q_1,q_2 \in \Qinf$ are boundary marked points, then the difference $\beta_{q_1} - \beta_{q_2}$ between any branches of the corresponding fibrewise holomorphic forms has weight at most $j-1$.
\end{lemma}

\begin{proof}
Note that the segments $(q,q')$ for $q \in \Qinf$ are embedded as paths in the universal curve $\uX{S}$, which is in turn embedded as an irreducible component of the boundary divisor in  $\bM{\Spp}$ given by the equation $z=\zb$.  By definition $\alpha$ is the restriction of a multivalued holomorphic form on $\M{\Spp}$, and from this perspective, the forms $\beta_{q_1}$ and $\beta_{q_2}$ are simply the restrictions of different branches of the same holomorphic one-form on the moduli space to the boundary divisor $\uX{S}\subset \bM{\Spp}$.  Hence the difference $\beta_{q_1}-\beta_{q_2}$ is given by the monodromy of some polylogarithms on $\M{\Spp}$, and in particular it has weight strictly less than $j$.
\end{proof}

Combining this result with \autoref{lem:hol-restrict} and \autoref{prop:half-residue}, we find for any fixed $q_0 \in \Qinf$ that
\begin{align*}
\lim_{\epsilon\to 0} \int_{\outD{}}\alpha &= \Reglim{\epsilon}{0} \sum_{q \in \Qinf} \int_{\intvl{q}+\arc{q}} \alpha \\
&= \Reglim{\epsilon}{0} \sum_{q \in \Qinf} \int_{\intvl{q}+\arc{q}} \beta_q  \\
&= \Reglim{\epsilon}{0} \sum_{q \in \Qinf} \int_{\intvl{q}+\arc{q}} \beta_{q_0} \mod  \Vr[j-1]{P,Q}\\ 
&=\int_{\gamma}\beta_{q_0} \mod \Vr[j-1]{P,Q}  \\
&\in \Vr[j-1]{P,Q}
\end{align*}
where $\gamma_{q_0}$ denotes a loop that is freely homotopic to $\outD{}$ and based at some tangential base point $\qt_0$.  In the final step we have used the weight drop for iterated integrals of holomorphic forms around closed loops (\autoref{lem:loop-weight}).  This completes the proof of \autoref{thm:disk-moduli-pushforward}.

%%%%%%%%%%%%%%%%%%%%%%

\section{Application to formality morphisms}
\label{sec:formality}
\subsection{Weights of admissible graphs}

We now briefly explain how to apply our results to formality morphisms.  Recall from \cite{Kontsevich2003} that an \defn{admissible graph of type $(n,m)$} is a directed graph $\Gamma$ with $2n+m-2$ edges, a set $P$ of ``internal'' vertices and a set $Q$ of ``external'' vertices, where $|P|=n$ and $|Q|=m$. The external vertices are sinks and no parallel edges or self-loops are allowed, but note that this does not exclude anti-parallel edges as in \autoref{prop:zeta3-graph}.

For example, \autoref{fig:h8-graph} in the introduction is an admissible graph of type $(8,2)$; we draw the vertices in $Q$ on a dashed blue line representing the boundary of a disc.
This particular graph has precisely two outgoing edges at each internal vertex, which means that it contributes to the star product. Examples of more general admissible graphs can be found in \autoref{sec:software}.

To an admissible graph, one associates a volume form $\omega_{\Gamma}$ on the moduli space $\C{P,Q}$ as follows.  To each edge $p \to s$ with $p \in P$ and $s \in S = P\sqcup Q$ one associates a one-form $\alpha_{p\to s}$, called the propagator.  Then
\[
	\omega_\Gamma \defas \bigwedge_{\text{edges $e$}} \alpha_e \in \forms[2n+m-2]{\C{P,Q}}.
\]
Note that the sign depends on the additional choice of an ordering of the edges.

There are many different ways of defining propagators~\cite{Alekseev2016,Kontsevich1999,Kontsevich2003,Rossi2014}, but all of the explicitly known ones are captured by a one-parameter family described by Rossi--Willwacher~\cite{Rossi2014}:
\begin{equation}
	\alpha_{p \to s}^t
	\defas \frac{1-t}{2 \pi \iu} \tdlog \crossrat{p}{\pb}{s}{\infty} + \frac{t}{2 \pi \iu} \tdlog \crossrat{p}{\pb}{\overline{s}}{\infty}
	\in \A[1]{\C{P,Q}}[t]
	,
	\label{eq:angle-t}%
\end{equation}
where $t \in \RR$ is a real parameter. Note that this propagator is independent of $t$ whenever $s \in Q$, since then $s = \overline{s}$. 
The degree of $\omega_\Gamma^t$ as a polynomial in $t$ is thus at most $2n-2$, because $\omega_\Gamma^t=0$ if one of the boundary vertices is not attached to any edge.
There is only one exception to this rule, namely the unique admissible graph of type $(0,2)$, which does not have any edges and for which $\omega_{\Gamma}^t=1$.

It is shown in \cite{Rossi2014} that  the corresponding volume integral
\begin{equation}
	c_\Gamma^t \defas \int_{\C{P,Q}} \omega_{\Gamma}^t
	\in \CC[t]
	\label{eq:formality-weight}%
\end{equation}
is absolutely convergent, for any value of $t$. Using \autoref{cor:arnold-integral}, we see that the coefficients of the polynomial $c_\Gamma^t\in \MZVipi[][t]$ are normalized MZVs of weight at most
\[
w(n,m) \defas \begin{cases}
	n+m-2 &	\text{if $m > 0$ and}\\
	n-1 &	\text{if $m = 0$.}
\end{cases}
\]
However, an additional constraint is evident, namely that the family \eqref{eq:angle-t} of propagators transforms under complex conjugation as
$
	\overline{\alpha_e^t}
	=\alpha_e^{1-t}
$.
Therefore the real part of $c_{\Gamma}^t$ is an even function of $t-\tfrac{1}{2}$, while the imaginary part is an odd function of $t-\tfrac{1}{2}$.  More precisely, we have the following
\begin{theorem}\label{thm:formality-weights}
If $\Gamma$ is an admissible graph of type $(n,m)\neq(0,2)$, then
\[
c_\Gamma^t = a(\tmt) + \iu(1-2t)b(\tmt)
\]
where $a,\iu b \in \MZVipi{}[\tmt]$ are polynomials in $\tmt = t(1-t)$ with the following properties:
\begin{itemize}
\item $a$ has degree at most $n-1$ and its coefficients lie in $\Re\MZVipi[w(n,m)] \subset \RR$,
\item $b$ has degree at most $n-2$ and its coefficients lie in $\Im\MZVipi[w(n,m)] \subset \RR$.
\end{itemize}
\end{theorem}
\begin{proof}
Since $t^2=t-\tau$, we have $\MZVipi[][t] = \MZVipi[][\tau] \oplus t\MZVipi[][\tau]$, so that we may write  $c_{\Gamma}^t = p(\tau) + t q(\tau)$ for some polynomials  $p,q\in\MZVipi[][\tau]$. The relation $\overline{c}_{\Gamma}^t=c_{\Gamma}^{1-t}$ implies $\overline{p}=p+q$, so that $q=-2\iu\Im p$ is divisible by two. We then take $b=\iu q/2=\Im p$ and $a=\Re p$.
\end{proof}

\subsection{An example}
\label{sec:example}

To illustrate the integration algorithm developed in \autoref{sec:cycles} and implemented in our software, we shall calculate the formality coefficient associated to a graph of type $(3,2)$. In what follows, we will often represent a differential form $\omega_\Gamma^t$ simply by drawing a picture of the graph $\Gamma$, and we list an order of the edges to specify the sign.
With this notation understood, the result of this section is as follows:
\begin{proposition}\label{prop:zeta3-graph}
The weight of the graph with edges $yx,y0,xy,x1,zx,z1$ is
\[
	\bigintsss_{\displaystyle\C{\sset{x,y,z},\sset{0,1}}}\rbrac{\zetagraph[0.7]}
	= (1-2t) \frac{\mzv{3}}{(2\ipi)^3}.
\]
\end{proposition}
We shall break the calculation into subtasks in which we calculate the pushforward along the maps that forget the marked points $z$, $y$ and $x$ in turn.

% compute the pushforwards of subgraphs along various projections.  , illustrating the general algorithm as developed in \autoref{sec:cycles} and implemented in our software. %We remark that the computation below is not the most efficient route to the result for this specific graph.  For instance, once can eliminate the contributions of most boundary components by making a more clever choice of primitives when applying Stokes' theorem.  Since our purpose here is to give explicit examples of the general algorithm developed in \autoref{sec:cycles} and implemented in our software, we refrain from such shortcuts.

\subsubsection{Pushforward of a wedge (forgetting $z$)}
\label{sec:wedge-integral}

In this section we will illustrate the algorithm in detail by proving
\begin{lemma}\label{lem:wedge-integral}
	For the projection $g \colon \C{P\cup \sset{z},Q} \to \C{P,Q}$ that forgets $z$, we have
\begin{equation}
g_{\ast}\rbrac{ \freewedgegraph[0.6] }
= g_{\ast} \left( \alpha^t_{z\rightarrow x} \wedge \alpha^t_{z\rightarrow y} \right)
= \frac{1}{2 \ipi} \log \frac{x-\yb}{y-\xb} \label{eq:example1}
\end{equation}
for any internal vertices $x,y\in P$.
The result is independent of $t$, and in the limits when $x \to p$ and $y \to q>p$ land on the boundary it becomes
\begin{equation}
g_{\ast}\rbrac{\wedgegraph[0.8]}
= \frac{1}{2 \ipi} \log \frac{x-q}{q-\xb}
\quad\text{and}\quad
g_{\ast}\rbrac{\pqwedgegraph[0.8]} = \frac{1}{2}.
\label{eq:pushwedge}%
\end{equation} 
\end{lemma}
Observe that for $x,y \in \HH$ the quantity $\frac{x-\yb}{y-\xb}\in \CC\setminus(-\infty,0]$ avoids the branch cut of the logarithm, so \eqref{eq:example1} is single-valued on $\C{P,Q}$.
As $(x,y)\rightarrow (p,q)$ with $p <q$, the quantity $\frac{x-\yb}{y-\xb}$ approaches $-1$ from the upper half-plane.  Hence the logarithm approaches $+\ipi$, giving \eqref{eq:pushwedge}.

Our \autoref{thm:disk-moduli-pushforward} almost completely fixes the solution without any calculation.  Indeed, since the pushforward is a polylogarithm of weight $1$ on $\C{P,Q}$ that depends only on the positions of $x$ and $y$, it is a linear combination of logarithms of cross ratios on $\M{\sset{x,\xb,y,\yb}}$. But it must  also be anti-symmetric under swapping $x\leftrightarrow y$, so it is necessarily a $\ZZ[t]$-linear combination of
\begin{equation*}
	\frac{1}{2\ipi} \log \frac{x-\yb}{y-\xb},\qquad
	\frac{1}{2\ipi} \log \frac{x-y}{\xb-\yb}\qquad\text{and}\qquad
	\frac{1}{2\ipi} \log \frac{x-\xb}{y-\yb}.
\end{equation*}
The second of these functions is not single-valued, and the third  is divergent when $x$ or $y$ approach the boundary. Hence the result must be a multiple of the first function, and merely its coefficient $\rho\in\ZZ[t]$ remains to be determined. This can be read off from the limit  $(x,y)\rightarrow(p,q)$. But in this limit the propagators \eqref{eq:angle-t} become independent of $t$. Hence $\rho\in\ZZ$ and we are free to fix the value of $t$ arbitrarily. We choose the harmonic propagator $t=1/2$.

Let us now explicitly follow the algorithm to compute the pushforward of
\[
\omega \defas 2\rbrac{\wedgegraph[0.8]}
= 2\alpha_{z\to x}^{1/2} \wedge \alpha_{z \to q}^{1/2},
\]
in other words, compute the integral of $\omega$ over $z\in \HH\setminus \sset{x}$ for fixed $x$ and $q$. The algorithm of \autoref{sec:forget-interior} proceeds as follows:
\begin{enumerate}
\item \textbf{Write the restriction of $\omega$ to the fibres in the natural basis of relative forms,} i.e.\ as a linear combination of the forms $\dlog{s} \wedge \bdlog{s'}$, where $\dlog{s} = \frac{\td z}{z-s}$ and $\bdlog{s'} = \frac{\td \zb}{\zb-s'}$ for $s \in S = \sset{p,\pb,q}$ and $s' \in S \cup \sset{z}$:
\begin{align*}
	\omega 
%	&= 2\alpha_{z\to x} \wedge \alpha_{z \to q} \\
%	&\equiv 2 \cdot \frac{1}{2}(\dlog{x}+\dlog{\xb}-\bdlog{x}-\bdlog{\xb}) \wedge \frac{1}{2}(2\dlog{q}-2\bdlog{q})
	&\equiv (\dlog{x}+\dlog{\xb}-\bdlog{x}-\bdlog{\xb}) \wedge (\dlog{q}-\bdlog{q})
	= \bdlog{q} \wedge (\dlog{x}+\dlog{\xb}) - (\bdlog{x}+\bdlog{\xb}) \wedge \dlog{q}. 
\end{align*}
Here $\equiv$ denotes equivalence modulo the ideal generated by $g^*\forms[1]{\C{P,Q}}$.

\item \textbf{Construct a possibly multivalued (1,0)-form $\talpha$ that is a fibrewise primitive for $\omega$.} 
We use the simple recipe~\eqref{eq:hyperlog-primitive} to construct primitives of hyperlogarithmic one-forms by integration with respect to $\zb$. This gives
\begin{equation*}
\talpha \defas \bhlog{q}(\dlog{x}+\dlog{\xb}) - (\bhlog{x}+\bhlog{\xb})\dlog{q}
\end{equation*}
where $\bhlog{s} = \frac{1}{2\ipi}\log(\zb -s)$ denotes the hyperlogarithm based at $\infty$ obtained by integrating $\bdlog{s}$. %In particular, $2\ipi\bhlog{s}=\log \zb + \log (1-s/\zb)$ vanishes at $\zb=\infty$ up to $\log\zb$ in accordance with our definition of regularized integrals.

\item \textbf{Cancel the monodromy of $\talpha$ to get a single-valued primitive $\alpha$.}
Recall the recipe in the proof of \autoref{lem:1-0-primitive}: since the monodromy $\delta \talpha$ is holomorphic, we can use \autoref{rem:monodromy-cancel} to produce a fibrewise holomorphic form $\talpha'$ with the same monodromy.  (See also \autoref{ex:disk-dilog}.)  The single-valued primitive is then given by
\begin{equation*}
\alpha = \talpha - \talpha' = \talpha - (1+h\delta')^{-1}h\delta\talpha
\end{equation*}
where $\delta'$ and $h$ are the operators on holomorphic hyperlogarithms defined in \autoref{sec:hyperlog-cohlgy}. This construction makes heavy use of the path concatenation formula, and in particular the recipe in \autoref{lem:loop-weight} to compute integrals around closed loops.

In our example, we only have a single monodromy around $x$ to consider. Because the rational forms $\dlog{s}$ in $\talpha$ are single-valued, we need only operate on their hyperlogarithmic coefficients. For these logarithms, we simply get
\[
\delta \bhlog{s} = \int_{\gamma_x}\bdlog{s} = \begin{cases}
	-1 & \text{when $s=\xb$ and} \\
	0 & \text{for $s \ne \xb$,}
\end{cases}
\]
so that $\delta\talpha = \dlog{q}$. The holomorphic form $h\delta\talpha = (1+h\delta')^{-1}h\delta\talpha = \hlog{x}\dlog{q}$ has exactly this monodromy, giving the single-valued primitive % ; in this case the operator $\delta'$ with relative weight $-2$ from \autoref{rem:monodromy-cancel} annihilates $h\delta\talpha$, which has weight one. 
\begin{align*}
\alpha% & = \alpha' - (1-h\delta')^{-1}(\hlog{x}\dlog{q})  \\
&= \alpha' -\hlog{x}\dlog{q}
=  \bhlog{q}(\dlog{x}+\dlog{\xb}) - \rbrac{\bhlog{x}+\bhlog{\xb}+\hlog{x}}\dlog{q}.
\end{align*}

\item \textbf{Compute the residues of $\alpha$ at the interior points,} as in \autoref{sec:inner-contribution}. This is done by specializing $(z,\zb) = (x,\xb)$ in the relevant coefficients, in our case for $\dlog{x}$:
\begin{equation*}
2 \ipi \Res_x \alpha 
= \Reglim{z}{x} \Reglim{\zb}{\xb} \bhlog{q}(\zb)
= \bhlog{q}(\xb).
\end{equation*}

\item \textbf{Express the restriction of $\alpha$ to the outer boundary components in terms of holomorphic hyperlogarithms.} This is the computation of the limit $z \to \zb$ in appropriate branches. 
\begin{figure}%
	\centering%
	\includegraphics[width=5cm]{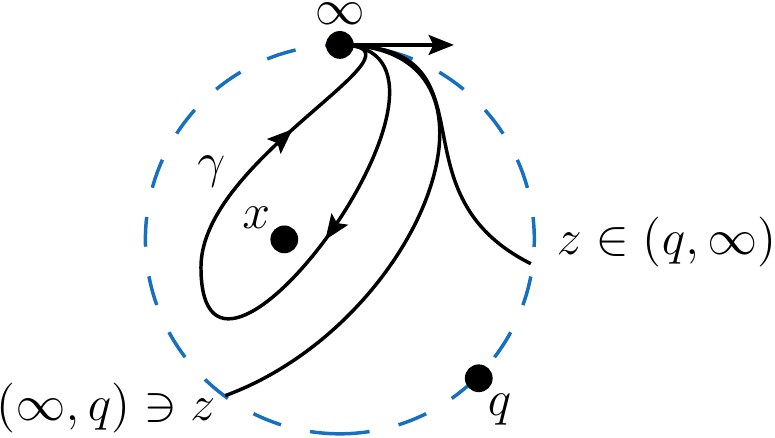}%
	\caption{Our choice of paths, starting at $\infty$, for the hyperlogarithm realization.}%
	\label{fig:example-paths}%
\end{figure}
In our example, the boundary component $(q,\infty)$ ends at our chosen tangential base point for hyperlogarithms. Thus the holomorphic and antiholomorphic hyperlogarithms agree on this component, and we have
\[
\beta_q =  \hlog{q}(\dlog{x}+\dlog{\xb}) - (2\hlog{x}+\hlog{\xb})\dlog{q}.
\]
On the other hand, the analytic continuation of $L_q(\zb)$ and $L_q(z)$ through the disk to the boundary component $(\infty,q)$ differ by the integral of $\dlog{q}-\bdlog{q}$ along a path between opposite tangential base points at $q$,
\[
	\bhlog{q}|_{(\infty,q)}=\left.\frac{\log(\zb-q)}{2\ipi}\right|_{\zb=z<q}
	= \frac{\log \abs{z-q}-\ipi}{2\ipi} = \hlog{q}|_{(\infty,q)}-1,
\]
for the path of integration indicated in \autoref{fig:example-paths}. It follows that
\[
\beta_\infty =(\hlog{q}-1)(\dlog{x}+\dlog{\xb}) -(2\hlog{x}+\hlog{\xb})\dlog{q}.
\]

\item \textbf{Integrate along the boundary,} summing up the contributions as in \autoref{sec:weight-drop}.  Using \eqref{eq:hyperlog-primitive}, we get primitives
\[
f_q = \hlog{xq}+\hlog{\xb q}-2\hlog{qx}-\hlog{q\xb}
\quad\text{and}\quad
f_\infty = f_q - \hlog{x}-\hlog{\xb}
\]
of $\beta_q$ and $\beta_{\infty}$ that have vanishing regularized at the tangential base point.  We thus find
\begin{align*}
\int_{\partial \uD{}}\alpha 
&= \Reglim{\epsilon}{0} \int_{\arc{q}+(q,\infty)^{\epsilon}} (\beta_q-\beta_{\infty}) + \int_{\partial\uD{}} \beta_{\infty} \\
&= (f_\infty-f_q)(q)- f_\infty(\gamma) \\
&= -\hlog{x}(q)-\hlog{\xb}(q) + \int_{\gamma}(2qx+q\xb-xq-\xb q +x + \xb)
\end{align*}
where $\gamma$ is any clockwise loop at $\infty$ that is homotopic to the boundary $\partial\uD{}$, see \autoref{fig:example-paths}. To compute the period of this loop, we follow \autoref{lem:loop-weight} and write $\gamma$ as the conjugation $\eta^{-1}\PathConc \gamma_x\PathConc \eta$ of a small clockwise loop $\gamma_x$ at $x$, by a path $\eta$ from $\infty$ to $x$. Applying the path concatenation formula, the only terms that contribute are the ones involving $\int_{\gamma_x} x = -1$, giving
\begin{align*}
\int_{\gamma}(2qx+q\xb-xq-\xb q+x+\xb) &= 2\int_{\eta^{-1}}q \int_{\gamma_x}x - \int_{\gamma_x} x\int_\eta q + \int_{\gamma_x} x  \\
&= 2\hlog{q}(x) + \hlog{q}(x) -1.
\end{align*}

\item \textbf{Add up the residue and boundary contributions.}
In our example,
\begin{align*}
\int_{\uD{}}\omega 
&= -2\ipi \Res_{x}\alpha + \int_{\partial\uD{}}\alpha 
=-\bhlog{q}(\xb)-\hlog{x}(q)-\hlog{\xb}(q)+3\hlog{q}(x)-1 \\
&%= 2\hlog{q}(x)-2\bhlog{q}(\xb)-1
= \frac{\log (x-q)-\log(\xb-q) -\ipi}{\ipi}
= \frac{1}{\ipi} \log \frac{x-q}{q-\xb}
\end{align*}
where $\log$ (as always) denotes the principal branch and we used the identity
\[
\hlog{x}(q)+\hlog{\xb}(q)
= \tfrac{1}{2\ipi}\log(q-x)(q-\xb) 
= \tfrac{1}{2\ipi}\log(x-q)(\xb-q)
= \hlog{q}(x)+\bhlog{q}(\xb)
\]
to write the expression in a basis of hyperlogarithms. (This can be done algorithmically following the recipe in the proof of \autoref{prop:fibration-basis}.) 
\end{enumerate}

\begin{remark}
	For the logarithmic propagator $t=0$, \autoref{lem:wedge-integral} reads
	\begin{equation*}
%		\int_{\HH\setminus\sset{x}} \td \log \frac{z-x}{\zb-x} \wedge \td \log \frac{z-y}{\zb-y}
		\int_{\HH\setminus\sset{x}} 
		\left(\frac{\td z}{z-x}-\frac{\td \zb}{\zb-x}\right)
		\wedge
		\left(\frac{\td z}{z-y}-\frac{\td \zb}{\zb-y}\right)
		= 2\ipi \log \frac{x-\yb}{y-\xb}.
	\end{equation*}
	Note that the integrand is a holomorphic function of $x$ and $y$, but the result is only real analytic. In particular, the integration does not commute with partial derivatives with respect to $\xb$ or $\yb$. This non-holomorphicity is possible because the integrand does not extend to $\bHH$ due to the singularities at $z\in\sset{x,y,\infty}$. \qed
\end{remark}

\subsubsection{Pushforward of a tadpole (forgetting $y$)}

\begin{lemma}\label{lem:bubble}
	For the projection $g \colon \C{P\cup\sset{y},Q} \to \C{P,Q}$ which forgets $y$ and the graph with the three edges $yx,y0,xy$, we have
\begin{equation}
g_{\ast}\rbrac{\loopgraph[0.8]}
 = \left( \frac{t}{2\ipi} \log \frac{x}{-\xb} - \frac{1-2t}{2\ipi} \log \frac{x-\xb}{x} \right)
 \cdot \rbrac{\singleedgeL[0.8]}.
 \label{eq:pushloop}%
\end{equation}
\end{lemma}
\begin{remark}
Note that for $t\neq \tfrac{1}{2}$, the coefficient of this form has a logarithmic divergence on the real line $x = \xb$.  However, this divergence is compensated by the vanishing of $\alpha_{x\rightarrow 0}^t = \tdlog \tfrac{x}{\xb}$, so that the form is still continuous up to the boundary of $\C{P,Q}$.  This demonstrates a general phenomenon: even if we start with a form that is analytic on the integration domain, the integral over the fibres may produce a form with logarithmic divergences on the boundary. \qed
\end{remark}
The computation of this example is very similar to \autoref{sec:wedge-integral}, so we skip the details. For $t=1$, the integrand decomposes into vertical and horizontal components
\[
	\omega_{\Gamma}^{t=1} =  \omega_1 \wedge \frac{1}{2\ipi} \tdlog \frac{x-\xb}{\xb} + \omega_2 \wedge \frac{1}{2\ipi}\tdlog \frac{x}{\xb}
\]
where the vertical components are given by
\begin{equation*}
	\omega_1 \equiv  (\bdlog{x}-\bdlog{\xb}) \wedge \dlog{0} \quad\text{and}\quad
	\omega_2 \equiv  (\bdlog{0}-\bdlog{x}) \wedge \dlog{\xb}.
\end{equation*}
The pushforward of $\omega_1$ turns out to vanish, and the single-valued primitive
\[
	\alpha = \talpha = (\bhlog{0}-\bhlog{x}) \dlog{\xb} = -\frac{1}{2\ipi}\log\left( 1-\frac{x}{\yb} \right)\dlog{\xb}
\]
of $\omega_2$ has no residues in $y\in\HH$ and vanishes at $y=\infty$. So the only contribution comes from the real line $y=\yb$, where the primitive is holomorphic. Thus
\begin{equation*}
	g_{\ast}\omega_2
	= \frac{1}{4\pi^2} \int_{\RR\setminus\sset{0}} \log \left(1-\frac{x}{y} \right) \frac{\td y}{y-\xb}
	=\frac{1}{2\ipi} \log \frac{\xb-x}{\xb}
\end{equation*}
by the residue theorem at $y=\xb$ since we can close the contour of integration in the lower half-plane where the logarithm is single-valued.
This result already implies \eqref{eq:pushloop} by considering the form of the real and imaginary parts of $g_*\omega^t_\Gamma$ as a polynomial in $t$ (similar to \autoref{thm:formality-weights}).

\subsubsection{Proof of the proposition (forgetting $x$)}
The previous results can be combined to establish \autoref{prop:zeta3-graph}, because the differential form associated to the graph under consideration factorizes as
\begin{equation*}
\zetagraph  =  \loopgraph \wedge \singleedgeR \wedge \wedgegraphR
\end{equation*}
such that we can compute the pushforwards that integrate over $y$ and $z$ independently using \autoref{lem:wedge-integral} and \autoref{lem:bubble}, respectively.  Therefore we have
\begin{equation*}
	\bigintsss_{\displaystyle\C{\sset{x,y,z},\sset{0,1}}}\rbrac{\zetagraph[0.7]}
	= \bigintsss_{\displaystyle\C{\sset{x},\sset{0,1}}} \mu \wedge \rbrac{\singleedgeR[0.6]} \wedge \nu
\end{equation*}
where $\mu$ and $\nu$ denote the forms from \eqref{eq:pushloop} and \eqref{eq:pushwedge}.
We compute the result for $t=0$ and note that this implies the full \autoref{prop:zeta3-graph}, since the integrand is linear in $t$, via \autoref{thm:formality-weights}. Using a tangential base point at $0$, we have
\begin{equation*}
	\mu|_{t=0} 
	= -\frac{1}{2\ipi} \log \left( 1-\frac{\xb}{x} \right) \td \log \frac{x}{\xb}
	= -\left( \int_0^{\xb} \dlog{x} \right) (\dlog{0}-\bdlog{0})
	= \bhlog{x} (\bdlog{0}-\dlog{0})
\end{equation*}
Similarly, with this base point, \eqref{eq:pushwedge} becomes $\nu=\hlog{1}-\bhlog{1}+\frac{1}{2}$ and the shuffle product $\bhlog{x} \cdot \bhlog{1} = \bhlog{1x} + \bhlog{x1}$ gives
\begin{equation*}
	\mu|_{t=0} \wedge \rbrac{\singleedgeR[0.6]} \wedge \nu
	= \left(\bhlog{x}\hlog{1}+\frac{1}{2}\bhlog{x} - \bhlog{1x}-\bhlog{x1} \right)
	\left( \bdlog{0} \wedge \dlog{1}-\bdlog{1} \wedge \dlog{0} \right).
\end{equation*}
Since $\C{1,2} \cong \HH$ is simply connected, there are no sources for monodromies and
\begin{equation*}
	\alpha=
	\left( \bhlog{0x}\hlog{1}-\bhlog{0x1}-\bhlog{01x}+\frac{1}{2}\bhlog{0x} \right) \dlog{1}
	-
	\left( \bhlog{1x}\hlog{1}-\bhlog{1x1}-\bhlog{11x}+\frac{1}{2}\bhlog{1x} \right) \dlog{0}
\end{equation*}
is automatically a single-valued primitive. The limits at $\xb\rightarrow x\in (0,1)$, like
\begin{equation*}
	\bhlog{1x}
	\rightarrow \lim_{\xb \rightarrow x}
	  \int_0^{\xb} [\dlog{1}|\dlog{x}]
	= \int_0^x [\dlog{1}|\dlog{x}]
	= \hlog{11} - \hlog{01}
	= \int_0^x [\dlog{1}-\dlog{0}|\dlog{1}],
\end{equation*}
produce periods of $\M{\sset{0,1,x}}$ that we rewrite as hyperlogarithms in $x$ with \autoref{prop:fibration-basis}. This algorithm is illustrated in \autoref{ex:fib-basis} in the case of $\bhlog{0x1}$ and is implemented as $\texttt{fibrationBasis}$ in the program {\HyperInt}, see \cite{Panzer:PhD,Panzer:HyperInt}. For the other hyperlogarithms of $\xb$ we find the limits
\begin{gather*}
	\bhlog{0x}  \rightarrow \tfrac{1}{24},  \qquad
	\bhlog{1x1} \rightarrow \hlog{111}-\hlog{101}, \qquad
	\bhlog{0x1} \rightarrow \hlog{001}-\hlog{101}+\tfrac{1}{24}\hlog{1},\\
	\bhlog{11x} \rightarrow \hlog{111} - \hlog{011} \qquad\text{and}\qquad
	\bhlog{01x} \rightarrow \hlog{101}+\hlog{011}-2\hlog{001}-\tfrac{1}{24}\hlog{1}.
\end{gather*}
%With the base point at zero, holomorphic and antiholomorphic polylogarithms agree on the interval $(0,1)$, so we obtain
%\begin{equation*}
%	\alpha|_{(0,1)}
%	= \beta_{(0,1)}
%	= \left( \hlog{11x}+\frac{1}{2}\hlog{1x} \right) \dlog{0}
%	- \left( \hlog{10x}+\frac{1}{2}\hlog{0x} \right) \dlog{1}
%\end{equation*}
%where we expanded $\hlog{1x}\hlog{1}=2\hlog{11x} + \hlog{1x1}$. We can now rewrite these periods as hyperlogarithms of $x$, using the algorithm behind \autoref{prop:fibration-basis}. This is implemented as $\texttt{fibrationBasis}$ in {\HyperInt}, see \cite{Panzer:PhD,Panzer:HyperInt}, and gives explicitly
%\begin{align*}
%	\hlog{1x}  &= \hlog{11} - \hlog{01}, &
%	\hlog{0x}  &= \tfrac{1}{24},  \\
%	\hlog{11x} &= \hlog{111} - \hlog{011}, &
%	\hlog{10x} &= \hlog{001}-\hlog{011}+\tfrac{1}{24} \hlog{1}.
%\end{align*}
The holomorphic form $\beta_0$ that agrees with $\alpha$ on $(0,1)$ is therefore given by
\begin{equation*}
	\beta_0
	=
	\left(\hlog{001}-\hlog{011}+ \frac{1}{24} \hlog{1} + \frac{1}{48} \right) \dlog{1}
	+\left(\hlog{011}-\hlog{111} + \frac{1}{2}\hlog{01}-\frac{1}{2}\hlog{11} \right)\dlog{0},
\end{equation*}
a one-form with holomorphic hyperlogarithm coefficients.
Its regularized integral over $(0,1)$ gives MZVs according to \autoref{thm:M03-periods}. Using shuffle regularization as illustrated for $\hlog{1001}$ in \autoref{ex:reg-MZV}, we calculate that
\begin{align*}
\int_{(0,1)}\beta_0
	&= \Reglim{x}{1} \left( 
		\hlog{1001}-\hlog{1011}+\hlog{0011}-\hlog{0111}
		+\frac{\hlog{001}-\hlog{011}}{2}
		+\frac{\hlog{11}}{24}
		+\frac{\hlog{1}}{48}
	\right) \\
	&= -\frac{1}{480} - \frac{\mzv{3}}{(2\ipi)^3} \in \MZVipi[4].
\end{align*}
According to \autoref{tab:period-gens}, this period is does \emph{not} lie in $\MZVipi[3]$, but when we add it to the remaining boundary contributions
\begin{align*}
	\Reglim{\epsilon}{0}\int_{(\infty,0)^{\epsilon}} \alpha
	&= \frac{1}{1440} + \frac{\mzv{3}}{2(2\ipi)^3}, &
	\Reglim{\epsilon}{0} \int_{\arc{1}} \alpha
	&= -\frac{1}{192} + \frac{\mzv{3}}{(2\ipi)^3}, \\
	\Reglim{\epsilon}{0}\int_{(1,\infty)^{\epsilon}} \alpha
	&= \frac{19}{2880} + \frac{\mzv{3}}{2(2\ipi)^3} \quad\text{and}&
	\Reglim{\epsilon}{0} \int_{\arc{0}} \alpha
	&=
	\Reglim{\epsilon}{\infty} \int_{\arc{\infty}} \alpha
	= 0
\end{align*}
we find that the total is indeed $\mzv{3}/(2\ipi)^3 \in \MZVipi[3]$, which proves \autoref{prop:zeta3-graph} and illustrates concretely the weight drop established in \autoref{sec:weight-drop}.

%%%%%%%%%%%%%%%%%%%%%%

\section{Software}
\label{sec:software}

\subsection{Overview}

We now give a brief overview of the software package mentioned in the introduction, which we recall is available at the following URL:
\begin{quote}
\DQcode
\end{quote}
This package has three main components:
\begin{itemize}
\item \Maple\ code, written by the second author, which implements the integration algorithm of \autoref{sec:cycles} for the exact symbolic calculation of volume integrals on $\C{P,Q}$ in terms of MZVs.  This code is also available as a separate module at
\begin{quote}
\Kontsevint
\end{quote}
It builds on the \HyperInt\ package~\cite{Panzer:HyperInt,Panzer:PhD} for calculations with hyperlogarithms.

\item A database listing the values of the coefficients $\KW{\Gamma}$ appearing in the star product for the harmonic, logarithmic and interpolating formality morphisms, up to order $\hbar^6$.  This database was produced by generating a complete list of isomorphism classes of such graphs using \nauty~\cite{McKayPiperno:II}, and directly computing the corresponding integrals using the \Maple\ code.

\item \Sage\ code, written by the first author as an undergraduate summer project, which manages the combinatorial part of the star product.  Using the database of weights, it allows for the exact symbolic calculation of the terms in the star product for an arbitrary Poisson structure on $\RR^n$.
\end{itemize}
We refer the reader to the online documentation for instructions on the use of this package, including a tutorial that illustrates its core functionality.

\subsection{Consistency checks}
\label{sec:consistent}

We checked that the output of our software is consistent with several known properties of the formality coefficients $\KW{\Gamma}$.  We summarize these tests as follows.

\paragraph{Associativity:} The associativity of star products imposes an infinite sequence of quadratic relations between the coefficients.  This fact was used in \cite{Buring2017} to produce software that reduces the calculation of all coefficients in the star product to a smaller list of unknown ones.  In particular, the calculation of all coefficients appearing the harmonic star product up to order $\hbar^4$ was reduced to linear combinations of 10 undetermined parameters in \cite[Table~8]{Buring2017}.  We checked that the coefficients in our database satisfy all of these constraints.  Furthermore, our results agree with the 1640 additional linear relations between harmonic coefficients at order $\hbar^5$ that were communicated to us by the authors of \cite{Buring2017}.

\paragraph{HKR graphs:} For each $m \ge 1$ there is a unique graph $\Tree{m}$ of type $(1,m)$.  These graphs have the following form:
	\begin{center}
		\HKRgraph{1}
		\HKRgraph{2}
		\HKRgraph{3}
		\HKRgraph{4}
		$\ldots$
	\end{center}
If we order the edges of $\Tree{m}$ from left to right, its coefficient is given by
\[
c_{\Tree{m}}^t=\frac{1}{m!}.
\]
As explained in \cite[Section~6.4.3]{Kontsevich2003}, this identity ensures that the linear part of the formality morphism gives the classical Hochschild--Kostant--Rosenberg (HKR) isomorphism between polyvector fields and Hochschild cohomology.  We checked that our software reproduces this result for all $m\leq 10$. 

\paragraph{Bernoulli graphs:} There is a sequence of graphs $\BerG{n}$ of type $(n,2)$ for $n \ge 1$ that have coefficients $c_{\BerG{n}}^t = (-1)^n B_n/n!$, where $B_n$ is the $n$th Bernoulli number; see \cite{Arnal:LieNilpotente} and \cite[Proposition~4.1]{Kathotia2000}. (While these papers state and prove the result only for $t=1/2$, it is not difficult to show that it persists for all $t$.) We checked that our software reproduces this result for all $n\leq 12$.

\paragraph{The Felder--Willwacher graph:} The paper \cite{Felder2010} exhibits a graph $\Gamma$ of type $(7,2)$ such that $\KW{\Gamma}^{1/2} = a + b\tfrac{\zeta(3)^2}{\pi^6}$ for some undetermined constants $a,b\in\QQ$ with $b \ne 0$.  Our software reproduces this result, and gives the exact values
\begin{equation*}
a = \frac{13}{2^{10}\cdot 3^4\cdot 5\cdot 7}  
\qquad\text{and}\qquad
b = \frac{1}{2^8}.
\end{equation*}

\paragraph{Wheels:}
For $n \ge 2$, consider the wheels $\WheelIn{n}$ of type $(n+1,0)$ with $n \ge 2$ inward pointing spokes:
	\begin{center}
		\inwheelgraph{2}
		\inwheelgraph{3}
		\inwheelgraph{4}
		\inwheelgraph{5} $\ldots$
	\end{center}
the wheels $\WheelOut{n}$ of type $(n+1,0)$ with $n \ge 2$ outward pointing spokes:
	\begin{center}
		\outwheelgraph{2}
		\outwheelgraph{3}
		\outwheelgraph{4}
		\outwheelgraph{5} $\ldots$
	\end{center}
and the graphs $\Amar{n}{n}$ of type $(n,2)$ obtained by replacing the centre of the inward wheel with an external vertex, and splitting off one of the spokes:
	\begin{center}
		\wheeldiffgraph{2}
		\wheeldiffgraph{3}
		\wheeldiffgraph{4}
		\wheeldiffgraph{5} $\ldots$
	\end{center}
To order the edges of the wheels, go around the rims (starting anywhere), and at each vertex, take first the outgoing edge along the rim, then the spoke. To order the edges for the graphs $\Amar{n}{n}$, take first the edge going to the boundary, then the outgoing edge along the rim (see \autoref{fig:Amar}).
The following facts are known:
\begin{enumerate}
\item We have the identity
\begin{equation}
 c^t_{\WheelIn{n}} + c_{\Amar{n}{n}^{}}^t =  c^t_{\WheelOut{n}}
\label{eq:Duflo-Curve}%
\end{equation}
as follows from the equivalence of the characteristic functions $f^{\textrm{Duflo}}$ and $f^{\textrm{curv}}$ established in \cite[Theorem 1]{Willwacher:CharacteristicClasses}.
\item The coefficients of the outward pointing wheels for the harmonic~\cite{VandenBergh2009} and logarithmic propagators~\cite[Appendix A]{Merkulov:ExoticSchouten} are given by:
\begin{align*}
c_{\WheelOut{n}}^{1/2} &= - \frac{B_n}{2 \cdot n!}
& c_{\WheelOut{n}}^0 &= \frac{\mzv{n}}{(2 \ipi)^n}
\end{align*}
where $B_n$ denotes the $n$th Bernoulli number.
Moreover, when $n$ is even, the coefficient $c_{\WheelOut{n}}^t$ is independent of $t$, and is in fact the same for all stable formality morphisms~\cite[Remark 3]{Willwacher:CharacteristicClasses}.
\item The coefficients of the inward pointing wheels $\WheelIn{n}$ are known to vanish when $t = 0$ \cite[Lemma~1.3]{Rossi:StdLogII} or when $t=1/2$ and $n$ is even~\cite{Shoiket:VanishingWheels}.  This ensures that the star product obtained by quantizing the linear Poisson bracket on the dual of a Lie algebra restricts to the usual product on the subalgebra of invariant polynomials.
\item The harmonic coefficients of the graphs $\Amar{n}{n}$ were calculated in \cite[Corollary~6.3]{BenAmar:RieffelKontsevich} and found to coincide with the outward pointing wheels:
\begin{equation*}
	c_{\Amar{n}{n}}^{1/2} = -\frac{B_n}{2\cdot n!}.
\end{equation*}
\end{enumerate}

We computed the coefficients of the graphs $\WheelIn{n}$, $\WheelOut{n}$ and $\Amar{n}{n}$ with our software for $n \le 7$. The results are consistent with the known facts above, and suggest the following formulas for the full $t$-dependence of the wheel coefficients:

\begin{conjecture}\label{con:wheels-t}
For Rossi--Willwacher's $t$-family of formality morphisms, the coefficients of the even inward pointing wheels vanish, that is
\[
c_{\WheelIn{2k}}^t = 0  
\]
for all $k \ge 1$. The odd wheels with $2k+1\geq 3$ spokes have coefficients
	\begin{align*}
		c_{\WheelOut{2k+1}}^t
		&
		= \frac{\mzv{2k+1}}{(2\ipi)^{2k+1}} (1-2t)
		\sum_{j=0}^{2k} \binom{2j}{j} (t(1-t))^j
		\quad\text{and}
		\\
		c_{\WheelIn{2k+1}}^t
		&
		= \frac{\mzv{2k+1}}{(2\ipi)^{2k+1}} (1-2t)
		\sum_{j=k+1}^{2k} \binom{2j}{j} (t(1-t))^j.
	\end{align*}
\end{conjecture}

\begin{figure}
	\begin{equation*}
		\Amar{4}{4} = \vcenter{\hbox{\includegraphics[scale=0.7]{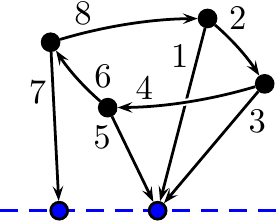}}}
		\qquad
		\Amar{7}{4} = \vcenter{\hbox{\includegraphics[scale=0.7]{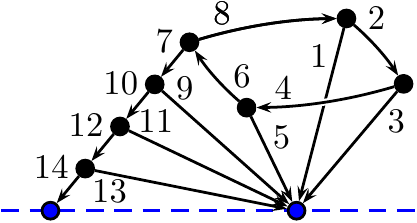}}}
	\end{equation*}%
	\caption{Two examples of the family of graphs $\Amar{n}{k}$ defined in \cite{BenAmar:RieffelKontsevich}, with an explicit order of the edges that represents our sign convention.}%
	\label{fig:Amar}%
\end{figure}
\paragraph{Amar family:}
By adding $n-k$ wedges to the graphs $\Amar{k}{k}$ discussed above, one obtains a family of graphs $\Amar{n}{k}$ of type $(n,2)$ with $n\geq k \geq 2$, as illustrated in \autoref{fig:Amar}. This family was studied in \cite{BenAmar:RieffelKontsevich} and we computed the $t$-dependent coefficients for all such graphs with $n\leq 8$. The results suggest the following
\begin{conjecture}\label{con:amar-t}
	For $n\geq k\geq 2$, the coefficient of $\Amar{n}{k}$ is
	\begin{equation*}
		c_{\Gamma_{n,k}}^{t}
		= \begin{cases}
			\frac{\mzv{n}}{(2\ipi)^n} 
			= - \frac{B_n}{2\cdot n!}

			& \text{if $n$ is even, and}
		\\
			\frac{\mzv{n}}{(2\ipi)^{n}} 
			(1-2t)
			\sum_{j=0}^{\lfloor\frac{k-1}{2}\rfloor} \binom{2j}{j} (t(1-t))^j 
			& \text{if $n$ is odd.}\\
		\end{cases}
	\end{equation*}
\end{conjecture}
Note that this conjecture is compatible with the known facts listed above. However, it would imply that
\begin{equation*}
	c_{\Amar{n}{k}}^0 = \frac{\mzv{n}}{(2\ipi)^n}
	\quad\text{and}\quad
	c_{\Amar{n}{k}}^{1/2} = - \frac{B_n}{2\cdot n!}
\end{equation*}
are independent of $k$, which contradicts \cite[Conjecture~6.4]{BenAmar:RieffelKontsevich}.

%%%%%%%%%%%%%%%%%%%%%%

\appendix
\section{An arithmetic lemma about MZVs}
\label{sec:arithmetic-appendix}

In this appendix, we provide a proof of the following arithmetic fact:
\begin{lemma}\label{lem:(n+1)!}
For $n \ge 1$, we have $\tfrac{1}{n!} \in \MZVipi[n-1]$.
\end{lemma}

\begin{proof}
For $n =1,2$, this is immediate from the definition, so assume $n \ge 3$.

If $n=2k+1$ is odd, we can invoke Hoffman's identity \eqref{eq:Hoffman-formula} to conclude that
\begin{equation}
\frac{1}{n!} 
= \frac{1}{(2k+1)!}
= (-4)^k \frac{\mzv{2,\ldots,2}}{(2\ipi)^{2k}}
\in 2^{2k}\MZVipi[2k] 
\subset \MZVipi[n-1]. \label{eq:oddfactorial}
\end{equation}

Similarly, if $n = 4k+2$ is congruent to two modulo four, the identity
\begin{equation}
	\zeta(\underbrace{1,3,1,3,\ldots,1,3}_{\text{$k$ repetitions of $1,3$}}) = \frac{2\cdot \pi^{4k}}{(4k+2)!}
	\label{eq:Broadhurst-Zagier}%
\end{equation}
from \cite[Corollary~2]{BorweinBradleyBroadhurstLisonek:SpecialValues}, valid for all $k\ge 1$, allows us to write
\begin{equation}
\frac{1}{(4k+2)!} 
= 2^{4k}  \cdot  \frac{1}{2} \cdot \frac{\zeta(1,3,\ldots,1,3)}{(2\ipi)^{4k}} 
\in 2^{4k} \cdot \MZVipi[1]  \cdot \MZVipi[4k] 
\subset 2^{4k} \MZVipi[n-1]. \label{eq:2mod4}
\end{equation}

We are unaware of a similarly elegant formula when $n\ge 4$ is divisible by four; say $n = 4k+4$.  But from the cases already considered \eqref{eq:oddfactorial} and \eqref{eq:2mod4}, we know that
\begin{align}
\frac{1}{(4k+5)!} \cdot { 4k+5 \choose 4j+2} = \frac{1}{(4(k-j)+3)!}\frac{1}{(4j+2)!} \in 2^{4k+2}\MZVipi[n-1] \label{eq:binom}
\end{align}
for all $0 \le j \le k$.  In light of the lemma below and B\'ezout's identity, we can write the number $(4k+5)2^{4k+2}$ as a $\ZZ$-linear combination of the binomial coefficients appearing above.  Hence by taking linear combinations of the expressions \eqref{eq:binom} and dividing by $2^{4k+2}$, we obtain $\frac{1}{n!} = \frac{1}{(4k+5)!} \cdot (4k+5) \in \MZVipi[n-1]$, as desired.
\end{proof}

\begin{lemma}\label{lem:binomial-gcds}
For $k \ge 1$, the number 
\[
d \defas \gcd \set{
			\binom{4k+5}{4j+2}
		}{
			0 \leq j \leq k
		},
\]
divides  $(4k+5)\cdot 2^{4k+1}$.
\end{lemma}

\begin{proof}
First note from $j=0$ that $d=uv$ for a divisor $u$ of $4k+5$ and a divisor $v$ of $2k+2$. Since $4k+5$ and $2k+2$ are relatively prime, so are $u$ and $v$. The largest power of two that divides $v$ is constrained by $2^r \leq 2k+2 \leq 2^{4k+1}$, and below we will show that $v$ has no odd prime divisors. It follows that $v$ is a divisor of $2^{4k+1}$, which shows the claim.
	
	Pick any odd prime $p$ that divides $k+1$, then $a_0=1$ in the $p$-adic expansion
	\begin{align*}
		4k+5 = \sum_{r\geq 0} a_r p^r
		\qquad\text{where}\quad
		0 \leq a_r < p.
	\end{align*}
	Suppose that we can find a family of digits $b_r$ such that
	\begin{align}
		0 \leq b_r \leq a_r \qquad\text{for all $r$ and}\qquad
		\sum_{r\geq 0} (a_r-b_r) p^r \equiv -1 \mod 4.
		\label{eq:p-adic}%
	\end{align}
	Then $\sum_{r\geq 0} b_r p^r = 4j+2$ for some $0\leq j \leq k$ and Lucas's theorem
	\begin{equation*}
		\binom{4k+5}{4j+2}
		\equiv \prod_{r \geq 0} \binom{a_r}{b_r}
		\not\equiv 0
		\mod p
	\end{equation*}
	shows that $p$ does not divide $d$. To finish the proof, we argue that \eqref{eq:p-adic} always admits a solution.
	If there is any $r$ such that $a_r>0$ and $p^r \equiv -1 \mod 4$, we have a solution by setting $b_r = a_r-1$ and $b_s=a_s$ for all other $s\neq r$.
	Now assume that $p^r \equiv 1 \mod 4$ whenever $a_r>0$, hence $4k+5 \equiv 1 \equiv a \mod 4$ with $a = \sum_{r\geq 0} a_r$. Due to $a_0=1$, the case $a=1$ would imply $4k+5=1$, incompatible with $k\geq 0$. So we must have $a \geq 5$, and we can find $0 \leq b_r \leq a_r$ such that $\sum_{r\geq 0} b_r = a-3$, solving \eqref{eq:p-adic}.
\end{proof}

%%%%%%%%%%%%%%%%%%%%%%%%%%%%%%%%%%%%%%%%%%%%%%%%%%%%%%%%%%%%

\section{Regularized limits of $\ZZ$-linear polylogarithms}
\label{sec:appendix-mpl}

In this section, we sketch a proof of statement \eqref{prop:moduli-log-sing} concerning the regularized restriction of polylogarithms to boundary strata in the moduli space.

By induction on the codimension, we can reduce the statement to the case in which the codimension $k$ of the stratum is equal to one.  Moreover, by taking derivatives of the expansion \eqref{eq:td-hlog} and using an induction on the weight, we further reduce the problem to dealing with the leading term in the expansion, i.e.~it is enough to prove that as $t\to 0$, the regularized limit of a polylogarithm is a polylogarithm on the corresponding boundary hypersurface.

The codimension-one strata are in bijection with proper subsets $A \subset S$ with $|A|\ge 2$.  Such a stratum is isomorphic to products of the form $\M{A} \times \M{B}$, where $B = S\setminus A \cup \sset{*}$ is the complement, together with an extra symbol $*$.

Choosing pairs $a_0 \ne a_1 \in A$, and $b_0=* \ne b_1 \in B$, we obtain trivializations of the universal curves $\uX{A}$ and $\uX{B}$, realizing them concretely as a family of marked curves whose underling curves are $\PP[1]$. The location of the marked points in $A$ and $B$ are indicated by maps $a \colon \U \to \PP[1]$ and $b \colon \U \to \PP[1]$ for $a \in A$ and $b \in B$. Note that in our trivialization, $a_0=b_0=0\in \PP[1]$ and $a_1=b_1=1 \in \PP[1]$ are constant.  For $s \in S$, define a map $s \colon \U \times \Do \to \PP[1]$ by the formula
\[
s(u,t) = \begin{cases}
	ta(u) & \text{if $s = a \in A$ and} \\
	b(u)  & \text{if $s = b \in B$.}
\end{cases}
\]
These maps define a family of  $S$-marked curves for which the points in $A$ collide with constant velocity as $t \to 0$, giving a map $\U \times \Do \to \M{S}$ that extends to an embedding of a tubular neighbourhood $\U \times \D \to \bM{S}$ for which the cross ratio $t = \crossrat{a_0}{b_1}{a_1}{\infty}$ is a transverse coordinate.

We now consider the behaviour of polylogarithms on the contractible set $\U \times \W$.  Consider first the polylogarithms of weight one.  Any such function is a linear combination of ordinary logarithms $\tfrac{1}{2\ipi}\log(s-s')$ where $s,s' \in S$ and the analysis breaks into three different cases: both $s,s' \in A$, both $s,s' \in B$ or $s \in A$ and $s' \in B$.  In the first case we have the expansion
\[
\frac{\log(s-s')}{2\ipi} = \frac{\log(ta(u)-ta'(u))}{2\ipi} = \frac{\log(a(u)-a'(u))}{2\ipi} + \frac{\log t}{2\ipi}
\]
showing that the regularized limit as $t \to 0$ is $\tfrac{1}{2\ipi}\log(a-a') \in \Vr[1]{A}$.  The other two cases are analyzed similarly; they are in fact holomorphic as $t\to 0$, giving limits in $\Vr[1]{B}$.

Hence using shuffle regularization, we can reduce the problem to understanding the behaviour of convergent iterated integrals built from the forms $\dlog{b} = \frac{\td z}{2\ipi(z - b(u))}$ for $b\in B$ and $\dlog{a} = \frac{\td z}{2\ipi(z-ta(u))}$ for $a \in A$, in the limit $t \to 0$.   By path concatenation, we can further reduce the problem to considering the following three types of paths, which generate the trivial local system of groupoids $\FGpd{\uX{S}}|_{\U\times \W}$:
\begin{enumerate}
\item Paths between marked points in $\pInf{S} \setminus A $ that do not collide with the marked points in $A$ as $t \to 0$.  Along such paths the forms $\dlog{a}$ are nonsingular and converge uniformly to the form $\frac{\td z}{2\ipi z} = \dlog{*}$ associated to the extra marked point $* \in B$.  Hence the limit lies in $\Vr{B}$ by the dominated convergence theorem.
\item Paths between the marked points in $A$ which contract to a point as $t \to 0$.  After making the conformal transformation $z \mapsto t/z$, this problem is equivalent to the previous one but with the roles of $A$ and $B$ reversed.

\item A path from $a_0 = * = 0$ to $\infty$.  Here we encounter a difficulty because the limiting forms $\lim_{t \to 0}\dlog{a} = \frac{\td z}{2\ipi z}$  have a singularity at the startpoint of integration.  But by shuffle regularization, one can reduce such integrals to convergent integrals built from elements in $\T{A}$ and $\T{S}\cdot (S\setminus A)$, which can again by analyzed using the dominated convergence theorem, giving a mixture of polylogarithms on $\M{A}$ and $\M{B}$.
\end{enumerate}
For more details on the computation of regularized limits of periods, see \cite[Section~3.3.3]{Panzer:PhD}.

\bibliographystyle{hyperamsplain}
\bibliography{dq-weights}

\end{document}